\documentclass[10pt]{article}

\usepackage{authblk}                                                          
\usepackage{geometry}                                                       
\usepackage[utf8]{inputenc}                                                
\usepackage{hyperref}                                                        
\usepackage{url}                                                                
\usepackage{booktabs}                                                       
\usepackage{amsfonts}                                                       
\usepackage{nicefrac}                                                         
\usepackage{microtype}                                                      
\usepackage{wrapfig}
\usepackage{amsmath,amsthm,amssymb,mathtools,graphicx,enumitem,hyphenat,float}
\usepackage{bbm}
\usepackage{subcaption}
\usepackage{accents, hyperref}
\usepackage{import}
\usepackage{algorithm}
\usepackage{algpseudocode}
\usepackage[sort]{natbib}
\setcitestyle{numbers,open={[},close={]}}
\title{Formatting Instructions For NeurIPS 2020}
\graphicspath{}
\setcitestyle{citesep={,}}
\usepackage{appendix}

\newcommand{\printfnsymbol}[1]{%
  \textsuperscript{\@fnsymbol{#1}}%
}

\newcommand\ec[2][]{\ensuremath{\mathbb{E}_{#1} \left[#2\right]}}
\newcommand\ecb[2][]{\ensuremath{\mathbb{E}_{B_{#1}} \left[#2\right]}}
\newcommand\ecn[2][]{\ec[#1]{\norm{#2}^2}}
\newcommand\ecd[2][]{\ensuremath{\mathbb{E}_{\mathcal{D}} \left[#2\right]}}

\newcommand\ev[1]{\left \langle #1 \right \rangle}
\newcommand\dotprod[1]{\left \langle #1 \right \rangle}
\newcommand\br[1]{\left ( #1 \right )}
\newcommand\pbr[1]{\left \{ #1 \right \} }
\newcommand\floor[1]{\left \lfloor #1 \right \rfloor}

\newcommand{\1}{\mathbbm{1}}
\newcommand{\N}{\mathbb{N}}

\newcommand{\R}{\mathbb{R}}

\DeclareMathOperator*{\argmin}{\arg\!\min}

\newcommand{\diag}{\mathrm{diag }}

\newcommand{\Proba}{\mathbb{P}}

\newcommand{\norm}[1]{\left\lVert#1\right\rVert}
\newcommand{\sqn}[1]{{\left\lVert#1\right\rVert}^2}
\newcommand{\trn}[1]{{\left\lVert#1\right\rVert}^2_{\mathrm{Tr}}}

\newcommand{\tr}{\mathrm{Tr}}

\newcommand{\prox}{\mathrm{prox}}
\renewcommand{\L}{\mathcal{L}}
\newcommand{\D}{\mathcal{D}}

\newcommand{\eqdef}{\overset{\text{def}}{=}}
\makeatletter
\newsavebox\myboxA
\newsavebox\myboxB
\newlength\mylenA

\newcommand{\cL}{\mathcal{L}}

\usepackage{tcolorbox}
\usepackage{pifont}
\definecolor{mydarkgreen}{RGB}{39,130,67}

\definecolor{mydarkred}{RGB}{192,47,25}

\newcommand*\overbar[2][0.75]{%
    \sbox{\myboxA}{$\m@th#2$}%
    \setbox\myboxB\null
    \ht\myboxB=\ht\myboxA%
    \dp\myboxB=\dp\myboxA%
    \wd\myboxB=#1\wd\myboxA
    \sbox\myboxB{$\m@th\overline{\copy\myboxB}$}
    \setlength\mylenA{\the\wd\myboxA}
    \addtolength\mylenA{-\the\wd\myboxB}%
    \ifdim\wd\myboxB<\wd\myboxA%
       \rlap{\hskip 0.5\mylenA\usebox\myboxB}{\usebox\myboxA}%
    \else
        \hskip -0.5\mylenA\rlap{\usebox\myboxA}{\hskip 0.5\mylenA\usebox\myboxB}%
    \fi}
\makeatother

\usepackage[colorinlistoftodos,bordercolor=orange,backgroundcolor=orange!20,linecolor=orange,textsize=scriptsize]{todonotes}

\newcommand{\rob}[1]{}

\usepackage{mdframed} 
\usepackage{thmtools}

\definecolor{shadecolor}{gray}{0.90}
\declaretheoremstyle[
headfont=\normalfont\bfseries,
notefont=\mdseries, notebraces={(}{)},
bodyfont=\normalfont,
postheadspace=0.5em,
spaceabove=1pt,
mdframed={
  skipabove=8pt,
  skipbelow=8pt,
  hidealllines=true,
  backgroundcolor={shadecolor},
  innerleftmargin=4pt,
  innerrightmargin=4pt}
]{shaded}

\declaretheorem[style=shaded,within=section]{definition}
\declaretheorem[style=shaded,sibling=definition]{theorem}
\declaretheorem[style=shaded,sibling=definition]{proposition}
\declaretheorem[style=shaded,sibling=definition]{assumption}
\declaretheorem[style=shaded,sibling=definition]{corollary}

\declaretheorem[style=shaded,sibling=definition]{lemma}
\declaretheorem[style=shaded,sibling=definition]{remark}

\begin{document}
\title{\bf Asynchronous Iterations in Optimization: New Sequence Results and Sharper Algorithmic Guarantees}
\author[1]{Hamid Reza\ Feyzmahdavian}
\author[2]{Mikael\ Johansson}
\affil[1]{ABB Corporate Research, V\"aster\aa s, Sweden} 
\affil[2]{KTH Royal Institute of Technology, Stockholm, Sweden}
\maketitle

%
%
\begin{abstract}
We introduce novel convergence results for asynchronous iterations that appear in the analysis of parallel and distributed optimization algorithms. The results are simple to apply and give explicit estimates for how the degree of asynchrony impacts the convergence rates of the iterates. Our results shorten, streamline and strengthen existing convergence proofs for several asynchronous optimization methods and allow us to establish convergence guarantees for popular algorithms that were thus far lacking a complete theoretical understanding. Specifically, we use our results to derive better iteration complexity bounds for proximal incremental aggregated gradient methods, to obtain tighter guarantees depending on the average rather than maximum delay for the asynchronous stochastic gradient descent method, to provide less conservative analyses of the speedup conditions for asynchronous block-coordinate implementations of Krasnosel'ski\u{i}–Mann iterations, and to quantify the convergence rates for totally asynchronous iterations under various assumptions on communication delays and update rates.
\end{abstract}

%
%
\section{Introduction}
With the ubiquitous digitalization of the society, decision problems are rapidly expanding in size and scope. Increasingly often, we face problems where data, computations, and decisions need to be distributed on multiple nodes. These nodes may be individual cores in a CPU, different processors in a multi-CPU platform, or servers in a geographically dispersed cluster. Representative examples include machine learning on data sets that are too large to conveniently store in a single computer, real-time decision-making based on high-velocity data streams, and control and coordination of infrastructure-scale systems.

Insisting that such multi-node systems operate synchronously limits their scalability, since the performance is then dictated by the slowest node, and the system becomes fragile to node failures. Hence, there is a strong current interest in developing asynchronous algorithms for optimal decision-making (see, e.g.,~\cite{RRW+:11,Liu:14,Peng:16,LPL:17,Mishchenko:2020} and references therein). Well-established models for parallel computations, such as bulk synchronous parallel~\citep{Val:90} or MapReduce~\citep{DeG:08}, are now being complemented by stale-synchronous parallel models~\citep{HCC+:13} and fully asynchronous processing paradigms~\citep{Han:15}. In many of these frameworks, the amount of asynchrony is a design parameter: in some systems, the delay is proportional to the number of parallel workers deployed~\citep{RRW+:11}; while other systems use communication primitives which enforce a hard limit on the maximum information delay~\citep{HCC+:13}. It is therefore useful to have theoretical results which characterize the level of asynchrony that can be tolerated by a given algorithm. To this end, this paper develops several theoretical tools for studying the convergence of asynchronous iterations.

The dynamics of asynchronous iterations are much richer than their synchronous counterparts and quantifying the impact of asynchrony on the convergence rate is mathematically challenging. Some of the first results on the convergence of asynchronous iterations were derived by~\cite{ChM:69} for solving linear equations. This work was later extended to nonlinear iterations involving maximum norm contractions~\citep{Bau:78} and monotone mappings~\citep{BeE:87}. Powerful convergence results for broad classes of asynchronous iterations under different assumptions on communication delays and update rates were presented by~\cite{Ber:83},~\cite{Tsitsiklis:86}, and in the celebrated book of~\cite{BeT:89}. Although the framework for modeling asynchronous iterations in~\cite{BeT:89} is both powerful and elegant, the most concrete results only guarantee asymptotic convergence and do not give explicit bounds on convergence times. Execution time guarantees are essential when iterative algorithms are used to find a decision under stringent real-time constraints. In this paper, we derive a number of convergence results for asynchronous iterations which explicitly quantify the impact of asynchrony on the convergence times of the iterates.

The convergence guarantees for influential asynchronous optimization algorithms such as~\textsc{Hogwild!}~\citep{RRW+:11}, Delayed \textsc{Sgd}~\citep{AgD:12}, \textsc{AsySCD}~\citep{Liu:14}, ARock~\citep{Peng:16} and \textsc{Asaga}~\citep{LPL:17} have been established on a per-algorithm basis, and are often based on intricate induction proofs. Such proofs tend to be long and sources of conservatism are hard to isolate. A closer analysis of these proofs reveals that they rely on a few common principles. In this paper, we attempt to unify these ideas, derive general convergence results for the associated sequences, and use these results to systematically provide stronger guarantees for several popular asynchronous algorithms. In contrast to the recent analysis framework proposed in~\cite{Mania:17}, which models the effect of asynchrony as noise, our results attempt to capture the inherent structure in the asynchronous iterations. This allows us to derive convergence results for complex optimization algorithms in a systematic and transparent manner, without introducing unnecessary conservatism. We make the following specific contributions:
\begin{itemize}
\item We identify two important families of sequences, characterized by certain inequalities, that appear naturally when analyzing the convergence of asynchronous optimization algorithms. For each family, we derive convergence results that allow to quantify how the degree of asynchrony affects the convergence rate guarantees. We use these sequence results to analyze several popular asynchronous optimization algorithms. 
\item First, we derive stronger convergence guarantees for the proximal incremental gradient method and provide a larger range of admissible step-sizes. Specifically, for $L$-smooth and convex objective functions, we prove an iteration complexity of $\mathcal{O}\bigl( L \tau/\epsilon\bigr)$, which improves upon the previously known rate $\mathcal{O}\bigl( L^2 \tau^3/\epsilon\bigr)$ given in~\cite{Sun:19}. We also show that for objective functions that satisfy a quadratic functional growth condition, the iteration complexity is $\mathcal{O}\bigl(Q \tau \log(1/\epsilon)\bigr)$, where $Q=L/\mu$  is the condition number. In this case, our result allows the algorithm to use larger step-sizes than those provided in~\cite{Vanli:18}, leading to a tighter  convergence rate guarantee.
\item Second, we analyze the asynchronous stochastic gradient descent method with delay-dependent step-sizes and extend the results of~\cite{Koloskova:2022} from non-convex to convex and strongly convex problems. We show that our sequence results are not limited to providing step-size rules and convergence rates that depend on the maximal delay. In particular, for convex problems, we derive an iteration complexity of
\begin{align*}
\mathcal{O}\left(\frac{L \tau_{\textup{ave}}}{\epsilon} + \frac{ \sigma^2}{\epsilon^2}\right),  
\end{align*}
where $\tau_{\textup{ave}}$ is the average delay, and $\sigma$ denotes the variance of stochastic gradients. For strongly convex problems, we obtain the iteration complexity 
\begin{align*}
\mathcal{O}\left(\left({Q \tau_{\textup{ave}}} + \frac{\sigma^2}{\epsilon}\right)\log\left(\frac{1}{\epsilon}\right)\right).
\end{align*}
Our guarantees improve the previously best known bounds given in~\cite{Arjevani:2020} and~\cite{Karimireddy:2020}, which are based on the maximal delay that can be significantly larger than the average delay. Similar to~\cite{Mishchenko:2022}, we also provide convergence guarantees and admissible step-sizes which depend only 
on the number of parallel workers deployed, rather than on the gradient delays.
\item Third, we give an improved analysis of the ARock framework for asynchronous block-coordinate updates of Krasnosel'ski\u{i}–Mann iterations. For pseudo-contractive operators, we show that ARock achieves near-linear speedup as long as the number of parallel computing elements is $o(m)$, where $m$ is the number of decision variables. Compared to the results presented in \cite{Hannah:17}, we improve the requirement for the linear speedup property from $\mathcal{O}\left(\sqrt{m}\right)$ to $o\left(m\right)$.
\item Finally, we present a uniform treatment of asynchronous iterations involving block-maximum norm contractions under partial and total asynchronism. Contrary to the results in~\cite{Bertsekas:15} which only established asymptotic convergence, we give explicit estimates of the convergence rate for various classes of bounded and unbounded communication delays and update intervals.
\end{itemize}

This paper generalizes and streamlines our earlier work~\citep{FeJ:14, FAJ:14, AFJ:16}. Specifically, we extend the sequence result in~\citep{FAJ:14} to a family of unbounded delays, which allows to deal with totally
asynchronous iterations. Compared to~\citep{FAJ:14}, we present two new Lemmas~\ref{Sec:Main Results Lemma2} and~\ref{Sec:Main Results Lemma3} in Section~\ref{Sec:Main Results}. The analysis in \cite{FeJ:14} is limited to contraction mappings in the maximum-norm, and does not provide any sequence results for analyzing asynchronous algorithms. However, Lemma~\ref{Sec:Main Results Lemma3} in Section~\ref{Sec:Main Results} recovers the results in~\cite{FeJ:14} as a special case. The sequence result presented in \citep{AFJ:16} is only applicable for deriving linear convergence rates, and is restricted to deliver step-size rules and convergence rates that depend on the maximum delay. To overcome these limitations, we introduce a novel sequence result in Lemma~\ref{Sec:Main Results Lemma4} that guarantees both linear and sub-linear rates of convergence and can provide convergence bounds that depend on the average delay. In~\citep{AFJ:16}, the analysis of the proximal incremental gradient method has a drawback that the guaranteed bound grows quadratically with the maximum delay $\tau$. In our work, we improve the dependence on $\tau$ from quadratic to linear.

\subsection{Notation and Preliminaries}
Here, we introduce the notation and review the key definitions that will be used throughout the paper. We let $\mathbb{R}$, $\mathbb{N}$, and $\mathbb{N}_0$ denote the set of real numbers, natural numbers, and the set of natural numbers including zero, respectively. For any $n\in\mathbb{N}$, 
\begin{align*}
[n]:=\{1,\ldots,n\}.
\end{align*}
For a real number $a$, we denote the largest integer less than or equal to $a$ by $\lfloor a \rfloor$ and define
\begin{align*}
(a)_{+} := \max\{a,\; 0\}. 
\end{align*}
We use $\|\cdot\|$ to represent the standard Euclidean norm on $\mathbb{R}^d$ and  $\langle x,\;y\rangle$ to denote the Euclidean (dot) inner product of two vectors $x,y\in\mathbb{R}^d$. We say that a function $f:\mathbb{R}^d \rightarrow \mathbb{R}$ is $L$-smooth if it is differentiable and 
\begin{align*}
\| \nabla f(y) - \nabla f(x) \| \leq L \| y - x\|,\quad \forall x,y\in\mathbb{R}^d.
\end{align*}
We say a convex function $f:\mathbb{R}^d \rightarrow \mathbb{R}$ is $\mu$-strongly convex if 
\begin{align*}
f(y) \geq f(x) + \langle \nabla f(x), y - x \rangle + \frac{\mu}{2} \| y -x \|^2,\quad \forall x,y\in\mathbb{R}^d.
\end{align*}
The notation $g(t) = \mathcal{O}\bigl(h(t)\bigr)$ means that there exist positive constants $M$ and $t_0$ such that $g(t) \leq M h(t)$ for all $t \geq t_0$, while $g(t) = o\bigl(h(t)\bigr)$ means that $\lim_{t \rightarrow \infty} g(t)/h(t) = 0$. Following standard convention, we use tilde $\tilde{\mathcal{O}}$-notation to hide poly-logarithmic factors in the problem parameters.

%
%
\section{Lyapunov Analysis for Optimization Algorithms}
\label{Sec:Motivation}

Convergence proofs for optimization algorithms are usually based on induction and often presented without invoking any general theorems. For more complex algorithms, this leads to lengthy derivations where it is difficult to distinguish mathematical innovations. The need to systemize convergence proofs for optimization algorithms was recognized in Polyak’s insightful textbook~\citep{Polyak:87}. Polyak argued that most results concerning convergence and rate of convergence of optimization algorithms can be derived using Lyapunonv’s second method, with typical Lyapunov functions being the objective function value, the norm of its gradient, or the squared distance between the current iterate and the optimal set. In addition, he derived and collected a number of useful sequence results which allowed to shorten, unify and clarify many convergence proofs~\cite[Chapter 2]{Polyak:87}.  

To make these ideas more concrete, consider the simple gradient descent method
\begin{align*}
x_{k+1} &= x_k -\gamma \nabla f(x_k), \quad k\in\mathbb{N}_0.
\end{align*}
Assume that $f:\mathbb{R}^d \rightarrow \mathbb{R}$ is $\mu$-strongly convex and $L$-smooth. If $x^\star$ is the minimizer of $f$ on~$\mathbb{R}^d$, the standard convergence proof (e.g., Theorem $2.1.15$ in \cite{Nesterov:13}) establishes that the iterates satisfy
\begin{align*}
	\Vert x_{k+1}-x^{\star}\Vert^2 &\leq \left(1 - \frac{2\gamma\mu L}{\mu+L} \right) \Vert x_k - x^{\star}\Vert^2 -  \gamma\left(\frac{2}{\mu+L} - \gamma\right) \Vert \nabla f(x_k)\Vert^2.
\end{align*}
In terms of
$V_k =  \Vert x_k - x^{\star}\Vert^2$ and $W_k =  \Vert  \nabla f(x_k)\Vert^2$, the inequality reads
\begin{align*}
	V_{k+1} &\leq \left(1 - \frac{2\gamma\mu L}{\mu+L} \right) V_k -  \gamma\left(\frac{2}{\mu+L} - \gamma\right) W_k.
\end{align*}
For step-sizes $\gamma\in \bigl(0, 2/{(\mu + L)}\bigr]$, the second term on the right-hand side can be dropped and, hence, $V_k$ is guaranteed to decay by at least a factor 
$q=1-2\gamma \mu L/(\mu + L)$. That is,
\begin{align}
V_{k+1} &\leq  q V_k,\quad k\in\mathbb{N}_0. 
\label{Sec:Main Results Eq1}
\end{align}
This implies linear convergence of the iterates, i.e., 
\begin{align*}
V_k\leq q^k V_0, \quad k\in\mathbb{N}_0.
\end{align*}

When $f$ is convex, but not necessarily strongly convex, and $L$-smooth, the iterates satisfy 
\begin{align*}
2\gamma \bigl(f(x_{k}) - f(x^\star)\bigr) +  \Vert x_{k+1} -x^{\star}\Vert^2 \leq \Vert x_k - x^{\star} \Vert^2 -  \gamma\left(\frac{1}{L} - \gamma\right) \Vert \nabla f(x_k)\Vert^2.
\end{align*}
Let $V_k$ and $W_k$ be as before, while $X_k= 2\gamma (f(x_k) - f(x^\star))$. Then, the inequality above can be rewritten as
\begin{align}
 X_{k} + V_{k+1}  & \leq V_k  -  \gamma\left(\frac{1}{L} - \gamma\right) W_k.
\label{Sec:Main Results Eq2}
\end{align}
If $\gamma \in (0, 1/L]$, the second term on the right-hand side is non-positive and can be dropped. 
Summing both sides of~\eqref{Sec:Main Results Eq2} over $k$ and using telescoping cancellation then gives
\begin{align*}
 \sum_{k=0}^K X_{k} + V_{K+1}\leq V_0, \quad K\in\mathbb{N}_0.
\end{align*}
The well-known $\mathcal{O}(1/k)$ convergence rate of $f(x_k)$ to $f(x^\star)$ follows from the fact that $V_k$ is non-negative and that $X_k$ is non-increasing~\cite[Theorem $10.21$]{Beck:17}. 

In the analysis of gradient descent above, it is natural to view the iteration index $k$ as a surrogate for the accumulated execution time. This interpretation is valid  when the computations are done on a single computer and the time required to perform gradient computations and iterate updates are constant across iterations. In distributed and asynchronous systems, on the other hand, computations are performed in parallel on nodes with different computational capabilities and workloads, without any global synchronization to a common clock. In these systems, the iteration index is simply an ordering of events, typically associated with reading or writing the iterate vector from memory. Since the global memory may have been updated from the time that it was read by a node until the time that the node returns its result, $x_{k+1}$ may not depend on $x_k$ but rather on some earlier iterate $x_{k-\tau_k}$, where $\tau_k\in\{0,\ldots, k-1\}$. More generally, many asynchronous optimization algorithms result in iterations on the form
\begin{align*}
    x_{k+1} &= {\mathcal M}(x_{k}, x_{k-1}, \dots, x_{k-\tau_k}).
\end{align*}
We call these \emph{asynchronous iterations}. The information delay $\tau_k$ represents the age (in terms of event count) of the oldest information that the method uses in the current update, and can be seen as a measure of the amount of asynchrony in the system. When we analyze the convergence of asynchronous iterations, $\Vert x_{k+1}-x^{\star}\Vert^2$ will not only depend on $\Vert x_k-x^{\star}\Vert^2$, but also on $\Vert x_{k-1}-x^{\star}\Vert^2, \dots, \Vert x_{k-\tau_k}-x^{\star}\Vert^2$, and sometimes also on $\Vert \nabla f(x_{k-1})\Vert^2, \dots, \Vert \nabla f(x_{k-\tau_k})\Vert^2$.  Hence, it will not be enough to consider simple sequence relationships such as (\ref{Sec:Main Results Eq1}) or (\ref{Sec:Main Results Eq2}). Because of asynchrony, the right-hand side of the inequalities will involve delayed versions of $V_k$ and $W_k$ that perturb the convergence of the synchronous iteration. 

To fix ideas, consider the gradient descent method with a constant delay of $\tau$ in the gradient computation:
\begin{align}
    x_{k+1} &= x_k - \gamma \nabla f(x_{k-\tau}). \label{eqn:delayed_gd}
\end{align}
To analyze its convergence when $f$ is $\mu$-strongly convex and $L$-smooth, we add and subtract $\gamma \nabla f(x_k)$ to the right-hand side of (\ref{eqn:delayed_gd}) and study the distance of the iterates to the optimum. By applying the same analysis as for the gradient descent method with $\gamma\in (0, 2/(\mu+L)]$, we obtain
\begin{align}
    \Vert x_{k+1}-x^{\star}\Vert^2 &\leq \left(1 - \frac{2\gamma\mu L}{\mu+L} \right) \Vert x_k - x^{\star}\Vert^2
    + \omega_k,
    \label{eqn:delayed_gd_analysis}
\end{align}
where the perturbation term $\omega_k$ accounts for the impact of the delay in gradient computation and is given by
\begin{align}
    \omega_k &= 2\gamma \bigl\langle \nabla f(x_{k})-\nabla f(x_{k-\tau}),\; x_k-\gamma\nabla f(x_k)-x^{\star} \bigr\rangle + \gamma^2 \Vert \nabla f(x_k) - \nabla f(x_{k-\tau}) \Vert^2.
    \label{eqn:delayed_gd_error}
\end{align}
According to Lemma~\ref{Lemma_delayedgradient_1} in Appendix~\ref{Appendix_Movation}, we can bound $\omega_k$ by
\begin{align}
    \omega_k &\leq \left( 
    \gamma^4 L^4 \tau^2 + 2\gamma^2L^2 \tau
    \right) \max_{(k-2\tau)_+\leq \ell \leq k}\left\{ \Vert x_\ell-x^{\star}\Vert^2 \right\}.
    \label{eqn:delayed_gd_errorbound}
\end{align}
Substituting~\eqref{eqn:delayed_gd_errorbound} into~\eqref{eqn:delayed_gd_analysis}, we conclude that
\begin{align*}
    \Vert x_{k+1}-x^{\star}\Vert^2 &\leq \left(1 - \frac{2\gamma\mu L}{\mu+L} \right) \Vert x_k - x^{\star}\Vert^2 + \left( 
    \gamma^4 L^4 \tau^2 + 2\gamma^2L^2 \tau
    \right) \max_{(k-2\tau)_+\leq \ell \leq k}\left\{ \Vert x_\ell-x^{\star}\Vert^2 \right\}.
\end{align*}
In terms of $V_k=\Vert x_k-x^{\star}\Vert^2$, the iterates therefore satisfy 
\begin{align}
    V_{k+1} &\leq \left(1 - \frac{2\gamma\mu L}{\mu+L} \right) V_k + 
    \left( 
    \gamma^4 L^4 \tau^2 + 2\gamma^2L^2 \tau
    \right) \max_{(k-2\tau)_+\leq \ell \leq k} V_{\ell}.
    \label{Sec:Main Results Eq1-1}
\end{align}
We see that the inequality~\eqref{Sec:Main Results Eq1-1} includes delayed versions of $V_k$ on the right-hand side and reduces to~\eqref{Sec:Main Results Eq1} when $\tau=0$. 

When $f$ is $L$-smooth and convex, but not strongly convex, we first note that
\begin{align}
    \Vert x_{k+1}-x^{\star}\Vert^2 &= 
    \Vert x_k-x^{\star}\Vert^2 - 2\gamma \langle \nabla f(x_{k-\tau}), x_k-x^{\star} \rangle + \gamma^2\Vert \nabla f(x_{k-\tau})\Vert^2.
    \label{eqn:delayed_gd_expansion}
\end{align}
By Lemma~\ref{Lemma_delayedgradient_2} in Appendix~\ref{Appendix_Movation}, the inner product is lower bounded by
\begin{align*}
f(x_k)-f^\star + \frac{1}{2L}\Vert \nabla f(x_{k-\tau})\Vert^2 - \frac{\gamma^2 L\tau}{2} \sum_{\ell=(k-\tau)_+}^{k-1} \Vert \nabla f(x_{\ell-\tau)}\Vert^2 \leq \langle \nabla f(x_{k-\tau}), x_k-x^{\star} \rangle.
\end{align*}
Substituting this bound into (\ref{eqn:delayed_gd_expansion}) yields
\begin{align*}
    2\gamma (f(x_k)-f^{\star}) + \Vert x_{k+1}-x^{\star}\Vert^2 \leq
    \Vert x_k-x^{\star}\Vert^2 & +
    \gamma^3 L \tau \sum_{\ell=(k-\tau)_+}^{k-1} \Vert \nabla f(x_{\ell -\tau)}\Vert^2 
     -\gamma\left(\frac{1}{L}-\gamma\right)
    \Vert \nabla f(x_{k-\tau})\Vert^2.
\end{align*}
With $X_k=2\gamma(f(x_{k})-f^{\star})$, $V_k=\Vert x_k-x^{\star}\Vert^2$, and $W_k=\Vert \nabla f(x_{k-\tau})\Vert^2$, the iterates hence satisfy a relationship on the form
\begin{align}
X_{k} + V_{k+1} &\leq V_k + \gamma^3 L \tau \sum_{\ell = (k-\tau)_+}^{k-1} W_{\ell} -\gamma\left(\frac{1}{L}-\gamma\right) W_k.
\label{Sec:Main Results Eq2-2}
\end{align}
Comparing with~\eqref{Sec:Main Results Eq2}, the right-hand side involves delayed versions of $W_k$.

In the next section, we study the convergence of sequences which include~(\ref{Sec:Main Results Eq1-1}) and~(\ref{Sec:Main Results Eq2-2}) as special cases. Our first set of results considers sequence relationships on the form
\begin{align}
V_{k+1} &\leq q V_k +  p \max_{(k-\tau_k)_+\leq \ell \leq k} V_{\ell} ,\quad k\in\mathbb{N}_0.
\label{Sec:Main Results Eq3}
\end{align}
Here, the perturbation caused by asynchrony at iteration $k$ is modeled as a function on the order of $V_{\ell}$ scaled by a factor $p$, where $\ell \in [k-\tau_k, k]$ and $\tau_k$ is the age of the outdated information. Such sequences have appeared, for example, in the analysis of incremental aggregated gradient methods~\citep{Gurbuzbalaban:17}, accelerated incremental aggregated gradient methods with curvature information~\citep{Wai:20}, asynchronous quasi-Newton methods~\citep{Eisen:17}, and asynchronous forward–backward methods for solving monotone inclusion problems~\citep{Stathopoulos:19}.

The second set of our results considers iterate relationships on the form
\begin{align}
X_{k} + V_{k+1} &\leq q_k V_k + p_k \sum_{\ell = (k-\tau_k)_+}^{k} W_{\ell} - r_k W_k + e_k,\quad k\in\mathbb{N}_0.
\label{Sec:Main Results Eq4}
\end{align}
Here, the perturbation due to asynchrony does not introduce delayed $V_k$-terms, but manifests itself through the presence of delayed $W_k$-terms instead. As we will show in this paper, these relationships appear naturally in the analysis of the proximal incremental aggregated gradient method~\citep{AFJ:16}, the asynchronous stochastic gradient descent method~\citep{AgD:12}, and asynchronous Krasnosel'ski\u{i}–Mann method for pseudo-contractive operators~\citep{Peng:16} .

\section{Novel Sequence Results for Asynchronous Iterations}
\label{Sec:Main Results}

In this section, we develop specific convergence results for iterations on the form (\ref{Sec:Main Results Eq3}) and (\ref{Sec:Main Results Eq4}). Our results attempt to balance simplicity, applicability and power, and provide explicit bounds on how the amount of asynchrony affects the guaranteed convergence rates. As we will demonstrate later, the results allow for a simplified and uniform treatment of several asynchronous optimization algorithms.

\subsection{Results for Iterations on the form (\ref{Sec:Main Results Eq3})}

Our first result, introduced in~\cite{FAJ:14}, establishes convergence properties of iterations on the form~(\ref{Sec:Main Results Eq3}) when delays are bounded.
\begin{lemma}
\label{Sec:Main Results Lemma1}
Let $\{V_k\}$ be a non-negative sequence satisfying
\begin{align}
V_{k+1}\leq q V_k + p \max_{(k-\tau_k)_+\leq \ell \leq k} V_{\ell}, \quad k\in \mathbb{N}_0,
\label{Sec:Main Results Lemma1 Eq1}
\end{align}
for non-negative constants $q$ and $p$. Suppose there is a non-negative integer $\tau$ such that
\begin{align*}
0\leq \tau_k \leq \tau, \quad k\in\mathbb{N}_0.
\end{align*}
If $q + p < 1$, then 
\begin{align*}
V_k \leq \rho^k V_0,\quad k\in\mathbb{N}_0,
\end{align*}
where $\rho=(q + p)^{\frac{1}{1+\tau}}$.
\end{lemma}

\begin{proof} 
See Lemma $3$ in~\cite{FAJ:14}.
\end{proof}

Consider the delay-free counterpart of~\eqref{Sec:Main Results Lemma1 Eq1}:
\begin{align*}
V_{k+1} \leq (q + p) V_{k},\quad k\in\mathbb{N}_0.
\end{align*}
Clearly, if $q+p<1$, the sequence $\{V_k\}$ converges linearly at a rate of $\rho = q + p$. Lemma~\ref{Sec:Main Results Lemma1} shows that the convergence rate of~$\{V_k\}$ is still linear in the presence of bounded delays. Lemma~\ref{Sec:Main Results Lemma1} also gives an explicit bound on the impact that an increasing delay has on the convergence rate. As can be expected, the guaranteed convergence rate deteriorates with increasing $\tau$. More precisely, $\rho$ is monotonically increasing in $\tau$, and approaches one as $\tau$ tends to infinity. 

The next result extends Lemma~\ref{Sec:Main Results Lemma1} to a family of unbounded delays, which allows to deal with \textit{totally asynchronous iterations}~\cite[Chapter $6$]{BeT:89}, and shows that the sequence $\{V_k\}$ can still be guaranteed to converge.
\begin{lemma}
\label{Sec:Main Results Lemma2}
Let $\{V_k\}$ be a non-negative sequence such that
\begin{eqnarray}
V_{k+1} \leq q V_k + p  \max_{(k-\tau_k)_+\leq \ell \leq k} V_{\ell}, \quad k\in \mathbb{N}_0,
\label{Sec:Main Results Lemma2 Eq1}
\end{eqnarray}
for some non-negative scalars $q$ and $p$. Suppose that the delay sequence $\{\tau_k\}$ satisfies
\begin{align}
\lim_{k \rightarrow +\infty} k - \tau_k = +\infty.
\label{Sec:Main Results Lemma2 Eq2}
\end{align}
If $q + p < 1$, then $\{V_k\}$ asymptotically converges to zero: 
\begin{align*}
\lim_{k \rightarrow +\infty} V_k = 0.
\end{align*}
\end{lemma}

\begin{proof} 
See Appendix~\ref{Sec:Main Results Lemma2 Proof}.
\end{proof}

Lemma~\ref{Sec:Main Results Lemma2} provides a test for asymptotic convergence of asynchronous iterations with delays satisfying~\eqref{Sec:Main Results Lemma2 Eq2}. Assumption~\eqref{Sec:Main Results Lemma2 Eq2} holds for bounded delays, irrespectively of whether they are constant or time-varying. Moreover, delays satisfying~\eqref{Sec:Main Results Lemma2 Eq2} can be unbounded, as exemplified by $\tau_k=\lfloor 0.2k\rfloor$ and $\tau_k=\lfloor \sqrt{k}\rfloor$. This constraint on delays guarantees that as the iteration count $k$ increases, the delay $\tau_k$ grows at a slower rate than time itself, thereby allowing outdated information about process updates to be eventually purged from the computation. To see this, let us assume that the update step in the gradient descent method is based on gradients computed at stale iterates rather than the current iterate, i.e.,
\begin{align*}
x_{k+1} = x_k - \gamma \nabla f(x_{k-\tau_k}).    
\end{align*}
If $\tau_k$ satisfies~\eqref{Sec:Main Results Lemma2 Eq2}, then given any time $K_1\in\mathbb{N}$, there exists a time $K_2 \in \mathbb{N}$ such that
\begin{align*}
k-\tau_k \geq K_1,\quad \forall k\geq K_2.
\end{align*}
This means that given any time $K_1$, out-of-date information prior to $K_1$ will not be used in updates after a sufficiently long time $K_2$. Therefore,~\eqref{Sec:Main Results Lemma2 Eq2} is satisfied in asynchronous algorithms as long as no processor ceases to update~\citep{BeT:89}.

Although Lemma~\ref{Sec:Main Results Lemma2} establishes convergence guarantees for the sequence $\{V_k\}$ also under unbounded delays, it no longer provides any finite-time guarantee or rate of convergence. The next result demonstrates that such guarantees can be obtained when we restrict how the possibly unbounded delay sequence is allowed to evolve.

\begin{lemma}
\label{Sec:Main Results Lemma3}
Let $\{V_k\}$ be a non-negative sequence satisfying
\begin{eqnarray}
V_{k+1}\leq q V_k + p \max_{(k-\tau_k)_+ \leq \ell \leq k} V_{\ell}, \quad k\in \mathbb{N}_0,
\label{Sec:Main Results Lemma3 Eq1}
\end{eqnarray}
for some non-negative constants $q$ and $p$ such that $q + p < 1$. In addition, assume that there exists a function $\Lambda:\mathbb{R}\rightarrow \mathbb{R}$ such that the following conditions hold:
\begin{itemize}
\item[(i)] $\Lambda(0)=1$.
\item[(ii)] $\Lambda$ is non-increasing.
\item [(iii)] $\lim_{k\rightarrow +\infty} \Lambda(k) = 0$ and 
\begin{align}
(q + p) \Lambda(k - \tau_k) \leq \Lambda(k+1), \quad k\in \mathbb{N}_0.
\label{Sec:Main Results Lemma3 Eq2}
\end{align}
\end{itemize}
Then $V_k \leq \Lambda(k) V_0$ for all $k\in\mathbb{N}_0$.
\end{lemma}

\begin{proof} 
See Appendix~\ref{Sec:Main Results Lemma3 Proof}.
\end{proof}

According to Lemma~\ref{Sec:Main Results Lemma3}, any function $\Lambda$ satisfying conditions $(i)$--$(iii)$ can be used to quantify how fast the sequence $\{V_k\}$ converges to zero. For example, if $\Lambda(t) = \rho^t$ with $\rho\in(0,1)$, then $\{V_k\}$ converges at a linear rate; and if $\Lambda(t) = t^{-\eta}$ with $\eta>0$, then $\{V_k\}$ is upper bounded by a polynomial function of time. Given $q$ and $p$, it is clear from~\eqref{Sec:Main Results Lemma3 Eq2} that the admissible choices for~$\Lambda$ and, hence, the convergence bounds that we are able to guarantee depend on the delay sequence $\{\tau_k\}$. To clarify this statement, we will analyze a  special case of unbounded delays in detail. Assume that $\{\tau_k\}$ can grow unbounded at a linear rate, i.e., 
\begin{align}
\tau_k \leq \alpha k + \beta, \quad k\in\mathbb{N}_0,
\label{Sec:Main Results Eq5}
\end{align}
where $\alpha\in (0,1)$ and $\beta\geq 0$. The associated convergence result reads as follows.

\begin{corollary}
\label{Sec:Main Results Corollary1}
Let $\{V_k\}$ be a non-negative sequence such that
\begin{eqnarray*}
V_{k+1}\leq q V_k + p \max_{(k-\tau_k)_+ \leq \ell \leq k} V_{\ell}, \quad k\in \mathbb{N}_0,
\end{eqnarray*}
for some non-negative scalars $q$ and $p$. Suppose the delay sequence $\{\tau_k\}$ satisfies~\eqref{Sec:Main Results Eq5}. If $q + p < 1$, then  
\begin{align*}
V_k \leq \left(\frac{\alpha k}{1-\alpha + \beta}  + 1\right)^{-\eta} V_0, \quad k\in\mathbb{N}_0,
\end{align*}
where $\eta = \ln(q+p)/\ln(1-\alpha)$.
\end{corollary}

\begin{proof} 
Conditions $(i)$--$(iii)$ of Lemma~\ref{Sec:Main Results Lemma3} are satisfied by the function
\begin{align*}
\Lambda(t) = \left(\frac{\alpha t}{1-\alpha + \beta}  + 1\right)^{-\eta}.
\end{align*}
\end{proof}

Corollary~\ref{Sec:Main Results Corollary1} shows that for unbounded delays satisfying~\eqref{Sec:Main Results Eq5}, the convergence rate of the sequence $\{V_k\}$ is $\mathcal{O}(k^{-\eta})$. Note that $\alpha$, the rate at which the unbounded delays grow large, affects~$\eta$. Specifically, $\eta$ is monotonically decreasing with $\alpha$ and approaches zero as $\alpha$ tends to one. Hence, the guaranteed convergence rate slows down as the growth rate of the delays increases.

\subsection{Results for Iterations on the form (\ref{Sec:Main Results Eq4})}

We will now shift our attention to the convergence result for iterations on the form~(\ref{Sec:Main Results Eq4}). This result adds a lot of flexibility in how we can model and account for different perturbations that appear in the analysis of asynchronous optimization algorithms, and will be central to the developments in Subsections~\ref{Sec:PIAG}, \ref{Sec:SGD}, and~\ref{Sec:ARock}.

\begin{lemma}
\label{Sec:Main Results Lemma4}
Let $\{ V_{k}\}$, $\{W_{k}\}$, and $\{X_{k}\}$ be non-negative sequences satisfying 
\begin{align}
X_{k} + V_{k+1} \leq q_k V_{k} + p_k \sum_{\ell = (k - \tau_k)_+}^{k} W_{\ell} - r_k W_{k} + e_k,\quad k\in\mathbb{N}_0,
\label{Sec:Main Results Lemma4 Eq0}
\end{align}
where $e_k\in\mathbb{R}$, $q_k\in[0,1]$, and $p_k, r_k\geq 0$ for all $k$. Suppose that there is a non-negative integer $\tau$ such that
\begin{align*}
0\leq \tau_k \leq \tau, \quad k\in\mathbb{N}_0.
\end{align*}
For every $K\in\mathbb{N}_0$, the following statements hold:
\begin{enumerate}
\item Assume that $q_k = 1$ for $k\in\mathbb{N}_0$. If
\begin{align}
\sum_{ \ell = 0}^\tau p_{k + \ell} \leq r_k
\label{Sec:Main Results Lemma4 Eq1}
\end{align}
is satisfied for all $k\in\mathbb{N}_0$, then 
\begin{align*}
\sum_{k=0}^{K} X_{k} &\leq V_0 + \sum_{k=0}^{K} e_{k},\\ V_{K+1}&\leq V_0 + \sum_{k=0}^{K} e_k.
\end{align*}
\item Assume that $p_k=p>0$ and $r_k=r>0$ for  $k\in\mathbb{N}_0$. Assume also that there exists a constant $q\in (0, 1)$ such that $q_k\geq q$ for $k\in\mathbb{N}_0$. If 
\begin{align*}
2\tau + 1 &\leq \min\left\{\frac{1}{1-q}, \frac{r}{p}\right\},
\end{align*}
then 
\begin{align*}
V_{K+1} &\leq  Q_{K+1}\left(V_0 + \sum_{k=0}^{K} \frac{e_{k}}{Q_{k+1}}\right),\\
\sum_{k=0}^{K} \frac{X_{k}}{Q_{k+1}} &\leq  V_0 + \sum_{k=0}^{K} \frac{e_{k}}{Q_{k+1}},
\end{align*}
where $Q_k$ is defined as 
\begin{align*}
Q_k = \prod_{\ell=0}^{k-1} q_{\ell}, \quad k \in\mathbb{N},    
\end{align*}
with $Q_0 = 1$.
\end{enumerate}
\end{lemma}

\begin{proof} 
See Appendix~\ref{Sec:Main Results Lemma4 Proof}.
\end{proof}

Consider the non-delayed counterpart of~\eqref{Sec:Main Results Lemma4 Eq0} with $e_k\equiv 0$, $q_k=q\in(0,1)$, $p_k=p>0$ and $r_k=r>0$:
\begin{align*}
X_{k} + V_{k+1} \leq q V_{k} + (p - r) W_{k},\quad k\in\mathbb{N}_0.
\end{align*}
Assume that $q\in(0,1)$ and $p\leq r$, or equivalently,
\begin{align*}
1\leq \frac{1}{1-q}  \quad \textup{and} \quad 1\leq \frac{r}{p}.
\end{align*}
In this case, the sequence $\{V_k\}$ converges linearly to zero at a rate of $q$. In general, the existence of delays may impair performance, induce oscillations and even instability. However, Lemma~\ref{Sec:Main Results Lemma4} shows that for the delayed iteration~\eqref{Sec:Main Results Lemma4 Eq0}, the convergence rate of $\{V_k\}$ is still $q$ if the maximum delay bound $\tau$ satisfies
\begin{align*}
2\tau + 1\leq \frac{1}{1-q}  \quad \textup{and} \quad 2\tau + 1\leq \frac{r}{p}.
\end{align*}
This means that up to certain value of the delay, the iteration~\eqref{Sec:Main Results Lemma4 Eq0} and its delay-free counterpart have the same guaranteed convergence rate. 

%
%
%
%
\section{Applications to Asynchronous Optimization Algorithms}
\label{Sec:Applications of Main Results}
Data-driven optimization problems can grow large both in the number of decision variables and in the number of data points that are used to define the objective and constraints. It may therefore make sense to parallelize the associated optimization algorithms over both data and decision variables, see Figure~\ref{fig:dopt_architectures}. One popular framework for parallelizing
algorithms in the data dimension is the parameter server~\citep{LZY+:11}. Here, a master node (the server) maintains the decision vector, while the data is divided between a number of worker nodes. When a worker node is queried by the server, it computes and returns the gradient of the part of the objective function defined by its own data. The master maintains an estimate of the gradient of the full objective function, and executes a (proximal) gradient update whenever it receives gradient information from one of the workers. As soon as the master completes an update, it queries idle worker nodes with the updated decision vector. If the asynchrony, measured in terms of the maximum number of iterations carried out by the master between two consecutive gradient updates from any worker, is bounded, then convergence can be guaranteed under mild assumptions on the objective function~\citep{Li:13,AFJ:16}.

\begin{figure}[t]
\centerline{\includegraphics[width=\hsize]{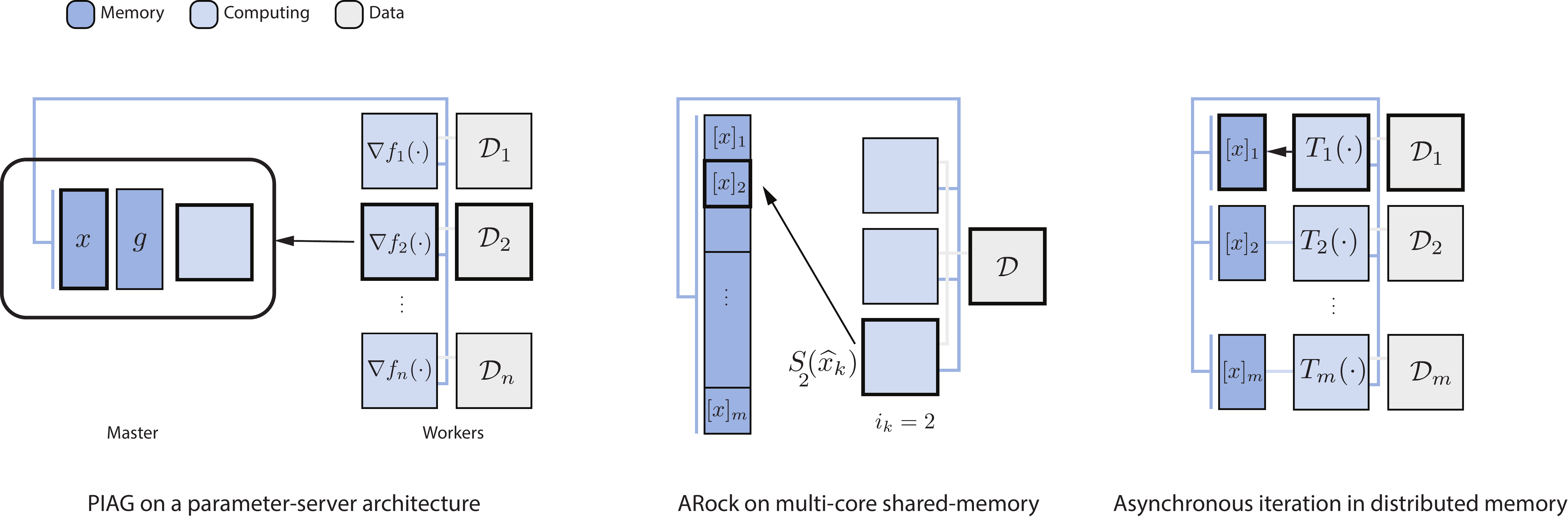}}
\caption{\footnotesize{
Three different parallel architectures studied for various algorithms in Section~\ref{Sec:Applications of Main Results}. The \textsc{Piag} and \textsc{Sgd} algorithms run on the parameter server~(left) distribute data over multiple nodes that are able to evaluate the corresponding loss function gradients, while the master maintains and updates the decision vector. In the ARock framework~(middle), multiple computing units access shared memory and update (randomly selected) sub-vectors of the overall decision vector in parallel. Finally, the totally asynchronous framework~(right) allows to model loosely coupled distributed architectures where computing nodes retrieve parts of the global decision vector from remote nodes, and evaluate components of an operator to update their local decisions.}}
\label{fig:dopt_architectures}
\end{figure}

A natural way to parallelize problems with high-dimensional decision vectors is to use block-coordinate updates. In these methods, the decision vector is divided into sub-vectors, and different processing elements update the sub-vectors in parallel. In the partially and totally asynchronous models of~\cite{BeT:89}, each processing element is responsible for storing and updating one sub-vector, and it does so using delayed information of the remaining decision variables retrieved from the other (remote) processors. Under weak assumptions on the communications delays and update rates of individual processors, convergence can be proven for contraction mappings with respect to the block-maximum norm~\cite[Section $6.3$]{BeT:89}. However, only some special combinations of algorithms and optimization problems result in iterations that are contractive with respect to the block-maximum norm~\cite[Section $3.1$]{BeT:89}. Another type of block-coordinate updates are used in the ARock framework~\citep{Peng:16}. Here, the decision vector is stored in shared memory, and the parallel computing elements pick sub-vectors uniformly at random to update whenever they terminate their previous work. Under an assumption of bounded asynchrony, convergence of ARock can be established for a wide range of objective functions~\citep{Peng:16}. 

In the remaining parts of this paper, we demonstrate how the sequence results introduced in the previous section allows to strengthen the existing convergence guarantees for the algorithms discussed above. Specifically, we improve iteration complexity bounds for the proximal incremental aggregated gradient (\textsc{Piag}) method, which is suitable for implementation in the parameter server framework, with respect to both the amount of asynchrony and the problem conditioning; we derive tighter guarantees for the asynchronous stochastic gradient descent method, which depend on the average delay rather on maximal delay; we prove the linear rate of convergence for ARock under larger step-sizes and provide better scalability properties with respect to the number of parallel computing elements; and we describe a unified Lyapunov-based approach for analysis of totally and partially asynchronous iterations involving maximum norm contractions, that allows to derive convergence rate guarantees also outside the partially asynchronous regime.

%
%
\subsection{Proximal Incremental Aggregated Gradient Method}
\label{Sec:PIAG}
We begin by considering composite optimization problems of the form
\begin{align}
\underset{x\in \mathbb{R}^d}{\textup{minimize}}\quad P(x) :=F(x) + R(x).
\label{Sec:PIAG Eq1}
\end{align}
Here, $x$ is the decision variable, $F$ is the average of many component functions $f_i$, i.e.,
\begin{align*}
F(x)=\frac{1}{n}\sum_{i=1}^n f_i(x),
\end{align*}
and $R$ is a proper closed convex function that may be non-differentiable and extended real-valued. We use $\mathcal{X}^\star$ to denote the set of optimal solutions of~\eqref{Sec:PIAG Eq1} and $P^\star$ to denote the corresponding optimal value. We impose the following assumptions on Problem~\eqref{Sec:PIAG Eq1}.
\begin{assumption}
\label{Sec:PIAG Assumption1}
The optimal set $\mathcal{X}^\star$ is non-empty.  
\end{assumption}
\begin{assumption}
\label{Sec:PIAG Assumption2}
Each function $f_i:\mathbb{R}^d\rightarrow \mathbb{R}$, $i\in [n]$, is convex and $L_i$-smooth.
\end{assumption}
Note that under Assumption~\ref{Sec:PIAG Assumption2}, the average function $F$ is $L_F$-smooth~\citep{Xiao:2014}, where
\begin{align}
L_F \leq L:= \frac{1}{n}\sum_{i=1}^n L_i.
\label{LipschitzConstant}
\end{align}

In the optimization problem~\eqref{Sec:PIAG Eq1}, the role of the regularization term~$R$ is to favor solutions with certain structures. Common choices of $R$ include: the \emph{$\ell_1$ norm}, $R(x) = \lambda \|x\|_1$ with $\lambda >0$, used to promote sparsity in solutions; and the indicator function of a non-empty closed convex set ${\mathcal X}\subseteq \mathbb{R}^d$,
\begin{align*}
R(x) =\begin{cases}
0, & \textup{if} \;\;x\in \mathcal{X},\\
+\infty, & \textup{otherwise}, 
\end{cases}
\end{align*}
used to force the admissible solutions to lie in $\mathcal{X}$. A comprehensive catalog of regularization terms is given in~\cite{Beck:17}.

Optimization problems on the form~\eqref{Sec:PIAG Eq1} are known as \textit{regularized empirical risk minimization} problems and arise often in machine learning, signal processing, and statistical estimation (see, e.g.,~\cite{HTF:09}). In such problems, we are given a collection of $n$ training samples $\left\{(a_1,b_1),\ldots,(a_n,b_n)\right\}$, where each $a_i\in\mathbb{R}^d$ is a feature vector, and each $b_i\in\mathbb{R}$ is the desired response. A classical example is the least-squares regression where the component functions are given by 
\begin{align*}
f_i(x) = \frac{1}{2}(a_i^\top x - b_i)^2, \quad i \in [n],
\end{align*}
and popular choices of the regularization terms include $R(x)=\lambda_1\|x\|_2^2$ (ridge regression), $R(x)=\lambda_2\|x\|_1$ (Lasso), or $R(x)=\lambda_1\|x\|_2^2+\lambda_2\|x\|_1$ (elastic net) for some  
non-negative parameters $\lambda_1$ and $\lambda_2$. Another example is logistic regression for binary classification problems, where each $b_i \in\{-1,1\}$ is the desired class label and the component functions are 
\begin{align*}
f_i(x) = \log \left(1+\exp(-b_i a_i^\top x)\right), \quad i \in [n].
\end{align*}

A standard method for solving Problem~\eqref{Sec:PIAG Eq1} is the \textit{proximal gradient}~(\textsc{Pg}) method, which consists of a gradient step followed by a proximal mapping. More precisely, the \textsc{Pg} method is described by Algorithm~\ref{Sec:PIAG Algorithm1}, where $\gamma$ is a positive step-size, and the prox-operator (proximal mapping) is defined as
\begin{align*}
\textup{prox}_{\gamma R}(x) = \underset{u\in \mathbb{R}^d}{\textup{argmin}} \left\{\frac{1}{2}\|u - x\|^2 + \gamma R(u)\right\}.
\end{align*}
Under Assumptions~\ref{Sec:PIAG Assumption1} and~\ref{Sec:PIAG Assumption2}, the iterates generated by the \textsc{Pg} method with $\gamma=\frac{1}{L}$ satisfy
\begin{align}
P(x_k) - P^\star \leq \frac{L\Vert x_0-x^{\star}\Vert^2}{2k }
\label{Sec:PIAG Eq2}
\end{align}
for all $k\in\mathbb{N}$~\cite[Theorem $10.21$]{Beck:17}. This means that Algorithm~\ref{Sec:PIAG Algorithm1} achieves an $\mathcal{O}(1/k)$ rate of convergence in function values to the optimal value.

\begin{algorithm}[H]
\textbf{Input:} $x_0 \in \mathbb{R}^d$, step-size $\gamma>0$, number of iterations $K\in \mathbb{N}$
  \begin{algorithmic}[1]
    \State Initialize $k\leftarrow 0$
     \While {$k<K$}
      \State Set $g_k  \leftarrow \frac{1}{n}\sum_{i=1}^n \nabla f_i(x_k)$
      \State Set $x_{k+1} \leftarrow \textup{prox}_{\gamma R}(x_k - \gamma g_k)$
     \State Set $k \leftarrow k+1$
    \EndWhile
  \end{algorithmic}
  \caption{Proximal Gradient (\textsc{Pg}) Method}
  \label{Sec:PIAG Algorithm1}
\end{algorithm}

Each iteration of the \textsc{Pg} method requires computing the gradients for all~$n$ component functions. When~$n$ is large, this per iteration cost is  expensive, and hence often results in slow convergence. An effective alternative is the \textit{proximal incremental aggregated gradient}~(\textsc{Piag}) method that exploits the additive structure of~\eqref{Sec:PIAG Eq1} and operates on a \textit{single} component function at a time, rather than on the entire cost function~\citep{Tseng:14}. The \textsc{Piag} method evaluates the gradient of only one component function per iteration, but keeps a memory of the most recent gradients of all component functions to approximate the full gradient~$\nabla F$. Specifically, at iteration $k$, the method will have stored $\nabla f_i(x_{[i]})$ for all $i\in [n]$, where $x_{[i]}$ represents the latest iterate at which~$\nabla f_i$ was evaluated. An integer $j\in [n]$ is then chosen and the full gradient $\nabla F(x_k)$ is approximated by
\begin{align*}
g_k = \frac{1}{n}\left(\nabla f_{j} (x_k) - \nabla f_{j}(x_{[j]}) + \sum_{i=1}^n \nabla f_i\bigl(x_{[i]}\bigr)\right).
\end{align*}
The aggregated gradient vector $g_k$ is employed to update the current iterate $x_k$ via
\begin{align*}
x_{k+1} = \textup{prox}_{\gamma R}(x_k - \gamma g_k).
\end{align*}
Thus, the \textsc{Piag} method uses outdated gradients from previous iterations for the components $f_i$, $i \neq j$, and does need not to compute gradients of these components at iteration $k$. 

A formal description of the \textsc{Piag} method is presented as Algorithm~\ref{Sec:PIAG Algorithm2}. Let us define $s_{i, k}$ as the iteration number in which the gradient of the component function~$f_i$ is updated for the last time before the completion of the $k$th iteration. Then, we can rewrite the aggregated gradient vector $g_k$ as 
\begin{align*}
g_k = \frac{1}{n}\sum_{i=1}^n \nabla f_i\bigl(x_{s_{i, k}}\bigr).
\end{align*}
Note that $s_{i,k}\in\{0,\ldots,k\}$ for each $i\in [n]$ and $k\in\mathbb{N}_0$. In a traditional (serial) implementation of~\textsc{Piag}, the selection of~$j$ on line $7$ follows a deterministic rule (for example, the cyclic rule) and the gradient of the component function $f_j$ is computed at the current iterate, i.e., $s_{j,k}=k$. In the parameter server implementation, the index $j$ will not be chosen but rather assigned to the identity of the worker that returns its gradient to the master at iteration $k$. Since workers can exchange information with the master independently of each other, worker $j$ may evaluate $\nabla f_j$ at an outdated iterate $x_{s_{j, k}}$, where $s_{j, k}<k$, and send the result to the server. We assume that each component function is sampled at least once in the past $\tau$ iterations of the \textsc{Piag} method. In other words, there is a fixed non-negative integer $\tau$ such that the indices $s_{i, k}$ satisfy
\begin{align*}
(k - \tau)_+\leq s_{i, k}\leq k, \quad \; i \in [n],\; k\in\mathbb{N}_0.
\end{align*}
The value $\tau$ can be viewed as an upper bound on the delay encountered by the gradients of the component functions. For example, if the component functions are chosen one by one using a deterministic cyclic order on the index set $\{1,\ldots, n\}$, then $\tau = n-1$. 

\begin{algorithm}[H]
\textbf{Input:} $x_0 \in \mathbb{R}^d$, step-size $\gamma>0$, number of iterations $K\in \mathbb{N}$
  \begin{algorithmic}[1]
    \For {$i=1$ to $n$}
     \State Compute $\nabla f_i(x_0)$
     \State Store $\nabla  f_{i}(x_{[i]}) \leftarrow \nabla f_i(x_0)$
     \EndFor
     \State Initialize $k\leftarrow 0$
     \While {$k<K$}
     \State Choose $j$ from $\{1,\ldots, n\}$
     \State Compute $\nabla f_j(x_k)$
     \State Set $g_k \leftarrow \frac{1}{n}\left(\nabla f_{j} (x_k) - \nabla f_{j}(x_{[j]}) + \sum_{i=1}^n \nabla f_i\bigl(x_{[i]}\bigr)\right)$
     \State Store  $\nabla  f_{j}(x_{[j]}) \leftarrow \nabla f_j(x_k)$ 
      \State Set $x_{k+1} \leftarrow \textup{prox}_{\gamma R}(x_k - \gamma g_k)$
     \State Set $k \leftarrow k+1$
    \EndWhile
  \end{algorithmic}
  \caption{Proximal Incremental Aggregated Gradient (\textsc{Piag}) Method}
  \label{Sec:PIAG Algorithm2}
\end{algorithm}

The following result shows that in the analysis of the \textsc{Piag} method, we can establish iterate relationships on the form~\eqref{Sec:Main Results Eq4}.

\begin{lemma}
\label{Sec:PIAG Lemma1}
Let Assumptions~\ref{Sec:PIAG Assumption1} and~\ref{Sec:PIAG Assumption2} hold. Suppose that $\{\alpha_k\}$ is a sequence defined by 
\begin{align*}
\alpha_k = k + \alpha_0,\quad k\in\mathbb{N}_0,
\end{align*}
where $\alpha_0$ is a non-negative scalar. Let 
\begin{align*}
V_k = 2\gamma \alpha_k\bigl(P(x_{k})-P^\star\bigr)  + \|x_k - x^\star\|^2
\end{align*}
and $W_k = \|x_{k+1}-x_k\|^2$ for $k\in\mathbb{N}_0$. Then, the iterates $\{x_k\}$ generated by Algorithm~\ref{Sec:PIAG Algorithm2} satisfy
\begin{align*}
V_{k+1} &\leq V_k +  \gamma L\bigl(\alpha_k + \tau+1\bigr) \sum_{\ell = (k-\tau)_+}^{k} W_{\ell} - \left(2 \alpha_k +1 -  \gamma L\tau \alpha_k\right) W_k,\quad k\in\mathbb{N}_0.
\end{align*}
\end{lemma}

\begin{proof}
See Appendix~\ref{Sec:PIAG Lemma1 Proof}.
\end{proof}

Using this iterate relationship, the sequence result in Lemma~\ref{Sec:Main Results Lemma4} yields the following convergence guarantee for the \textsc{Piag} method.

\begin{theorem} 
\label{Sec:PIAG Theorem1}
Let Assumptions~\ref{Sec:PIAG Assumption1} and~\ref{Sec:PIAG Assumption2} hold. Suppose that $\gamma \in\left(0, \gamma_{\max}\right]$ with
\begin{align*}
\gamma_{\max} =  \frac{1}{L\bigl(2\tau+1\bigr)}.
\end{align*}
Then, for every $k\in\mathbb{N}$ and any $x^\star \in  \mathcal{X}^\star$, the iterates $\{x_k\}$ generated by Algorithm~\ref{Sec:PIAG Algorithm2} satisfy
\begin{align*}
P(x_k) - P^\star \leq \frac{ \frac{1}{2\gamma}\|x_0-x^{\star}\|^2 + \tau\bigl( P(x_0)-P^\star\bigr)}{k + \tau}.
\end{align*}
\end{theorem}

\begin{proof}
See Appendix~\ref{Sec:PIAG Theorem1 Proof}.
\end{proof}

According to Theorem~\ref{Sec:PIAG Theorem1}, the \textsc{Piag} iterations converge at a rate of $\mathcal{O}({1/k})$ when the constant step-size $\gamma $ is appropriately tuned. The convergence rate depends on the choice of~$\gamma$. For example, if we pick 
\begin{align}
\gamma = \frac{1}{L\bigl(2\tau+1\bigr)},
\label{Sec:PIAG Eq3}
\end{align}
then the corresponding $\{x_k\}$ converges in terms of function values with the rate
\begin{align}
P(x_{k}) - P^\star \leq \frac{L\|x_0-x^{\star}\|^2 + 2\tau \left( L\|x_0-x^{\star}\|^2 + P(x_0)-P^\star\right)}{2(k + \tau)},\quad k\in\mathbb{N}.
\label{Sec:PIAG Eq3-4}
\end{align}
One can verify that the right-hand side is monotonically increasing in $\tau$. Hence, the guaranteed convergence rate slows down as the delays increase in magnitude. In the case that $\tau = 0$, the bound~\eqref{Sec:PIAG Eq3-4} reduces to~\eqref{Sec:PIAG Eq2}, which is achieved by the \textsc{Pg} method. From~\eqref{Sec:PIAG Eq3-4}, we can see that if
\begin{align*}
k \geq  K_{\epsilon}=\frac{L\|x_0-x^{\star}\|^2 + 2\tau \left( L\|x_0-x^{\star}\|^2 + P(x_0)-P^\star\right)}{2\epsilon} - \tau,
\end{align*}
then \textsc{Piag} with the step-size choice~\eqref{Sec:PIAG Eq3} achieves an accuracy of $P(x_k) - P^\star\leq \epsilon$. This shows that the \textsc{Piag} method has an iteration complexity of $\mathcal{O}\bigl(L(\tau+1)/\epsilon\bigr)$ for convex problems. Therefore, as $\tau$ increases, the complexity bound deteriorates linearly with $\tau$. Note that the linear dependence on the maximum delay bound~$\tau$ is unavoidable and cannot further be improved~\citep{Arjevani:2020}.

\begin{remark}
\cite{Sun:19} analyzed convergence of the \textsc{Piag} method under Assumptions~\ref{Sec:PIAG Assumption1} and~\ref{Sec:PIAG Assumption2} and proved~$\mathcal{O}(C/k)$ convergence rate, where $C$ is a positive constant. While the constant $C$ is implicit in~\cite{Sun:19}, \cite{Huang:21} showed that the analysis in~\cite{Sun:19} guarantees an $\mathcal{O}(\tau^3 L^2/\epsilon)$ iteration complexity for the \textsc{Piag} method. In comparison with this result, Theorem~\ref{Sec:PIAG Theorem1} gives a better dependence on the Lipschitz constant ($L$ vs. $L^2$) and on the maximum delay bound ($\tau$ vs. $\tau^3$) in the iteration complexity .    
\end{remark}

Next, we restrict our attention to composite optimization problems under the following \textit{quadratic functional growth} condition.
\begin{assumption}
\label{Sec:PIAG Assumption3}
There exists a constant $\mu>0$ such that
\begin{align}
P(x) - P^\star \geq \frac{\mu}{2}\bigl \|x - \Pi_{\mathcal{X}^\star}(x)\bigr\|^2,\quad \forall x \in \textup{dom}\; R,
\label{Sec:PIAG Eq4}
\end{align}
where $\textup{dom}\; R$, defined as $\textup{dom}\; R :=\{x\in\mathbb{R}^d\;|\; R(x)<+\infty\}$, is the effective domain of the function~$R$, and $\Pi_{\mathcal{X}^\star}(\cdot)$ denotes the Euclidean-norm projection onto the set $\mathcal{X}^\star$, i.e.,
\begin{align*}
\Pi_{\mathcal{X}^\star}(x) = \textup{argmin}_{u\in\mathcal{X}^\star}  \| u - x \|.
\end{align*}
\end{assumption}
Assumption~\ref{Sec:PIAG Assumption3} implies that the objective function grows faster than the squared distance between any feasible point and the optimal set. While every strongly convex function satisfies the quadratic functional growth condition~\eqref{Sec:PIAG Eq4}, the converse is not true in general~\citep{Necoara:19}. For example, if $A\in\mathbb{R}^{m\times d}$ is rank deficient and $b\in\mathbb{R}^m$, the function $F(x) =  \|Ax-b\|^2$ is not strongly convex, but it satisfies Assumption~\ref{Sec:PIAG Assumption3}. Other examples of objective functions which satisfy the quadratic functional growth condition can be found in~\cite{Necoara:19}. Let us define the condition number of the optimization problem~\eqref{Sec:PIAG Eq1} as $Q=L/\mu$. The role of $Q$ in determining the linear convergence rate of (proximal) gradient methods is well-known~\citep{Nesterov:13}.

We next show that under Assumption~\ref{Sec:PIAG Assumption3}, an iterate relationship on the form~\eqref{Sec:Main Results Eq4} appears in the analysis of the \textsc{Piag} method.

\begin{lemma}
\label{Sec:PIAG Lemma2}
Suppose that Assumptions~\ref{Sec:PIAG Assumption1},~\ref{Sec:PIAG Assumption2} and~\ref{Sec:PIAG Assumption3} hold. Let
\begin{align*}
V_k = \frac{2}{L}\bigl(P(x_{k})-P^\star\bigr)  +  \|x_k -\Pi_{\mathcal{X}^\star}(x_k)\|^2,
\end{align*}
and $W_k = \|x_{k+1}-x_k\|^2$ for $k\in\mathbb{N}_0$. Then, the iterates $\{x_k\}$ generated by Algorithm~\ref{Sec:PIAG Algorithm2} satisfy
\begin{align*}
V_{k+1} \leq \left(\frac{1}{1+\gamma \mu \theta}\right) V_k + \frac{1+ \gamma L\bigl(\tau+1\bigr)}{1+\gamma \mu \theta} \sum_{\ell=(k-\tau)_+}^{k} W_{\ell} -  \frac{\frac{2}{\gamma L} + 1 -  \tau}{1+\gamma \mu \theta} W_k,
\end{align*}
where $\theta = \frac{Q}{Q + 1}$.
\end{lemma}

\begin{proof}
See Appendix~\ref{Sec:PIAG Lemma2 Proof}.
\end{proof}

We use Lemma~\ref{Sec:PIAG Lemma2} together with Lemma~\ref{Sec:Main Results Lemma4} to derive the convergence rate of the \textsc{Piag} method for optimization problems whose objective functions satisfy the quadratic functional growth condition.

\begin{theorem} 
\label{Sec:PIAG Theorem2}
Let Assumptions~\ref{Sec:PIAG Assumption1},~\ref{Sec:PIAG Assumption2} and~\ref{Sec:PIAG Assumption3} hold. Suppose that the step-size $\gamma$ is set to
\begin{align*}
\gamma = \frac{h}{L\bigl(2\tau + 1\bigr)},\quad h\in(0,1].
\end{align*}
Then, for every $k\in\mathbb{N}$, the iterates $\{x_k\}$ generated by Algorithm~\ref{Sec:PIAG Algorithm2} satisfy
\begin{align*}
\|x_k - \Pi_{\mathcal{X}^\star}(x_k)\|^2 &\leq \left(1- \frac{1}{1 + (Q+1)(2\tau+1)/h}\right)^{k}\left(\frac{2}{L}\bigl( P(x_{0})-P^\star\bigr)  + \|x_0 - \Pi_{\mathcal{X}^\star}(x_0)\|^2\right),\\
P(x_k)-P^\star &\leq \left(1-  \frac{1}{1 + (Q+1)(2\tau+1)/h}\right)^{k}\left( P(x_{0})-P^\star\bigr)  + \frac{L}{2} \|x_0 - \Pi_{\mathcal{X}^\star}(x_0)\|^2\right).
\end{align*}
\end{theorem}

\begin{proof}
See Appendix~\ref{Sec:PIAG Theorem2 Proof}.
\end{proof}

Theorem~\ref{Sec:PIAG Theorem2} demonstrates that under Assumption~\ref{Sec:PIAG Assumption3}, the \textsc{Piag} method is linearly convergent by taking a constant step-size inversely proportional to the maximum delay $\tau$. The best guaranteed convergence rate is obtained for the step-size
\begin{align*}
\gamma =  \frac{1}{L\bigl(2\tau + 1\bigr)}.
\end{align*}
With this choice of $\gamma$, the iterates converge linearly in terms of function values with rate
\begin{align*}
P(x_k)-P^\star \leq \left(1 - \frac{1}{1 + (Q+1)(2\tau +1)}\right)^{k} \epsilon_0,
\end{align*}
where $\epsilon_0 = P(x_{0})-P^\star  + \frac{L}{2} \|x_0 - \Pi_{\mathcal{X}^\star}(x_0)\|^2$. Taking logarithm of both sides yields
\begin{align*}
\log(P(x_k)-P^\star) \leq k \log \left(1 - \frac{1}{1 + (Q+1)(2\tau +1)}\right) + \log\left(\epsilon_0\right).
\end{align*}
Since $\log(1+x) \leq x$ for any $x > -1$, it follows that
\begin{align*}
\log(P(x_k)-P^\star) \leq -\frac{k}{1 + (Q+1)(2\tau +1)} + \log\left(\epsilon_0\right).
\end{align*}
Therefore, if the number of iterations satisfy 
\begin{align*}
k\geq K_{\epsilon}= \bigr(1 + (Q+1)(2\tau +1)\bigl)\log\left(\frac{\epsilon_0}{\epsilon}\right),
\end{align*}
then $P(x_k)-P^\star\leq \epsilon$. We conclude that the \textsc{Piag} method achieves an iteration complexity of $\mathcal{O}\bigl(Q(\tau+1) \log(1/\epsilon)\bigr)$ for optimization problems satisfying the quadratic functional growth condition. Note that when $\tau=0$, this bound becomes $\mathcal{O}\bigl(Q\log(1/\epsilon)\bigr)$, which is the iteration complexity for the \textsc{Pg} method~\cite[Theorem $10.30$]{Beck:17}.

As discussed before, if the component functions are selected in a fixed cyclic order, then $\tau = n-1$. It follows from Theorem~\ref{Sec:PIAG Theorem1} and Theorem~\ref{Sec:PIAG Theorem2} that the iteration complexity of the \textsc{Piag} method with cyclic sampling is $\mathcal{O}(n L / \epsilon)$ for convex problems and $\mathcal{O}\bigr(n Q \log(1/\epsilon)\bigl)$ for problems whose objective functions satisfy the quadratic functional growth condition. Each iteration of the \textsc{Piag} method requires only one gradient evaluation compared to $n$ gradient computations in the \textsc{Pg} method. Therefore, in terms of the total number of component gradients evaluated to find an $\epsilon$-optimal solution, the iteration complexity of \textsc{Pg} and the cyclic \textsc{Piag} are the same.

\begin{remark}
\cite{Schmidt:17} proposed a randomized variant of \textsc{Piag}, called stochastic average gradient~(\textsc{Sag}), where the component functions are sampled uniformly at random. The iteration complexity of the \textsc{Sag} method, in expectation, is $\mathcal{O}\bigl(\max\{n, Q\}\log(1/\epsilon)\bigr)$ for strongly convex problems and $\mathcal{O}\bigl((n+L)/\epsilon)\bigr)$ for convex problems. The maximum allowable step-size for the \textsc{Sag} method is larger than that of the \textsc{Piag} method, which can lead to improved empirical performance~\cite[Figure $1$]{Schmidt:17}. Note, however, that in some applications, the component functions must be processed in a particular deterministic order and, hence, random sampling is not possible. For example, in source localization or distributed parameter estimation over wireless networks, sensors may only communicate with their neighbors subject to certain constraints in terms of geography and distance, which can restrict the updates to follow a specific deterministic order~\citep{Blatt:07}.     
\end{remark}

\subsubsection{Comparison of Our Analysis with Prior Work}

\cite{Blatt:07} proposed the incremental aggregated gradient (\textsc{Iag}) method for solving unconstrained optimization problems. For a special case where each component function is quadratic, they showed that the \textsc{Iag} method with a constant step-size achieves a linear rate of convergence. The authors obtained these results using a perturbation analysis of the eigenvalues of a periodic linear system. However, the analysis in \cite{Blatt:07} only applies to quadratic objective functions and  provides neither an explicit convergence rate nor an explicit upper bound on the step-sizes that ensure linear convergence. 

\cite{Tseng:14} proved global convergence and local linear convergence for the \textsc{Piag} method in a more general setting where the gradients of the component functions are Lipschitz continuous and each component function satisfies a local error bound assumption. For an $L$-smooth and (possibly) non-convex function $F$, the proof of Theorem $4.1$ in~\cite{Tseng:14} shows that the iterates generated by \textsc{Piag} satisfy
\begin{align}
P(x_{k+1})-P^\star \leq P(x_{k})-P^\star  + \frac{\gamma^2 L}{2} \sum_{\ell = (k-\tau)_+}^{k}\|d_{\ell}\|^2 - \gamma\left(1 - \frac{\gamma L \tau}{2}\right)\|d_k\|^2,
\label{PIAG_Tseng}
\end{align}
where $d_k$ is the search direction at iteration $k$ and given by
\begin{align*}
d_k = \underset{d\in \mathbb{R}^d}{\textup{argmin}} \left\{\langle g_k,\; d \rangle 
 + \frac{1}{2}\|d\|^2 + R(x_k + d)\right\}.  
\end{align*}
By defining $V_k = P(x_{k})-P^\star$, $X_k = \gamma(1-h) \|d_k\|^2$ with $h\in(0, 1)$, and $W_k = \|d_k\|^2$, we can rewrite~\eqref{PIAG_Tseng} as
\begin{align*}
X_k + V_{k+1} \leq V_k  + \frac{\gamma^2 L}{2} \sum_{\ell = (k-\tau)_+}^{k} W_{\ell} - \gamma\left(h - \frac{\gamma L \tau}{2}\right) W_k.
\end{align*}
The iterates satisfy a relationship on the form~\eqref{Sec:Main Results Eq4}. Thus, according to Lemma~\ref{Sec:Main Results Lemma4}, if 
\begin{align*}
\frac{\gamma^2 L(\tau+1)}{2}\leq  \gamma\left(h - \frac{\gamma L \tau}{2}\right),   
\end{align*}
then $\sum_{k=0}^{K} X_{k} \leq V_0$ for every $K\in\mathbb{N}$. This, in turn, implies that
\begin{align*}
\sum_{k=0}^{K} \|d_k\|^2 &\leq \frac{P(x_{0})-P^\star}{\gamma(1-h)}, \quad \textup{for}\;\gamma = \frac{2h}{L(2\tau + 1)},  
\end{align*}
which is the same result as in~\cite{Tseng:14} for non-convex objective functions. While the analysis in~\cite{Tseng:14} is more general than that in~\cite{Blatt:07}, as the authors did not limit the objective function to be strongly convex and quadratic, explicit rate estimates and a characterization of the step-size needed for linear convergence were still missing in~\cite{Tseng:14}.

\cite{Gurbuzbalaban:17} provided the first explicit linear rate result for the \textsc{Iag} method. According to the proof of Theorem $3.3$ in \cite{Gurbuzbalaban:17}, the iterates generated by \textsc{Iag} satisfy a relationship of the form~\eqref{Sec:Main Results Eq3} given by
\begin{align*}
V_{k+1} &\leq \left(1-\frac{2\gamma \mu L}{\mu + L}\right) V_k +  \left(9\gamma^4 L^4 \tau^2 + 6 \gamma^2 L^2 \tau\right) \max_{(k-2\tau)_+\leq \ell \leq k} V_{\ell},
\end{align*}
where $V_k = \|x_k -x^\star\|^2$. As in \cite{Gurbuzbalaban:17}, we can use Lemma~\ref{Sec:Main Results Lemma1} to obtain an iteration complexity of $\mathcal{O}\bigl(Q^2\tau^2 \log(1/\epsilon)\bigr)$ for the \textsc{Iag} method. However, this analysis has a drawback that the guaranteed bound grows quadratically with both the condition number $Q$ and the maximum delay $\tau$.

The quadratic dependence of the iteration complexity on $Q$ was improved by \cite{AFJ:16} for strongly convex composite objectives and by \cite{Zhang:21} under a quadratic functional growth condition. Specifically, the proof of Theorem $1$ in~\cite{AFJ:16} shows that for the iterates generated by \textsc{Piag}, we have
\begin{align*}
V_{k+1} \leq \left(\frac{1}{1 + \gamma \mu}\right)V_k  + \left(\frac{\gamma L (\tau+1)}{1 + \gamma \mu}\right) \sum_{\ell = (k-\tau)_+}^{k} W_{\ell} - \left(\frac{1}{1 + \gamma \mu}\right) W_k,
\end{align*}
where $V_k = \|x_k -x^\star\|^2$ and $W_k = \|x_{k+1}-x_k\|^2$. Since the iterates satisfy a relationship on the form~\eqref{Sec:Main Results Eq4}, it follows from Lemma~\ref{Sec:Main Results Lemma4} that 
\begin{align*}
\|x_k -x^\star\|^2 &\leq \left(\frac{1}{1 + \gamma \mu}\right)\|x_0 -x^\star\|^2, \quad \textup{for}\; \gamma \in\left(0, \frac{1}{L(2\tau + 1)(\tau+1)}\right].  
\end{align*}
The above convergence rate lead to an iteration complexity of $\mathcal{O}\bigl(Q\tau^2 \log(1/\epsilon)\bigr)$, which matches the bound derived in \cite{AFJ:16}. However, this analysis has some limitations. First, the number of iterations required to reach an $\epsilon$-optimal solution increases quadratically with the maximal delay $\tau$. Second, the constant step-size that guarantees linear convergence is inversely proportional to the square of $\tau$. 

\cite{Vanli:18} improved the quadratic dependence on $\tau$ to a linear one by showing that the iteration complexity of the \textsc{Piag} method is $\mathcal{O}\bigl(Q\tau \log(1/\epsilon)\bigr)$. According to the proof of Theorem~$3.9$ in~\cite{Vanli:18}, the iterates generated by \textsc{Piag} satisfy
\begin{align*}
V_{k+1} \leq \left(\frac{1}{1 + \gamma \mu/16}\right)V_k  + \left(\frac{3L}{4(1 + \gamma \mu/16)}\right) \sum_{\ell = (k-\tau)_+}^{k-1} W_{\ell} - \left(\frac{1}{4\gamma(1 + \gamma \mu/16)}\right) W_k,
\end{align*}
where $V_k =  P(x_k) - P^\star$ and $W_k = \|x_{k+1}-x_k\|^2$. Therefore, by Lemma~\ref{Sec:Main Results Lemma4}, we have 
\begin{align*}
P(x_k) - P^\star &\leq \left(\frac{1}{1 + \gamma \mu/16}\right)^k (P(x_0) - P^\star), \quad \textup{for}\;\gamma \in\left(0, \frac{1}{3L(2\tau + 1)}\right],  
\end{align*}
which is the same bound as in~\cite{Vanli:18}. In comparison with this result, Theorem~\ref{Sec:PIAG Theorem2} allows \textsc{Piag} to use larger step-sizes that leads to a tighter guaranteed convergence rate. This improvement is achieved through our choice of the sequence $V_k$, which includes two terms: $P(x_k) - P^\star$ and $\|x_k - x^\star\|^2$. In addition, the analysis in~\cite{Vanli:18} only applies to strongly convex problems.

We have thus shown that our sequence results can be used to obtain the convergence guarantees established in~\cite{Tseng:14, Gurbuzbalaban:17, AFJ:16, Vanli:18}, as their analysis involves recurrences of the form~\eqref{Sec:Main Results Eq3} or~\eqref{Sec:Main Results Eq4}. Upon comparing our analysis with previous work, it becomes clear that for a specific algorithm, such as \textsc{Piag}, the range of step-sizes and convergence rates guaranteed by Lemmas~\ref{Sec:Main Results Lemma1}-~\ref{Sec:Main Results Lemma4} depend heavily on the choice of sequences $V_k$, $W_k$, and $X_k$. Selecting these sequences involves considering the characteristics of the algorithm and properties of the optimization problem that the algorithm aims to solve. For example, for the \textsc{Piag} method, we see that 
\begin{align*}
V_k = P(x_{k})-P^\star    
\end{align*}
for non-convex problems, 
\begin{align*}
V_k = 2\gamma (k + \tau)\bigl(P(x_{k})-P^\star\bigr)  + \|x_k - x^\star\|^2
\end{align*}
for convex problems, and 
\begin{align*}
V_k = \frac{2}{L}\bigl(P(x_{k})-P^\star\bigr)  +  \|x_k -\Pi_{\mathcal{X}^\star}(x_k)\|^2
\end{align*}
for problems whose objective functions satisfy the quadratic functional growth condition result in sharper theoretical bounds and larger step-sizes than the previous state-of-the-art.

%
%
\subsection{Asynchronous \textsc{SGD}}
\label{Sec:SGD}

To demonstrate the versatility of the sequence results, and that they are not limited to deliver step-size rules and convergence rates that depend on the maximal delay, we now consider stochastic optimization problems of the form
\begin{align}
\underset{x\in \mathbb{R}^d}{\textup{minimize}}\quad F(x) := \mathbb{E}_{\xi \sim \mathcal{D}}\bigl[f(x, \xi)\bigr].
\label{Sec:SGD Eq1}
\end{align}
Here, $x$ is the decision variable, $\xi$ is a random variable drawn from the probability distribution $\mathcal{D}$, and $f(\cdot, \xi):\mathbb{R}^d \rightarrow \mathbb{R}$ is differentiable for each $\xi$. This objective captures, for example, supervised learning where $x$ represents parameters of a machine learning model to be trained. In this case, $\mathcal{D}$ is an unknown distribution of labelled examples, $\xi$ is a data point, $f(x, \xi)$ is the loss of the model with parameters $x$ on the data point $\xi$, and $F$ is the generalization error. We use $\mathcal{X}^\star$ to denote the set of optimal solutions of Problem~\eqref{Sec:SGD Eq1} and $F^\star$ to denote the corresponding optimal value.

In machine learning applications, the distribution $\mathcal{D}$ is often unknown, which makes it challenging to solve Problem~\eqref{Sec:SGD Eq1}. To support these applications, we do not assume knowledge of~$F$, only access to a stochastic oracle. Each time the oracle is queried with an $x \in \mathbb{R}^d$, it generates an independent and identically distributed (i.i.d.) sample $\xi$ from $\mathcal{D}$ and returns $\nabla f(x, \xi)$, which is an unbiased estimate of $\nabla F(x)$, i.e.,
\begin{align*}
\nabla F(x) = \mathbb{E}_{\xi \sim \mathcal{D}}\bigl[\nabla f(x,\xi)\bigr].  
\end{align*}
We then use the stochastic gradient $\nabla f(x,\xi)$, instead of $\nabla F(x)$, in the update rule of the  optimization algorithm that attempts to minimize $F$.

The classical stochastic gradient descent~(\textsc{Sgd}) method is among the first and the most commonly used algorithms developed for solving Problem~\eqref{Sec:SGD Eq1}. Its popularity comes
mainly from the fact that it is easy to implement and
has low computational cost per iteration. The \textsc{Sgd} method proceeds iteratively by drawing an i.i.d sample $\xi_k$ from $\mathcal{D}$, computing $\nabla f(x_k, \xi_k)$, and updating the current vector $x_k$ via
\begin{align*}
x_{k+1} = x_k - \gamma_k \nabla f(x_k, \xi_k), 
\end{align*}
where $\gamma_k$ is a step-size (or learning rate). For a $L$-smooth and convex function $F$, the iteration complexity of the \textsc{Sgd} method is
\begin{align*}
\mathcal{O}\left(\frac{L}{\epsilon} + \frac{ \sigma^2}{\epsilon^2}\right),    
\end{align*}
which is $\mathcal{O}(\sigma^2/\epsilon^2)$  asymptotically in $\epsilon$~\citep{Lan:12}. 

The \textsc{Sgd} method is inherently serial in the sense that gradient computations take place on a single processor which has access to the whole dataset
and updates iterations sequentially, i.e., one after another. However, it is often infeasible for a single machine to store and process the vast amounts of data that we encounter in practical problems. In these situations, it is common to implement the \textsc{Sgd} method in a master-worker architecture in which several worker processors compute stochastic gradients in parallel based on their portions of the dataset while a master processor stores the decision vector and updates the current iterate. The master-worker implementation can be executed in two ways: synchronous and asynchronous.

In the synchronous case, the master will perform an update and broadcast the new decision vector to the workers when it has collected stochastic gradients from \emph{all} the workers. Given $M$ workers, the master performs the following mini-batch \textsc{Sgd} update 
\begin{align*}
x_{k+1} = x_k - \frac{\gamma_k}{M} \sum_{m=1}^ M \nabla f\bigl(x_k, \xi_{k, m}\bigr), 
\end{align*}
when stochastic gradients $\nabla f\bigl(x_k, \xi_{k, 1}\bigr), \ldots, \nabla f\bigl(x_k, \xi_{k, M}\bigr)$ are computed and communicated back by workers. In terms of the total number of stochastic gradients evaluated to find an $\epsilon$-optimal solution, the iteration complexity of the mini-batch \textsc{Sgd} method is
\begin{align}
\mathcal{O}\left(\frac{M L}{\epsilon} + \frac{ \sigma^2}{\epsilon^2}\right)    
\label{SGD_minibatch_rate}
\end{align}
for $L$-smooth and convex objective functions~(\cite{Dekel:2012}). This bound is asymptotically $\mathcal{O}(\sigma^2/\epsilon^2)$, which is exactly the asymptotic iteration complexity
achieved by the \textsc{Sgd} method. However,
using $M$ workers in parallel, the mini-batch \textsc{Sgd} method can achieve updates at a rate roughly $M$ times faster. This means that mini-batch \textsc{Sgd} is expected to enjoy a near-linear speedup in the number of workers.

A main drawback of the mini-batch \textsc{Sgd} method is that the workers need to synchronize in each round and compute stochastic gradients at the same decision vector. Due to various factors, such as differences in computational capabilities and communication bandwidth, or interference from other running jobs, some workers may evaluate stochastic gradients slower than others. This causes faster workers to be idle during each iteration and the algorithm suffers from the \textit{straggler} problem, in which the algorithm can only move forward at the pace of the slowest worker. The asynchronous \textsc{Sgd} method offers a solution to this issue by allowing the workers to compute gradients at different rates without synchronization and letting the master to perform updates using outdated gradients. In other words, there is no need for workers to wait for the others to finish their gradient computations, and the master can update the decision vector every time it receives a stochastic gradient from some worker. For this method, the master performs the update
\begin{align}
x_{k+1} = x_k - \gamma_k  \nabla f\bigl(x_{k-\tau_k}, \xi_{k}\bigr),
\label{Sec:SGD Eq2}
\end{align}
where $\tau_k$ is the delay of the gradient at iteration $k$. The value of $\tau_k$, which is often much greater than zero, captures the staleness of the information used to compute the stochastic gradient involved in the update of $x_k$. In~\eqref{Sec:SGD Eq2}, the index of stochastic noise $\xi_k$ is equal to the iteration number $k$ to show that previous iterates $x_{k'}$ for $k'\leq k$ do not depend on this stochastic noise~\citep{AgD:12}. Algorithm~\ref{Algorithm:SGD} describes the asynchronous \textsc{Sgd} method executed in a master-worker setting with one master and $M$ workers. 

\begin{algorithm}[H]
\textbf{Input:}   $x_0 \in \mathbb{R}^d$
  \begin{algorithmic}[1]
    \State Server sends $x_0$ to all workers 
     \For {$k=0, 1, \ldots$ }
      \State Workers compute stochastic gradients at the assigned points in parallel 
      \State Gradient $\nabla f\bigl(x_{k-\tau_k}, \xi_{k}\bigr)$ arrives from some worker $m_k\in\{1,\ldots, M\}$
      \State Server updates $x_{k+1} \leftarrow x_{k} - \gamma_k \nabla f\bigl(x_{k-\tau_k}, \xi_{k}\bigr)$
     \State Server sends $x_{k+1}$ to worker $m_k$
    \EndFor
  \end{algorithmic}
  \caption{Asynchronous \textsc{Sgd}}
  \label{Algorithm:SGD}
\end{algorithm}

To characterize the iteration complexity and the convergence rate of Algorithm~\ref{Algorithm:SGD}, we make the following assumptions.

\begin{assumption}
\label{Sec:SGD Assumption1}
The objective function $F:\mathbb{R}^d\rightarrow \mathbb{R}$ is $L$-smooth.  
\end{assumption}
\begin{assumption}
\label{Sec:SGD Assumption2}
The stochastic gradients have bounded variance, i.e., there exists a
constant $\sigma \geq 0$ such that
\begin{align*}
\mathbb{E}_{\xi \sim \mathcal{D}}\bigl[\|\nabla f(x,\xi) - \nabla F(x)\|^2\bigr] \leq \sigma^2.
\end{align*}
\end{assumption}
\begin{assumption}
\label{Sec:SGD Assumption3}
There is a non-negative integer $\tau_{\max}$ such that
\begin{align*}
0\leq \tau_k \leq \tau_{\max},\quad k\in\mathbb{N}_0.
\end{align*}
\end{assumption}

Most existing theoretical guarantees for the asynchronous \textsc{Sgd} method show that the number of iterations required to reach an $\epsilon$-optimal solution grows with the maximum delay~$\tau_{\max}$ (see, e.g., \cite{AgD:12, Lian:15, Feyzmahdavian:16, Mania:17, Arjevani:2020, Karimireddy:2020}). In particular,~\cite{Arjevani:2020} analyzed positive semi-definite quadratic functions and~\cite{Karimireddy:2020} considered general $L$-smooth convex functions and proved that if the delays are always constant ($\tau_k = \tau_{\max} $ for $k\in\mathbb{N}_0$), the asynchronous \textsc{Sgd} method attains an iteration complexity bound of the form
\begin{align}
\mathcal{O}\left(\frac{L(\tau_{\max}+1)}{\epsilon} + \frac{ \sigma^2}{\epsilon^2}\right).
\label{SGD_rate_constantdelay}
\end{align}
This bound is unable to exclude the possibility that the performance of asynchronous \textsc{Sgd} degrades proportionally to~$\tau_{\max}$. In heterogeneous environments involving workers with different computational speeds, $\tau_{\max}$ is effectively determined by the slowest worker. This suggests that similar to mini-batch \textsc{Sgd}, the asynchronous \textsc{Sgd} method may face the straggler issue. The step-size used to achieve the bound~\eqref{SGD_rate_constantdelay} is inversely proportional to~$\tau_{\max}$, and may hence be quite small which degrades the performance of the algorithm in practice.

To mitigate the impact of stragglers on asynchronous optimization,~\cite{Aviv:2021} and~\cite{Cohen:2021} developed delay-adaptive \textsc{Sgd} methods whose rates depend on the \textit{average} delay $\tau_{\textup{ave}}$ rather than on $\tau_{\max}$. Performance of these methods are more robust to asynchrony since $\tau_{\textup{ave}}$ can be significantly smaller than $\tau_{\max}$, especially when some workers are much slower than others. For instance, consider the implementation of the asynchronous \textsc{Sgd} method using two workers, where one worker is $1000$ times faster than the other in computing gradients. When the master waits for the slower worker to finish computing a gradient, it can use the faster worker's gradients to produce $1000$ updates. As a result, $\tau_k$ is equal to $0$ for $k=0,\ldots,999$, and $\tau_{1000}=1000$. This indicates that the average delay $\tau_{\textup{ave}}$ would be approximately $1$, while $\tau_{\max}=1000$. Although this scenario is hypothetical and constructed to make a point, actual delays tend to have $\tau_{\max}\gg \tau_{\rm ave}$. For example, the delay measurements for a $40$ worker implementation of asynchronous \textsc{Sgd} reported in \cite[Figure~$1$]{Mishchenko:2022} have $\tau_{\max}=1200$, while $\tau_{\rm ave}=40$.

While the methods developed in~\cite{Aviv:2021} and~\cite{Cohen:2021} are more robust to straggling workers, there are some limitations to their convergence guarantees. Specifically, the rates in \cite{Aviv:2021} were derived under the assumption that the gradients are uniformly bounded, while the results in~\cite{Cohen:2021} only hold with probability $\frac{1}{2}$. \cite{Koloskova:2022} and~\cite{Mishchenko:2022} have recently addressed these limitations, and shown that asynchronous \textsc{Sgd} is always faster than mini-batch \textsc{Sgd} regardless of the delay patterns. In particular, \cite{Koloskova:2022} proposed a delay-adaptive step-size rule, under which asynchronous \textsc{Sgd} achieves an iteration complexity 
\begin{align*}
\mathcal{O}\left(\frac{L(\tau_{\textup{ave}}+1)}{\epsilon} + \frac{ \sigma^2}{\epsilon^2}\right)
\end{align*}
for $L$-smooth non-convex functions. \cite{Mishchenko:2022} provided convergence guarantees for $L$-smooth convex, strongly convex, and non-convex functions which depend only on the number of workers used to implement the algorithm.

Similar to~\cite{Koloskova:2022}, we investigate the convergence of the asynchronous \textsc{Sgd} method with delay-dependent step-sizes. We define the step-size sequence as
\begin{align}
\gamma_k = 
\begin{cases}
\gamma,\quad & \tau_k\leq \tau_{\textup{th}},\\
0,\quad & \tau_k>\tau_{\textup{th}},
\end{cases}
\label{SGD_Adaptive_Stepsize}
\end{align}
where the threshold parameter $\tau_{\textup{th}}$ is a non-negative constant satisfying
\begin{align*}
\min\left\{2\tau_{\textup{ave}},\;\tau_{\textup{max}}\right\}\leq \tau_{\textup{th}}.  
\end{align*}
The adaptive rule~\eqref{SGD_Adaptive_Stepsize} sets the step-size to a constant value $\gamma$ when the delay of the $k$th iteration is at most $\tau_{\textup{th}}$. Otherwise, it sets the step-size to zero, effectively dropping the gradients with large delays. Note that~\eqref{SGD_Adaptive_Stepsize} reduces to a constant step-size rule by setting $\tau_{\textup{th}} \geq \tau_{\textup{max}}$. In this case, the step-size remains fixed at $\gamma$ for all iterations, regardless of the delay values. The main goal of this section is to use the sequence results from Section~\ref{Sec:Main Results} to: $(i)$ extend the results of~\cite{Koloskova:2022} to $L$-smooth convex and strongly convex functions; $(ii)$ improve upon the previously best-known rate of the asynchronous \textsc{Sgd} method with constant step-sizes for strongly convex functions given in~\cite{Karimireddy:2020}; and $(iii)$ recover the convergence rates derived in~\cite{Mishchenko:2022}.

\begin{remark}
\label{Remark_sgd_averagedelay}
To use the delay-dependent step-size rule~\eqref{SGD_Adaptive_Stepsize}, prior knowledge of $\tau_{\textup{ave}}$ is required. This is challenging since gradient delays are usually difficult to predict before implementing the algorithm. However,~\cite{Koloskova:2022} and~\cite{Mishchenko:2022} have shown that for the asynchronous \textsc{Sgd} method running on $M$ workers, it holds  that
\begin{align*}
\tau_{\textup{ave}} \leq M - 1.   
\end{align*}
Therefore, Algorithm~\ref{Algorithm:SGD} can be implemented without prior knowledge of $\tau_{\textup{ave}}$ by setting
\begin{align*}
\tau_{\textup{th}} = 2(M-1).    
\end{align*}
\end{remark}

The following result shows that in the analysis of the asynchronous \textsc{Sgd} method, we can establish iterate relationships on the form~\eqref{Sec:Main Results Eq4}.

\begin{lemma}
\label{Sec:SGD Lemma1}
Suppose that Assumptions~\ref{Sec:SGD Assumption1}--~\ref{Sec:SGD Assumption3} hold. Let 
\begin{align*}
V_k = \mathbb{E}\bigl[\|x_{k} - x^\star\|^2\bigr],\;
X_k = 2\gamma_k \mathbb{E}\bigl[F(x_{k}) - F^\star\bigr],\; \textup{and}\; W_k= \gamma_k^2\mathbb{E}\bigl[\bigl\| \nabla F(x_{k-\tau_k}) \bigr\|^2\bigr]
\end{align*}
for $k\in\mathbb{N}_0$, where the expectation is over all random variables $\xi_0,\ldots, \xi_{k-1}$. For the iterates generated by Algorithm~\ref{Algorithm:SGD}, the following statements hold:
\begin{enumerate}
    \item If $F$ is convex, then
\small    
\begin{align*}
X_k + V_{k+1} &\leq V_k +  2\gamma_k\tau_k L \sum_{\ell = (k - \tau_k)_+}^{k-1} W_{\ell} - \left(\frac{1}{\gamma_k L} - 1\right) W_k + \left(\gamma_k^2 + 2\gamma_k L \sum_{\ell = (k - \tau_k)_+}^{k-1} \gamma_{\ell}^2\right)\sigma^2
\end{align*}
\normalsize
for $\gamma_k>0$, and $ V_{k+1} = V_k$ for $\gamma_k=0$.
\item If $F$ is $\mu$-strongly convex, then
    \small
\begin{align*}
V_{k+1} &\leq (1 - \gamma_k \mu) V_k +  2\gamma_k\tau_k L \sum_{\ell = (k - \tau_k)_+}^{k-1} W_{\ell} - \left(\frac{1}{\gamma_k L} - 1\right) W_k + \left(\gamma_k^2 + 2 \gamma_k L\sum_{\ell = (k - \tau_k)_+}^{k-1} \gamma_{\ell}^2\right)\sigma^2
\end{align*}
\normalsize
for $\gamma_k>0$, and $ V_{k+1} = V_k$ for $\gamma_k=0$.
\end{enumerate}
\end{lemma}

\begin{proof}
See Appendix~\ref{Sec:SGD Lemma1 Proof}.
\end{proof}

Using this iterate relationship, Lemma~\ref{Sec:Main Results Lemma4} allows us to derive convergence
guarantees for the asynchronous~\textsc{Sgd} method. Our proof is based on the observation that, even if some gradients with large delays are discarded, gradients with delays less than or equal to the threshold parameter $\tau_{\textup{th}}$ are sufficient to ensure the convergence of the algorithm.

\begin{theorem}
\label{Sec:SGD Theorem1}
Suppose Assumptions~\ref{Sec:SGD Assumption1}--~\ref{Sec:SGD Assumption3} hold. Let $\tau_{\textup{th}}$ be a non-negative constant such that
\begin{align*}
\min\left\{2\tau_{\textup{ave}},\;\tau_{\textup{max}}\right\}\leq \tau_{\textup{th}},   
\end{align*}
where $\tau_{\textup{ave}}$ is the average delay, i.e.,
\begin{align*}
\tau_{\textup{ave}} = \frac{1}{K+1} \sum_{k=0}^K \tau_k.    
\end{align*}
Using the delay-dependent step-size rule~\eqref{SGD_Adaptive_Stepsize} in Algorithm~\ref{Algorithm:SGD} yields the following results:
\begin{enumerate}
\item For convex $F$ and any $x^\star \in \mathcal{X}^\star$, choosing
\begin{align*}
\gamma \in \left(0, \frac{1}{L(\tau_{\textup{th}}\sqrt{2}+1)}\right],     
\end{align*}
guarantees that
\begin{align*}
\mathbb{E}\bigl[F(\bar{x}_K) - F^\star \bigr] \leq \frac{\|x_{0} - x^\star\|^2 }{\gamma (K + 1)} + (1 + \sqrt{2})\gamma \sigma^2, \quad K\in\mathbb{N}_0,
\end{align*}
where $\bar{x}_K$ is the weighted average of $x_0,\ldots, x_K$ defined as
\begin{align*}
\bar{x}_K = \frac{1}{\sum_{k=0}^K \gamma_k} \sum_{k=0}^K \gamma_k x_k.   
\end{align*}
\item For strongly convex $F$,
choosing
\begin{align*}
\gamma \in \left(0, \frac{1}{L(2\tau_{\textup{th}}+1)}\right], 
\end{align*}
ensures that
\begin{align*}
\mathbb{E}\bigl[\|x_{K} - x^\star\|^2\bigr] \leq \textup{exp}\left(-\frac{\gamma \mu K}{2}\right) \|x_{0} - x^\star\|^2 + \frac{2\gamma\sigma^2}{\mu},\quad K\in \mathbb{N}.
\end{align*}
\end{enumerate}
\end{theorem}

\begin{proof}
See Appendix~\ref{Sec:SGD Theorem1 Proof}.
\end{proof}

According to Theorem~\ref{Sec:SGD Theorem1}, the asynchronous \textsc{Sgd} method converges to a ball around the optimum at a rate of $\mathcal{O}(1/K)$ for convex functions and at a linear rate for strongly convex functions. The choice of $\gamma$ in the delay-adaptive step-size rule~\eqref{SGD_Adaptive_Stepsize} affects both the convergence rate and the residual error: decreasing $\gamma$ reduces the residual error, but it also results in a slower convergence. Next, we present a possible strategy for
selecting $\gamma$ to achieve an $\epsilon$-optimal solution.

\begin{theorem}
\label{Sec:SGD Theorem2}
Suppose that Assumptions~\ref{Sec:SGD Assumption1}--~\ref{Sec:SGD Assumption3} hold. Given any $\epsilon>0$, the following statements hold: 
\begin{enumerate}
\item For convex $F$, choosing  
\begin{align*}
\gamma = \min \left\{ \frac{1}{L(\tau_{\textup{th}}\sqrt{2}+1)},\; \frac{\epsilon}{2(\sqrt{2}+1)\sigma^2}\right\}  
\end{align*}
implies $\mathbb{E}\bigl[F(\bar{x}_K) - F^\star \bigr] \leq \epsilon$ after
\begin{align*}
\mathcal{O}\left(\frac{L(\tau_{\textup{th}}+1)}{\epsilon} + \frac{ \sigma^2}{\epsilon^2}\right).   
\end{align*}
iterations. 
\item For strongly convex $F$, setting $\gamma$ to
\begin{align*}
\gamma = \min \left\{ \frac{1}{L(2\tau_{\textup{th}}+1)},\;\frac{\epsilon \mu}{4\sigma^2}\right\}, 
\end{align*}
guarantees that $\mathbb{E}\bigl[\|x_{K} - x^\star\|^2\bigr]\leq \epsilon$ after 
\begin{align*}
\mathcal{O}\left(\left(\frac{L(\tau_{\textup{th}}+1)}{\mu} + \frac{\sigma^2}{\epsilon\mu^2}\right)\log\left(\frac{1}{\epsilon}\right)\right).  
\end{align*}
iterations.
\end{enumerate}
\end{theorem}

\begin{proof}
See Appendix~\ref{Sec:SGD Theorem2 Proof}.
\end{proof}

According to Theorem~\ref{Sec:SGD Theorem2}, the value of $\gamma$ needed to achieve an $\epsilon$-optimal solution depends on two terms. The role of the first term is to decrease the effects of asynchrony on the convergence of the algorithm while the second term aims to control the noise from stochastic gradient information. The first term is monotonically decreasing in the threshold parameter~$\tau_{\textup{th}}$. At the same time, the guaranteed bounds on the iteration complexity are monotonically increasing in $\tau_{\textup{th}}$. Therefore, choosing a smaller value for $\tau_{\textup{th}}$ allows the algorithm to use larger step-sizes by discarding more gradients with long delays, and results in fewer iterations required to attain an $\epsilon$-optimal solution. 

When the threshold parameter $\tau_{\textup{th}}$ is set to $2\tau_{\textup{ave}}$, Theorem~\ref{Sec:SGD Theorem2} shows that the asynchronous \textsc{Sgd} method with delay-adaptive step-sizes achieves an iteration complexity of
\begin{align*}
\mathcal{O}\left(\frac{L(\tau_{\textup{ave}}+1)}{\epsilon} + \frac{ \sigma^2}{\epsilon^2}\right)  
\end{align*}
for convex objective functions. The average delay $\tau_{\textup{ave}}$ only appears in the first term and its negative impact on the convergence rate is asymptotically negligible once
\begin{align}
\tau_{\textup{ave}} \leq \mathcal{O}\left(\frac{ \sigma^2}{L\epsilon}\right).
\label{SGD_Speedup}
\end{align}
In this case, the asynchronous \textsc{Sgd} method has the same asymptotic iteration complexity as the serial \textsc{Sgd} method. Our asymptotic rate $\mathcal{O}(\sigma^2/\epsilon^2)$ for asynchronous \textsc{Sgd} is consistent with the results in \cite{Arjevani:2020} and \cite{Karimireddy:2020}. However, their requirement to guarantee such rate is
\begin{align*}
\tau_{\max} \leq \mathcal{O}\left(\frac{ \sigma^2}{L\epsilon}\right),
\end{align*}
which is more conservative than~\eqref{SGD_Speedup} since necessarily $\tau_{\textup{ave}}\leq\tau_{\max}$. 

According to Remark~\ref{Remark_sgd_averagedelay}, we can select the threshold parameter as $\tau_{\textup{th}} = 2(M-1)$. By doing so, it follows from Theorem~\ref{Sec:SGD Theorem2} that for convex problems, we have 
\begin{align*}
\mathcal{O}\left(\frac{L M}{\epsilon} + \frac{ \sigma^2}{\epsilon^2}\right).  
\end{align*}
By comparing this bound with~\eqref{SGD_minibatch_rate}, we can see that the asynchronous \textsc{Sgd} method with delay-adaptive step-sizes attains the same iteration complexity as the mini-batch \textsc{Sgd} method. Since each update of mini-batch \textsc{Sgd} takes the time needed by the slowest worker, its expected per-iteration time is slower than that of asynchronous \textsc{Sgd}. This means that asynchronous \textsc{Sgd} outperforms mini-batch \textsc{Sgd} regardless of the delays in the gradients, as previously proved by~\cite{Koloskova:2022} and \cite{Mishchenko:2022}.

By setting $\tau_{\textup{th}} = 2\tau_{\textup{ave}}$, the iteration complexity for Algorithm~\ref{Algorithm:SGD} is 
\begin{align*}
\mathcal{O}\left(\left(\frac{L(\tau_{\textup{ave}}+1)}{\mu} + \frac{\sigma^2}{\epsilon\mu^2}\right)\log\left(\frac{1}{\epsilon}\right)\right)
\end{align*}
for strongly convex objective functions. In the case that $\tau_{\textup{ave}}=0$, the preceding guaranteed bound reduces to the one obtained in~\cite{Gower:19} for the serial \textsc{Sgd} method. We can see that the average delay $\tau_{\textup{ave}}$ can be as large as 
$\mathcal{O}\left({ \sigma^2}/{(\mu L\epsilon)}\right)$ without affecting the asymptotic rate of Algorithm~\ref{Algorithm:SGD}. 

An alternative approach to tune the step-size in Algorithm~\ref{Algorithm:SGD} is to use $\gamma$ that depends on a prior knowledge of the number of iterations to be performed~\citep{Lan:12, Mishchenko:2022}. Assume that the number of
iterations is fixed in advance, say equal to $\mathcal{K}$. The following result provides the convergence rate of Algorithm~\ref{Algorithm:SGD} in terms of $\mathcal{K}$.

\begin{theorem}
\label{Sec:SGD Theorem3}
Let Assumptions~\ref{Sec:SGD Assumption1}--~\ref{Sec:SGD Assumption3} hold. Given $\mathcal{K}\in \mathbb{N}$, the following statements hold: 
\begin{enumerate}
\item For convex $F$ and
\begin{align*}
\gamma = \min \left\{ \frac{1}{L(\tau_{\textup{th}}\sqrt{2}+1)},\; \frac{\|x_{0} - x^\star\|}{\sigma\sqrt{\sqrt{2} + 1}\sqrt{\mathcal{K} + 1}}\right\},  
\end{align*}
Algorithm~\ref{Algorithm:SGD} ensures
\begin{align*}
\mathbb{E}\bigl[F(\bar{x}_{\mathcal{K}}) - F^\star \bigr]= \mathcal{O}\left(\frac{L(\tau_{\textup{th}}+1)}{\mathcal{K}} + \frac{\sigma}{\sqrt{\mathcal{K}}}\right).    
\end{align*}
\item For strongly convex $F$ and
\begin{align*}
\gamma = \min \left\{ \frac{1}{L(2\tau_{\textup{th}}+1)},\;\frac{2}{\mu \mathcal{K}}\log\left(1 + \frac{\mu^2\mathcal{K}\|x_{0} - x^\star\|^2}{4\sigma^2}\right)\right\}, 
\end{align*}
Algorithm~\ref{Algorithm:SGD}  guarantees that  
\begin{align*}
\mathbb{E}\bigl[\|x_{\mathcal{K}} - x^\star\|^2\bigr] = \tilde{\mathcal{O}}\left(\exp\left(-\frac{\mu \mathcal{K}}{L(\tau_{\textup{th}}+1)}\right)+ \frac{\sigma^2}{\mathcal{K}}\right).  
\end{align*}
\end{enumerate}
\end{theorem}

\begin{proof}
See Appendix~\ref{Sec:SGD Theorem3 Proof}.
\end{proof}

For strongly convex problems, the previously best known convergence rate under constant step-sizes was given in~\cite{Karimireddy:2020} and is expressed as 
\begin{align*}
\tilde{\mathcal{O}}\left(\tau_{\max}\exp\left(-\frac{\mu \mathcal{K}}{L(\tau_{\max}+1)}\right)+ \frac{\sigma^2}{\mathcal{K}}\right).
\end{align*}
However, when we set $\tau_{\textup{th}} = \tau_{\max}$, Theorem~\ref{Sec:SGD Theorem3} tells us that the asynchronous \textsc{Sgd} method with a constant step-size converges at a rate of
\begin{align*}
\tilde{\mathcal{O}}\left(\exp\left(-\frac{\mu \mathcal{K}}{L(\tau_{\max}+1)}\right)+ \frac{\sigma^2}{\mathcal{K}}\right).
\end{align*}
Therefore, we obtain a sharper guaranteed convergence rate than the one presented in \cite{Karimireddy:2020}. Note also that our theoretical guarantee matches with
the bound in~\cite{Arjevani:2020} derived for strongly convex quadratic functions. 

By selecting the threshold parameter as $\tau_{\textup{th}} = 2(M-1)$, it follows from Theorem~\ref{Sec:SGD Theorem3} that the the asynchronous \textsc{Sgd} method achieves a convergence rate of 
\begin{align*}
\mathbb{E}\bigl[F(\bar{x}_{\mathcal{K}}) - F^\star \bigr]= \mathcal{O}\left(\frac{L M}{\mathcal{K}} + \frac{\sigma}{\sqrt{\mathcal{K}}}\right) 
\end{align*}
for convex functions and 
\begin{align*}
\tilde{\mathcal{O}}\left(\exp\left(-\frac{\mu \mathcal{K}}{L M}\right)+ \frac{\sigma^2}{\mathcal{K}}\right)
\end{align*}
for strongly convex functions, which match the rates provided by~\cite{Mishchenko:2022}. 

\subsubsection{Comparison of Our Analysis with Prior Work}

\cite{Arjevani:2020} analyzed the asynchronous \textsc{Sgd} method for convex quadratic functions and provided the first tight convergence rate. However, their proof technique is based on the generating function method, which is only applicable to quadratic functions. To extend these results to general smooth functions, \cite{Karimireddy:2020} employed the perturbed iterate technique, which was originally developed by~\cite{Mania:17} to analyze asynchronous optimization algorithms. In Theorem~\ref{Sec:SGD Theorem3}, we establish a convergence rate for convex problems that matches the one derived by \cite{Karimireddy:2020} while also providing a tighter guarantee for strongly convex problems. Next, we show how to recover the convergence rate result for non-convex objective functions given in \cite{Karimireddy:2020}. For $L$-smooth (possibly non-convex) functions $F$, the proof of Theorem~$16$ in \cite{Karimireddy:2020} shows that the iterates generated by asynchronous \textsc{Sgd} satisfy
\begin{align}
 \mathbb{E}\bigl[F(\tilde{x}_{k+1}) - F^\star]  \leq \mathbb{E}\bigl[F(\tilde{x}_{k}) - F^\star] & +  \gamma^3 L^2\tau_{\max}  \sum_{\ell = (k - \tau_{\max})_+}^{k-1} \mathbb{E}\bigl[\|\nabla F(x_{\ell})\|^2\bigr] \nonumber\\
&\hspace{-2.5cm} - \frac{\gamma \left(1 - L\gamma\right)}{2}\mathbb{E}\bigl[\|\nabla F(x_{k})\|^2\bigr] + \frac{(3\gamma^3 L^2 \tau_{\max} + \gamma^2L)\sigma^2}{2},
\label{SGD_Stich}
\end{align}
where $\{\tilde{x}_k\}$ with $\tilde{x}_0 = x_0$ is a virtual sequence defined in the proof. Let 
\begin{align*}
V_k = \mathbb{E}\bigl[F(\tilde{x}_{k}) - F^\star\bigr],\; W_k = \mathbb{E}\bigl[\|\nabla F(x_{k})\|^2\bigr],\; X_k = \frac{\gamma}{4} \mathbb{E}\bigl[\|\nabla F(x_{k})\|^2\bigr],
\end{align*}
and $e = {(3\gamma^3 L^2 \tau_{\max} + \gamma^2L)\sigma^2}/{2}$. 
Then, we can rewrite~\eqref{SGD_Stich} as
\begin{align*}
X_k + V_{k+1} \leq V_k & +  \gamma^3 L^2\tau_{\max}  \sum_{\ell = (k - \tau_{\max})_+}^{k-1} W_\ell - \frac{\gamma \left(1 - 2L\gamma\right)}{4}W_k + e,
\end{align*}
which is on the form~\eqref{Sec:Main Results Eq4}. Thus, according to Lemma~\ref{Sec:Main Results Lemma4}, if 
\begin{align*}
\gamma^3 L^2 \tau_{\max}^2 \leq  \frac{\gamma\left(1 - 2 L\gamma\right)}{4},   
\end{align*}
then $\sum_{k=0}^{K} X_{k} \leq V_0 + (K+1)e$ for $K\in\mathbb{N}$. This implies that
\begin{align*}
\sum_{k=0}^{K} \mathbb{E}\bigl[\|\nabla F(x_{k})\|^2\bigr] \leq \frac{4\bigl(F({x}_{0}) - F^\star\bigr)}{\gamma} + {5\gamma L(K+1)\sigma^2},\quad \textup{for}\; \gamma \in \left(0, \frac{1}{2L (\tau_{\max} + 1)}\right].
\end{align*}
By setting 
\begin{align*}
\gamma = \min \left\{ \frac{1}{2L (\tau_{\max} + 1)}, \sqrt{\frac{4\bigl(F({x}_{0}) - F^\star\bigr)}{5 L(K+1) \sigma^2}}\right\},
\end{align*}
which minimizes the right-hand side of the inequality above with respect to $\gamma$, we obtain 
\begin{align*}
\min_{0\leq k \leq K}\left\{ \mathbb{E}\bigl[\|\nabla F(x_{k})\|^2\bigr]\right\} = \mathcal{O}\left(\frac{L(\tau_{\max} + 1)\bigl(F({x}_{0}) - F^\star\bigr)}{K+1} + \sigma \sqrt{\frac{L \bigl(F({x}_{0}) - F^\star\bigr)}{K+1}}\right).
\end{align*}
This convergence guarantee matches the bound derived in~\cite{Karimireddy:2020}.

Also focusing on non-convex problems, \cite{Koloskova:2022} proposed a delay-adaptive step-size rule that eliminates the dependence of the convergence rate on $\tau_{\max}$. According to the proof of Theorem~$8$ in \cite{Koloskova:2022}, the following inequality holds for the sequence generated by asynchronous \textsc{Sgd}:
\begin{align*}
X_k + V_{k+1} &\leq V_k +  \gamma_k\tau_k L^2 \sum_{\ell = (k - \tau_k)_+}^{k-1} W_{\ell} - \frac{1}{4\gamma_k} W_k + \left(\gamma_k^2 L + \gamma_k L^2 \sum_{\ell = (k - \tau_k)_+}^{k-1} \gamma_{\ell}^2\right)\sigma^2,
\end{align*}
where $V_k = \mathbb{E}\bigl[F(x_k) - F^\star\bigr]$, $X_k = \frac{\gamma_k}{2} \mathbb{E}\bigl[\|\nabla F(x_{k})\|^2\bigr]$, and $W_k= \gamma_k^2\mathbb{E}\bigl[\bigl\| \nabla F(x_{k-\tau_k}) \bigr\|^2\bigr]$. Similar to the proof for Theorem~\ref{Sec:SGD Theorem1}, it follows from Lemma~\ref{Sec:Main Results Eq4} that
\begin{align*}
\frac{1}{\sum_{k=0}^{K} \gamma_k}\sum_{k=0}^{K} \gamma_k \mathbb{E}\bigl[\|\nabla F(x_{k})\|^2\bigr] \leq \frac{4\bigl(F({x}_{0}) - F^\star\bigr)}{\gamma(K+1)} + {6\gamma L\sigma^2},\quad \textup{for}\; \gamma \in \left(0, \frac{1}{2L (\tau_{\textup{th}} + 1)}\right].
\end{align*}
By setting $\tau_{\textup{th}} = 2\tau_{\textup{ave}}$, the above convergence rate leads to an iteration complexity of
\begin{align*}
\mathcal{O}\left(\frac{\tau_{\textup{ave}}}{\epsilon} + \frac{\sigma^2}{\epsilon^2}\right).
\end{align*}
This bound is the same as the one provided in~\cite{Koloskova:2022}. Alternatively, choosing $\tau_{\textup{th}} = 2(M-1)$ results in 
\begin{align*}
\mathcal{O}\left(\frac{M}{\epsilon} + \frac{\sigma^2}{\epsilon^2}\right),
\end{align*}
which matches the iteration complexity derived in \cite{Mishchenko:2022} using the perturbed iterate technique. While we recover the convergence guarantees presented in \cite{Mishchenko:2022}, recurrences of the form~\eqref{Sec:Main Results Eq3} and~\eqref{Sec:Main Results Eq4} did not appear in their proofs. The reason is that the analysis in \cite{Mishchenko:2022} relies on upper bounds for the delayed terms that are expressed in terms of $M$, rather than on the delay values themselves.

%
%
\subsection{Asynchronous Coordinate Update Methods}
\label{Sec:ARock}
Many popular optimization algorithms, such as gradient descent, projected gradient descent, proximal-point, forward-backward splitting, Douglas-Rachford splitting, and the alternating direction method of multipliers~(\textsc{Admm}), can be viewed as special instances of the Krasnosel'ski\u{i}–Mann (\textsc{Km}) method for operators~\citep{Peng:16}. The only way they differ is in their choice of operator. The \textsc{Km} method has the form
\begin{align}
x_{k+1} = (1-\gamma)x_k + \gamma T(x_k),\quad k\in\mathbb{N}_0,
\label{KM-iteration1}
\end{align}
where $\gamma \in(0,1)$ is the step-size, and $T:\mathbb{R}^d \rightarrow \mathbb{R}^d$ is a nonexpansive operator, i.e.,
\begin{align*}
\| T(x) - T(y) \|\leq \|x - y \|, \quad \forall x,y\in \mathbb{R}^d.
\end{align*}
Any vector $x^\star \in\mathbb{R}^d$ satisfying $T(x^\star) = x^\star$ is called a \emph{fixed point} of $T$ and the \textsc{Km} method can be viewed as an algorithm for finding such a fixed point. Let $S = I_d - T$, where $I_d$ is the identity operator on $\mathbb{R}^d$. Then,~\eqref{KM-iteration1} can be rewritten as
\begin{align}
x_{k+1} = x_k - \gamma S(x_k),\quad k\in\mathbb{N}_0.
\label{KM-iteration2}
\end{align}
This shows that the \textsc{Km} method can be interpreted as a taking a step of length $\gamma$ in the opposite direction of~$S$ evaluated at the current iterate. For example, the gradient descent method for minimization of a $L$-smooth convex function $f:\mathbb{R}^d \rightarrow \mathbb{R}$ can be formulated as the \textsc{Km} method with 
\begin{align*}
S &= \frac{2}{L} \nabla f,\quad \textup{or equivalently}, \quad T = I_d - \frac{2}{L} \nabla f.
\end{align*}
We represent $x$ as $x=\bigl([x]_1,\ldots, [x]_m\bigr)$, where $[x]_i \in \mathbb{R}^{d_i}$ with $d_1$, $\ldots, d_m$ being positive integer numbers satisfying $d = d_1 + \ldots + d_m$. We denote by $S_i:\mathbb{R}^{d}\rightarrow \mathbb{R}^{d_i}$ the $i$th block of $S$, so $S(x)=\bigl(S_1(x),\ldots, S_m(x)\bigr)$. Thus, we can rewrite~\eqref{KM-iteration2} as
\begin{align*}
\left[x_{k+1}\right]_i = \left[x_k\right]_i - \gamma S_i(x_k),\quad i\in [m].
\end{align*}

ARock is an algorithmic framework for parallelizing the \textsc{Km} method in an asynchronous fashion~\citep{Peng:16}. In ARock, multiple agents (machines, processors, or cores) have access to a shared memory for storing the vector~$x$, and are able to read and update $x$ simultaneously without using locks. Conceptually, ARock lets each agent repeat the following steps:
\begin{itemize}
\item Read $x$ from the shared memory without locks and save it in a local cache as $\widehat{x}$;
\item Choose an index $i\in [m]$ uniformly at random and use $\widehat{x}$ to compute $S_i(\widehat{x})$;
\item Update component $i$ of the shared $x$ via
\begin{align*}
[x]_i \leftarrow [x]_i - \gamma S_i(\widehat{x}).
\end{align*}
\end{itemize}
Since the agents are being run independently without synchronization, while one agent is busy reading~$x$ and evaluating $S_i(x)$, other agents may repeatedly update the value stored in the shared memory. Therefore, the value $\widehat{x}$ read from the shared memory may differ from the value of $x$ to which the update is made later. In other words, each agent can update the shared memory using possibly out-of-date information.

\begin{algorithm}[H]
\textbf{Input:}   $x_0 \in \mathbb{R}^d$, step-size $\gamma>0$, number of iterations $K\in \mathbb{N}$
  \begin{algorithmic}[1]
  \State Initialize global counter $k\leftarrow 0$
     \While {$k < K$}
      \State Read each position of shared memory denoted by $\widehat{x}_k$
      \State Sample $i_k$ from $\{1,\ldots, m\}$ with equal probability $\frac{1}{m}$
      \State Set $[x_{k+1}]_{i_k} \leftarrow [x_{k}]_{i_k} - \gamma S_{i_k}(\widehat{x}_k)$
     \State Set $k \leftarrow k+1$
    \EndWhile
  \end{algorithmic}
  \caption{ARock}
  \label{Algorithm:ARock}
\end{algorithm}

Algorithm~\ref{Algorithm:ARock} describes ARock. We assume that the write operation on line~$5$ is atomic, in the sense that the updated result will successfully appear in the shared memory by the end of the execution. In practice, this assumption can be enforced through compare-and-swap operations~\citep{RRW+:11}. Updating a scalar is a single atomic instruction on most modern hardware. Thus, the atomic write assumption in Algorithm~\ref{Algorithm:ARock} naturally holds when each block is a single scalar, i.e., $m=d$ and $d_i =1$ for all $i \in [m]$.

We define one \textit{iteration} of ARock as a modification on any block of $x$ in the shared memory. A global counter $k$ is introduced to track the total number of iterations so that $x_k$ is the value of $x$ in the shared memory after $k$ iterations. We use $i_k$ to denote the component that is updated at iteration $k$, and $\widehat{x}_k$ for value of $x$ that is used in the calculation of $S_{i_k}(\cdot)$. Since every iteration changes one block of the shared memory and all writes are required to be atomic, 
\begin{align}
x_k &= \widehat{x}_k + \sum_{j\in J_k}(x_{j+1} - x_{j}),
\label{ArockRead}
\end{align}
where $J_k\subseteq\{0, \ldots, k-1\}$. The set $J_k$ contains the indices of the iterations during which a block of shared memory is updated between the time when $\hat{x}_k$ is read and the time when $x_k$ is written. Thus, the sum in~\eqref{ArockRead} represents all updates of the shared vector that have occurred from the time that agent $i_k$ begins reading $\widehat{x}_k$ from memory until it finishes evaluating $S_{i_k}( \widehat{x}_k)$ and writes the result to memory (see~\cite{Peng:16} for a more detailed discussion). Note that individual blocks of the shared vector may be updated by multiple agents while it is in the process of being read by another agent. Therefore, the components of $\widehat{x}_k$ may have different ages and the vector $\widehat{x}_k$ may never actually exist in the shared memory during the execution of Algorithm~\ref{Algorithm:ARock}. This phenomenon is known as \textit{inconsistent read}~\citep{Liu:15}.

To analyze the convergence of ARock, we need to make a few assumptions similar to~\cite{Peng:16}.

\begin{assumption}
\label{Assumption:ARock}
For Algorithm~\ref{Algorithm:ARock}, the following properties hold:
\begin{enumerate}
\item \textbf{(Pseudo-contractivity)} The operator $T = I_d -S$ has a fixed point and is pseudo-contractive with respect to the Euclidean norm with contraction modulus $c$. That is, there exists $c\in (0,1)$ such that
\begin{align*}
\| T(x) - x^\star\| \leq c \| x - x^\star\|,\quad \forall x\in\mathbb{R}^d.
\end{align*}
\item \textbf{(Bounded delay)} There is a non-negative integer $\tau$ such that
\begin{align*}
(k - \tau)_+ \leq \min\bigl\{j\;|\; j \in J_k \bigr\}, \quad \forall k\in\mathbb{N}_0.
\end{align*}
\item \textbf{(Independence)} Random variables $i_k$ for $k=0,\ldots, K$ are independent of each other.
\end{enumerate}
\end{assumption} 
Under Assumption~\ref{Assumption:ARock}.1, the serial \textsc{Km} iteration~\eqref{KM-iteration2} with $\gamma\in(0,1]$ converges to the fixed point $x^\star$ at a linear rate~\cite[Chapter $3$]{BeT:89}. Assumption~\ref{Assumption:ARock}.2 guarantees that during any update cycle of an agent, the vector~$x$ in the shared memory is updated at most $\tau$ times by other agents. Therefore, no component of $\widehat{x}_k$ is older than $\tau$ for all $k\in\mathbb{N}_0$. The value of $\tau$ is an indicator of the degree of asynchrony in ARock. In practice, $\tau$ will depend on the number of agents involved in the computation. If all agents are working at the same rate, we would expect $\tau$ to be a multiple of the number of agents~\citep{Wright:15}. Similar to~\cite{Peng:16} and most results on asynchronous stochastic optimization for shared memory architecture (e.g., \cite{RRW+:11,Liu:14,Liu:15}), we assume that the age of the components of $\widehat{x}_k$ is independent of the block~$i_k$ being updated at iteration $k$. However, this may not hold in practice if, for example, some blocks are more expensive to update than others~\citep{Leblond:18}. This independence assumption can be relaxed using several techniques such as \textit{before read labeling}~\citep{Mania:17}, \textit{after read labeling}~\citep{Leblond:18}, \textit{single coordinate consistent ordering}~\citep{Cheung:20}, and \textit{probabilistic models of asynchrony}~\citep{Sun:17,Cannelli:20}.

The following result shows that in the analysis of ARock, we can establish iterate relationships on the form~\eqref{Sec:Main Results Eq4}.

\begin{lemma}
\label{Sec:ARock Lemma1}
Suppose that Assumption~\ref{Assumption:ARock} holds. Let $V_k = \mathbb{E}\bigl[\|x_{k} - x^\star\|^2\bigr] $ and $W_k =  \mathbb{E}\bigl[\|S(\widehat{x}_k)\|^2\bigr]$ for $k\in\mathbb{N}_0$, where the expectation is over all choices of index $i_k$ up to step $k$. Then, the iterates generated by Algorithm~\ref{Algorithm:ARock} satisfy
\begin{align*}
V_{k+1} &\leq \left(1 - \frac{\gamma(1-c^2)}{m\left(1+\gamma \left(\frac{\tau}{m} + \sqrt{\frac{\tau}{m}}\right)\right)} \right) V_k  + \frac{2\gamma^2}{m\tau}\left(\frac{\tau}{m} + \sqrt{\frac{\tau}{m}}\right)\sum_{\ell = (k - \tau)_+}^{k-1}W_ {\ell} \\
&\hspace{5mm} - \frac{\gamma}{m}\left(1 - \gamma\left(1 + \frac{\tau}{m} + \sqrt{\frac{\tau}{m}}\right)\right) W_k.
\end{align*}
\end{lemma}

\begin{proof}
See Appendix~\ref{Sec:ARock Lemma1 Proof}.
\end{proof}

Using this iterate relationship, the sequence result in Lemma~\ref{Sec:Main Results Lemma4} allows us to derive new convergence guarantees for ARock. 

\begin{theorem}
\label{Theorem:ARock}
Suppose that Assumption~\ref{Assumption:ARock} holds and that the step-size $\gamma$ is set to  
\begin{align*}
\gamma = \frac{h}{1+5\left(\frac{\tau}{m} + \sqrt{\frac{\tau}{m}}\right)}
\end{align*}
with $h\in (0,1]$. Then, the iterates generated by Algorithm~\ref{Algorithm:ARock} satisfy
\begin{align}
\mathbb{E}\bigl[\|x_{k} - x^\star\|^2\bigr]  \leq \left(1 - \frac{h(1-c^2)}{m\left(1+6 \left(\frac{\tau}{m} + \sqrt{\frac{\tau}{m}}\right)\right)} \right)^k \|x_{0} - x^\star\|^2
\label{ARock:Theorem1}
\end{align}
for all $k\in\mathbb{N}_0$. Moreover, the algorithm reaches an accuracy of  $\mathbb{E}\bigl[\|x_{k} - x^\star\|^2\bigr]\leq \epsilon$ after
\begin{align}
k \geq K_{\epsilon} = \frac{m\left(1+6 \left(\frac{\tau}{m} + \sqrt{\frac{\tau}{m}}\right)\right)}{h(1-c^2)} \log\left(\frac{\|x_0 - x^\star\|^2}{\epsilon}\right)
\label{ARock:Theorem2}
\end{align}
iterations.
\end{theorem}

\begin{proof}
See Appendix~\ref{Sec:ARock Theorem1 Proof}.
\end{proof}

Theorem~\ref{Theorem:ARock} shows that for pseudo-contractive operators, ARock converges in expectation at a linear rate. To quantify how $\tau$ can affect the convergence of ARock we define
%
\begin{align*}
\Gamma = \frac{\tau}{m} + \sqrt{\frac{\tau}{m}}.
\end{align*}
Clearly, $\Gamma$ is monotonically increasing in $\tau$ and is equal to zero for $\tau =0$. The maximum allowable step-size and the linear rate of ARock depend on $\Gamma$. As $\Gamma$ tends to infinity, the maximum allowable step-size decreases and approaches zero while the convergence factor increases and approaches one. Therefore, ARock should take a smaller step-size for a larger $\Gamma$, which will lead to a slower convergence rate. Note that if $\Gamma \approx 0$, or equivalently, $\tau \ll m$, then ARock with step-size $\gamma \approx 1$ has the linear rate
\begin{align*}
\rho = 1- \frac{1-c^2}{m},
\end{align*}
which is precisely the rate for the serial stochastic \textsc{Km} method~\citep{Hannah:17}. As discussed before, $\tau$ is related to the number of agents used in the algorithm. Therefore, the number of agents can be of the order of $o(m)$ without appreciably degrading the convergence rate of ARock.

The serial \textsc{Km} method~\eqref{KM-iteration2} with constant step-size $\gamma_{\textsc{Km}} =h\in (0,1]$ needs 
\begin{align*}
k\geq\frac{1}{h(1-c^2)} \log\left(\frac{\|x_0 - x^\star\|^2}{\epsilon}\right)
\end{align*}
iterations to satisfy $\|x_{k} - x^\star\|^2\leq \epsilon$~\citep{Hannah:17}. Each iteration of~\eqref{KM-iteration2} requires computing all $m$ blocks of $S$ and updating the whole vector~$x$. Thus, the overall complexity of the serial \textsc{Km} method, in terms of the total number of blocks updated to find an  $\epsilon$-optimal solution, is
\begin{align*}
K_{\textsc{Km}} = \frac{m}{h(1-c^2)} \log\left(\frac{\|x_0 - x^\star\|^2}{\epsilon}\right).
\end{align*}
On the other hand, by Theorem~\ref{Theorem:ARock}, ARock with step-size $\gamma_{\textup{ARock}} = h/(1+5\Gamma)$ performs
\begin{align*}
K_{\textup{ARock}} = \frac{m\left(1+6 \Gamma\right)}{h(1-c^2)} \log\left(\frac{\|x_0 - x^\star\|^2}{\epsilon}\right)
\end{align*}
component updates to return an $\epsilon$-optimal solution in average. In the case that $\Gamma \approx 0$, we have $\gamma_{\textup{ARock}} \approx \gamma_{\textsc{Km}}$ and $K_{\textup{ARock}}\approx K_{\textsc{Km}} $. Hence, as long as $\tau$ is bounded by $o(m)$, ARock can use the same step-size as the serial  \textsc{Km} method and achieve the same iteration bound. Furthermore, since ARock runs on $p$ agents in parallel, updates can occur roughly $p$ times more frequently, leading to a near-linear speedup in the number of agents.

\subsubsection{Comparison of Our Analysis with Prior Work}
\cite{Peng:16} proposed ARock for finding fixed points of nonexpansive operators in an asynchronous parallel fashion. They proved that if $T$ is nonexpansive and $S=I -T$ is  quasi-$\mu$-strongly monotone\footnote{The operator $S:\mathbb{R}^d \rightarrow \mathbb{R}^d$ is quasi-$\mu$-strongly monotone if it satisfies $\langle x -y,\; Sx \rangle \geq \mu \|x -y\|^2$ for all $x\in\mathbb{R}^d$ and $y\in \{z\in\mathbb{R}^d\; |\; Sz = 0\}.$}, then ARock with step-size
\begin{align*}
\gamma = \frac{1}{1 + \mathcal{O}\left(\frac{\tau^2}{\sqrt{m}}\right)}
\end{align*}
converges linearly to a fixed point and achieves a linear speedup for $\tau\leq \mathcal{O}\left(m^{1/4}\right)$. However, we will show that using Lemma~\ref{Sec:Main Results Lemma4} in their proof allows to improve the linear speed condition from $\tau\leq \mathcal{O}\left(m^{1/4}\right)$ to $\tau\leq \mathcal{O}\left(m^{1/2}\right)$. According to the proof of Theorem $4$ in \cite{Peng:16}, the iterates generated by ARock satisfy
\begin{align*}
V_{k+1} \leq \left(1 - \frac{\gamma \mu}{ 2 m} \right) V_k  + \frac{\gamma^2}{m^2}\left(\gamma \mu \tau  +  \sqrt{m}\right)\sum_{\ell = (k - \tau)_+}^{k-1}W_ {\ell} 
 - \frac{\gamma^2}{m}\left(\frac{1}{2\gamma} - 1 - \frac{\tau}{\sqrt{m}} \right) W_k, 
\end{align*}
where $V_k = \mathbb{E}\bigl[\|x_{k} - x^\star\|^2 $ and $W_k =  \mathbb{E}\bigl[\|S(\widehat{x}_k)\|^2\bigr]$ for $k\in\mathbb{N}_0$. Since this inequality is on the form~\eqref{Sec:Main Results Eq4}, it follows from Lemma~\ref{Sec:Main Results Lemma4} that if
\begin{align*}
\frac{\gamma^2}{m^2}\left(\gamma \mu \tau + \sqrt{m}\right)(2\tau)\leq  \frac{\gamma^2}{m}\left(\frac{1}{2\gamma} - 1 - \frac{\tau}{\sqrt{m}} \right),   
\end{align*}
then $V_k$ converges at a linear rate. This implies that
\begin{align*}
\mathbb{E}\bigl[\|x_{k} - x^\star\|^2 \leq \left(1 - \frac{\gamma \mu}{ 2 m} \right)^k \|x_{0} - x^\star\|^2,\quad \textup{for}\; \gamma \in \left(0,\; \frac{1}{2\left(1 + \frac{\tau(\sqrt{\mu} + 3)}{\sqrt{m}}\right)}\right].
\end{align*}
Therefore, $\tau$ can be as large as $\mathcal{O}\left(m^{1/2}\right)$ without affecting the maximum allowable step-size. Hence, using Lemma~\ref{Sec:Main Results Lemma4} in the proof of Theorem $4$ in \cite{Peng:16} improves the upper bound on $\tau$ by a factor of $m^{1/4}$.

\cite{Hannah:17} improved the results in \cite{Peng:16} by showing that for pseudo-contractive operators,
\begin{align*}
\gamma = \frac{1}{1 + \mathcal{O}\left(\frac{\tau}{\sqrt{m}}\right)}
\end{align*}
guarantees the linear convergence of ARock and $\tau\leq \mathcal{O}\left(m^{1/2}\right)$ ensures a linear speedup. Compared to the result presented in~\cite{Hannah:17}, not only can Theorem~\ref{Theorem:ARock} provide a larger value for the maximal allowable step-size, but it also improves the requirement for the linear speedup property from $\tau \leq \mathcal{O}\left(m^{1/2}\right)$ to $\tau \leq o\left(m\right)$. The analysis in~\cite{Hannah:17} involves a recurrence of the form
\begin{align*}
V_{k+1} &\leq q V_k + p \sum_{\ell = (k-\tau_k)_+}^{k} W_{\ell} - r W_k,\quad k\in\mathbb{N}_0,
\end{align*}
with the same quantities for $V_k$ and $W_k$ as our analysis. However, our coefficients for $q$, $p$, and $r$ are different, since we employ tighter upper bounds on the delayed terms. It is, in this case, the use of Lemmas~\ref{Lemma1_ARock} and~\ref{Lemma2_ARock} 
to bound $\mathbb{E}_k\bigl[\|x_k - \widehat{x}_k\|^2\bigr]$ and $\mathbb{E}_k\bigl[\langle \widehat{x}_k - x_{k},  S(\widehat{x}_k)\bigr\rangle\bigr]$, respectively, 
that are the keys to our improved results.

\paragraph{A Special Case: Asynchronous Coordinate Descent Method.} 
We now present a special case of ARock, namely the asynchronous coordinate descent algorithm for minimizing a class of composite objective functions. Specifically, we consider the problem
\begin{align}
\underset{x\in \mathbb{R}^d}{\textup{minimize}}\quad P(x) :=F(x) + R(x),
\label{ARock:OptimizationProblem}
\end{align}
where $F:\mathbb{R}^d \rightarrow \mathbb{R}$ is $\mu$-strongly convex and $L$-smooth, and $R:\mathbb{R}^d \rightarrow \mathbb{R}$ is separable in all coordinates, i.e.,
\begin{align*}
R(x) = \sum_{i=1}^m r_i\bigl([x]_i\bigr).
\end{align*}
Here, each $r_i:\mathbb{R}^{d_i} \rightarrow \mathbb{R}\cup \{+\infty\}$, $i\in [m]$, is a closed, convex, and extended real-valued function. The best known examples of separable regularizers include $\ell_1$ norm, $\ell_2$ norm square, and the indicator function of box constraints~\citep{Fercoq:15}. The minimizer of Problem~\eqref{ARock:OptimizationProblem} is the unique fixed point of $T_{\textup{prox}}$ defined as
\begin{align*}
T_{\textup{prox}}(x) = \textup{prox}_{\frac{2}{\mu + L}R}\left(x -\frac{2}{\mu + L} \nabla F(x)\right).
\end{align*}
The operator $T_{\textup{prox}}$ is contractive with contraction modulus
\begin{align*}
c = \frac{Q - 1}{Q + 1},
\end{align*}
where $Q = L/\mu$~\citep{Bertsekas:15}. To solve~\eqref{ARock:OptimizationProblem}, we apply ARock to $S = I_d - T_{\textup{prox}}$. Then, the update rule of Algorithm~\ref{Algorithm:ARock} at the $k$th iteration becomes
\begin{align*}
[x_{k+1}]_{i_k} \leftarrow [x_{k}]_{i_k} - \gamma\left([\widehat{x}_{k}]_{i_k} - \textup{prox}_{\frac{2}{\mu + L}r_{i_k}}\left([\widehat{x}_{k}]_{i_k} -\frac{2}{\mu + L} \nabla_{i_k} F(\widehat{x}_k)\right)\right),
\end{align*}
where $\nabla_i F(x)$ denotes the partial gradient of $F$ with respect to $[x]_i$. According to Theorem~\ref{Theorem:ARock}, the iterates generated by ARock with step-size $\gamma=1/(1+5\Gamma)$ satisfy
\begin{align*}
\mathbb{E}\bigl[\|x_{k} - x^\star\|^2\bigr]  \leq \left(1 - \frac{4Q}{m\left(1+6 \Gamma\right)(Q+1)^2}\right)^k \|x_{0} - x^\star\|^2,\quad k\in\mathbb{N}_0.
\end{align*}
If $\Gamma \approx 0$, or equivalently, $\tau \ll m$, then ARock using $\gamma \approx 1$ converges linearly at a rate of
\begin{align}
\rho_{\textup{ARock}} = 1 - \frac{4Q}{m(Q+1)^2}.
\label{ARockRate}
\end{align}
Note that when $m=1$, the preceding linear rate reduces to
\begin{align*}
\rho_{\textup{GD}} = \left(\frac{Q-1}{Q+1}\right)^2,
\end{align*}
which is the best convergence rate for the gradient descent method applied to strongly convex optimization.

\begin{remark}
\cite{Liu:15} proposed an asynchronous coordinate descent algorithm for solving optimization problems of the form~\eqref{ARock:OptimizationProblem}. They proved that the linear speedup is achievable if $\tau\leq \mathcal{O}\left(m^{1/4}\right)$. For a special case of Problem~\eqref{ARock:OptimizationProblem} where $R(x)\equiv 0$,~\cite{Liu:14} showed that the asynchronous coordinate descent method can enjoy the linear speedup if $\tau\leq \mathcal{O}\left(m^{1/2}\right)$. In comparison with~\cite{Liu:15} and~\cite{Liu:14}, our requirement for the linear speedup property is $\tau\leq o\left(m\right)$ and, hence, allows a larger value for $\tau$.
Recently,~\cite{Cheung:20} analyzed convergence of the asynchronous coordinate descent method for composite objective functions without assuming independence  between $i_k$ and $\widehat{x}_k$. Their analysis guarantees the linear speedup for $\tau\leq \mathcal{O}\left(m^{1/2}\right)$.
\end{remark}

\begin{remark}
In Theorem~\ref{Theorem:ARock}, the linear convergence of ARock is given in terms of the expected quadratic distance from the iterates to the fixed point. Note however that the literature on coordinate descent algorithms usually establishes convergence results using coordinate-wise Lipschitz constants of the function $F$ (see, e.g.,~\cite{Liu:15, Liu:14, Nesterov:12}). This allows to provide larger step-sizes which can lead to potentially better convergence bounds, especially in terms of the function values $P(x_k) - P^\star$. 
\end{remark}

%
%
\subsection{A Lyapunov Approach to Analysis of Totally Asynchronous Iterations}
\label{Sec:ContractiveMapping}
Finally, we study iterations involving maximum norm pseudo-contractions under the general asynchronous model introduced by~\cite{BeT:89}, which allows for heterogeneous and time-varying communication delays and update rates. Such iterations arise, for example, in algorithms for the solution of certain classes of linear equations, optimization problems and variational inequalities~\citep{BeT:89, Moallemi:10, Hale:17}, distributed algorithms for averaging~\citep{Mehyar:07}, power control algorithms for wireless networks~\citep{Feyzmahdavian:12}, and reinforcement algorithms for solving discounted Markov decision processes~\citep{Zeng:20}. We will demonstrate how the convergence results in~\cite{BeT:89} for maximum norm contractions can be derived and extended using Lemmas~\ref{Sec:Main Results Lemma1}--\ref{Sec:Main Results Lemma3}. This allows to unify and expand the existing results for partially and totally asynchronous iterations.

We consider iterative algorithms on the form
\begin{align}
x_{k+1} =T(x_k),\quad k\in \mathbb{N}_0,
\label{Sec:ContractiveMapping Eq1}
\end{align}
where $T:\mathbb{R}^d\rightarrow \mathbb{R}^{d}$ is a continuous mapping. This iteration aims to find a fixed point of $T$, that is, a vector $x^\star \in\mathbb{R}^{d}$ satisfying $x^\star = T(x^\star)$. Similar to Subsection~\ref{Sec:ARock}, we decompose the space $\mathbb{R}^d$ as a Cartesian product of $m$ subspaces:
\begin{align*}
\mathbb{R}^d = \mathbb{R}^{d_1} \times \cdots \times \mathbb{R}^{d_m},\quad d = \sum_{i=1}^m d_i.
\end{align*}
Accordingly, we can partition any vector $x\in\mathbb{R}^d$ as $x=\bigl([x]_1,\ldots, [x]_m\bigr)$ with $[x]_i \in \mathbb{R}^{d_i}$, $i\in [m]$. We denote by $T_i:\mathbb{R}^{d}\rightarrow \mathbb{R}^{d_i}$ the $i$th component of $T$, so $T(x)=\bigl(T_1(x),\ldots, T_m(x)\bigr)$. Then, we rewrite~\eqref{Sec:ContractiveMapping Eq1} as
\begin{align}
\bigl[x_{k+1}\bigr]_i = T_i\bigl([x_k]_1,\ldots, [x_k]_m\bigr), \quad  i \in [m].
\label{Sec:ContractiveMapping Eq2}
\end{align}
This iteration can be viewed as a network of $m$~agents, each responsible for updating one of the $m$ blocks of $x$ so as to find a global fixed point of the operator~$T$.

Let us fix some norms $\|\cdot\|_i$ for the spaces $\mathbb{R}^{d_i}$. The \textit{block-maximum norm} $\|\cdot\|_{b,\infty}^w$ on $\mathbb{R}^d$ is defined as
\begin{align*}
\bigl\|x\bigr\|_{b,\infty}^w&=\max_{i\in[m]} \;w_i \bigl\| [x]_i \bigr\|_i,
\end{align*}
where $w_i$ are positive constants. Note that when $d_i = 1$ for each $i \in [m]$, the block-maximum norm reduces to the maximum norm
\begin{align*}
\|x\|_{\infty}^w&=\max_{i\in[m]} \;w_i \bigl|[x]_i \bigr|.
\end{align*}
The following definition introduces \textit{pseudo-contractive} mappings with respect to the block-maximum norm. 

\begin{definition}
A mapping~$T:\mathbb{R}^d\rightarrow \mathbb{R}^d$ is called a pseudo-contraction with respect to the block-maximum norm if it has a fixed-point $x^{\star}$ and the property
\begin{align*}
\|T(x)-x^{\star}\|_{b,\infty}^w \leq c\;\|x-x^{\star}\|_{b,\infty}^w,\quad\forall x\in \mathbb{R}^d,
\end{align*}
where $c$, called the \textit{contraction modulus} of $T$, is a constant belonging to $(0,1)$.
\end{definition}
Pseudo-contractions have at most one fixed point~\cite[Proposition 3.1.2]{BeT:89}. Note that we follow~\cite{BeT:89} and include the existence of a fixed point in the definition of pseudo-contractions. Examples of iterative algorithms that involve pseudo-contractions with respect to the block-maximum norm can be found in~\cite[Chapter 3]{BeT:89}. An important property of pseudo-contractions is that they have a unique fixed point, to which the iterates produced by~\eqref{Sec:ContractiveMapping Eq2} converge at a linear rate~\cite[Proposition~3.1.2]{BeT:89}. More precisely, using the Lyapunov function $V_k = \|x_k -x^{\star}\|_{b,\infty}^w$, one can show that the sequence $\{x_k\}$ generated by (\ref{Sec:ContractiveMapping Eq2}) with initial vector $x_0\in \mathbb{R}^d$ satisfies $V_{k+1} \leq c V_k$ for all $k\in\mathbb{N}_0$. This implies that 
\begin{align*}
\|x_k -x^{\star}\|_{b,\infty}^w \leq c^k  \|x_0 -x^{\star}\|_{b,\infty}^w,\quad k\in\mathbb{N}_0.
\end{align*}

The algorithm described by~\eqref{Sec:ContractiveMapping Eq2} is \textit{synchronous} in the sense that all agents update their states at the same time and have instantaneous access to the states of all other agents. Synchronous execution is possible if there are no communication failures in the network and if all agents operate in sync with a global clock. In practice, these requirements are hard to satisfy: local clocks in different agents tend to drift, global synchronization mechanisms are complex to implement and carry substantial execution overhead, and the communication latency between agents can be significant and unpredictable. Insisting on synchronous operation  in an inherently asynchronous environment forces agents to spend  a significant time idle, waiting for the slowest agents (perhaps due to lower processing power or higher workload per iteration) to complete its work.

In an \textit{asynchronous} implementation of the iteration~\eqref{Sec:ContractiveMapping Eq2}, each agent can update its state at its own pace, using possibly outdated information from the other agents. This leads to iterations on the form
\begin{eqnarray}
[x_{k+1}]_i =\begin{cases}
       T_i\bigl ([x_{s_{i1, k}}]_1,\ldots, [x_{s_{i m, k}}]_m\bigr ), & k\in \mathcal{K}_i,\\
       [x_k]_i, & k\not\in \mathcal{K}_i.
	\end{cases}
\label{Sec:ContractiveMapping Eq3}
\end{eqnarray}
Here, $\mathcal{K}_i$ is the set of times when agent~$i$ executes an update, and $s_{i j,k}$ is the time at which the most recent version of $[x]_j$ available to agent $i$ at time $k$ was computed. The sets $\mathcal{K}_i$ need not be known to all the agents. Thus, there is no requirement for a shared global clock. Since the agents use only values computed in the previous iterations, $s_{i j,k}\leq k$ for all $k\in\mathbb{N}_0$. We can view the value
\begin{align*}
\tau_{i j, k}:= k - s_{i j, k}
\end{align*}
as the communication delay from agent~$j$ to agent~$i$ at time~$k$.  Clearly, the synchronous iteration~\eqref{Sec:ContractiveMapping Eq2} is a special case of~\eqref{Sec:ContractiveMapping Eq3} where $\mathcal{K}_i=\mathbb{N}_0$ and $\tau_{i j, k}=0$ for all $i,j \in [m]$ and all $k\in \mathbb{N}_0$.

Based on the assumptions on communication delays and update rates,~\cite{BeT:89} introduced a classification of asynchronous algorithms as either \textit{totally asynchronous} or \textit{partially asynchronous}. The totally asynchronous model is characterized by the following assumption. 
\begin{assumption}[\textbf{Total Asynchronism}]
\label{Sec:ContractiveMapping Assumption1}
For the asynchronous iteration~\eqref{Sec:ContractiveMapping Eq3}, the following properties hold:
\begin{enumerate}
\item The sets $\mathcal{K}_i$ are infinite subsets of $\mathbb{N}_0$ for each $i \in [m]$.
\item $\lim_{k\to \infty} {s}_{i j, k} = +\infty$ for all $i,j \in [m]$.
\end{enumerate}
\end{assumption}

Assumption~\ref{Sec:ContractiveMapping Assumption1}.1 guarantees that no agent ceases to execute its update while Assumption~\ref{Sec:ContractiveMapping Assumption1}.2 guarantees that outdated information about the agent updates is eventually purged from the computation (compare the discussion after Lemma~\ref{Sec:Main Results Lemma2}). Under total asynchronism, the communication delays $\tau_{i j, k}$ can become unbounded as $k$ increases. This is the main difference with partially asynchronous iterations, where delays are bounded; in particular, the following assumption holds.
\begin{assumption}[\textbf{Partial Asynchronism}]
\label{Sec:ContractiveMapping Assumption2}
For the asynchronous iteration~\eqref{Sec:ContractiveMapping Eq3}, there exist non-negative integers $B$ and $D$ such that the following conditions hold: 
\begin{enumerate}
\item At least one of the elements of the set $\{k, k+1,\ldots, k + B\}$ belongs to $\mathcal{K}_i$ for each $i \in [m]$ and all $k\in\mathbb{N}_0$.
\item $  k - D \leq s_{i j, k}\leq k$ for all $i,j \in [m]$ and all $k \in \mathcal{K}_i$.
\item $s_{i i, k} = k$ for all $i \in [m]$ and $k\in \mathcal{K}_i$.
\end{enumerate}
\end{assumption}

Assumptions~\ref{Sec:ContractiveMapping Assumption2}.1 and~\ref{Sec:ContractiveMapping Assumption2}.2 ensure that the time interval between updates executed by each agent and the communication delays are bounded by $B$ and $D$, respectively. This means that no agent waits an arbitrarily long time to compute or to receive a message from another agent. Assumption~\ref{Sec:ContractiveMapping Assumption2}.3 states that agent~$i$ always uses the latest version of its own component $[x]_i$. Note that when $B = D = 0$, the asynchronous iteration~\eqref{Sec:ContractiveMapping Eq3} under partial asynchronism reduces to the synchronous iteration~\eqref{Sec:ContractiveMapping Eq2}. 

We will now present a uniform analysis of the  iteration~\eqref{Sec:ContractiveMapping Eq3} involving block-maximum norm pseudo-contractions under both partial and total asynchronism, and study its convergence rate under different assumptions on the communication delays and update rates. To this end, we introduce $\tau_k$ to represent the maximum age of the outdated information being used to update blocks at global time $k\in\mathbb{N}_0$. Specifically, we define $\tau_k$ as 
\begin{align}
\tau_k := k - \min_{i \in [m]}\min_{j \in [m]} s_{i j, t_i(k)},
\label{Sec:ContractiveMapping Eq4}
\end{align}
where $t_i(k)$ is the most recent update time of agent $i$ at $k\in\mathbb{N}_0$, i.e.,
\begin{align*}
t_i(k) = \max \bigl \{\kappa\; | \; \kappa\leq k \; \land \; \kappa\in \mathcal{K}_i\bigr \}.
\end{align*}
In this way, if $k\in \mathcal{K}_i$ then $t_i(k) = k$, and if $k\in (\kappa^-, \kappa^+)$ for two consecutive elements $\kappa^-$ and $\kappa^+$ of $\mathcal{K}_i$, then $t_i(k) = \kappa^-$. For simplicity, we assume that $0\in\mathcal{K}_i$ for each $i$, so that $t_i(k)$ is well defined for all $k\in\mathbb{N}_0$.

The next result shows that for the asynchronous iteration~\eqref{Sec:ContractiveMapping Eq3}, if $V_k = \|x_k - x^{\star}\|_{b,\infty}^w$ is used as a candidate Lyapunov function (similarly to the convergence analysis for the synchronous case), we can establish iterate relationships on the form~\eqref{Sec:Main Results Eq3}.

\begin{lemma}
\label{Sec:ContractiveMapping Lemma1}
Suppose that $T$ is pseudo-contractive with respect to the block-maximum norm with contraction modulus $c$. Let $V_k = \|x_k - x^{\star}\|_{b,\infty}^w$. Then, the iterates $\{x_k\}$ generated by the asynchronous iteration~\eqref{Sec:ContractiveMapping Eq3} satisfy
\begin{align*}
V_{k+1} \leq c \max_{(k-\tau_k)_+ \leq \ell \leq k} V_{\ell},\quad k\in\mathbb{N}_0,
\end{align*}
where $\tau_k$ is defined in~\eqref{Sec:ContractiveMapping Eq4}.
\end{lemma}

\begin{proof}
See Appendix~\ref{Sec:ContractiveMapping Lemma1 Proof}.
\end{proof}

We apply Lemma~\ref{Sec:Main Results Lemma2} to show that for pseudo-contractions with respect to the block-maximum norm, the asynchronous iteration~\eqref{Sec:ContractiveMapping Eq3} converges asymptotically to the fixed point under total asynchronism.

\begin{theorem}
\label{Sec:ContractiveMapping Theorem1}
Let Assumption~\ref{Sec:ContractiveMapping Assumption1} hold. Suppose that $T$ is pseudo-contractive with respect to the block-maximum norm. Then, the sequence $\{x_k\}$ generated by the asynchronous iteration~\eqref{Sec:ContractiveMapping Eq3} converges asymptotically to the unique fixed point of~$T$.
\end{theorem}

\begin{proof}
See Appendix~\ref{Sec:ContractiveMapping Theorem1 Proof}.
\end{proof}

While convergent synchronous algorithms may diverge in the face of asynchronism, Theorem~\ref{Sec:ContractiveMapping Theorem1} shows that pseudo-contractive mappings in the block-maximum norm can tolerate arbitrarily large communication delays and update intervals satisfying Assumption~\ref{Sec:ContractiveMapping Assumption1}. In addition, as the next result shows, these iterations admit an explicit convergence rate bound when they are executed in a partially asynchronous fashion.

\begin{theorem}
\label{Sec:ContractiveMapping Theorem2}
Let Assumption~\ref{Sec:ContractiveMapping Assumption2} hold. Suppose that $T$ is pseudo-contractive with respect to the block-maximum norm with contraction modulus $c$. Then, the sequence $\{x_k\}$ generated by the asynchronous iteration~\eqref{Sec:ContractiveMapping Eq3} satisfies
\begin{align*}
 \|x_k - x^{\star}\|_{b,\infty}^w \leq  c^{\frac{k}{B + D + 1}}\|x_0 - x^{\star}\|_{b,\infty}^w,\quad k\in\mathbb{N}_0.
\end{align*}
\end{theorem}

\begin{proof}
See Appendix~\ref{Sec:ContractiveMapping Theorem2 Proof}.
\end{proof}

According to Theorem~\ref{Sec:ContractiveMapping Theorem2}, the asynchronous iteration~\eqref{Sec:ContractiveMapping Eq3} involving block-maximum norm pseudo-contractions remains linearly convergent for bounded communication delays and update rates. Note that $c^{\frac{1}{B + D + 1}}$
is monotonically increasing with $B$ and $D$, and approaches one as either $B$ or $D$ tends to infinity. Hence, the guaranteed convergence rate of~\eqref{Sec:ContractiveMapping Eq3} slows down as either the delays increase in magnitude or agents execute less frequently. The convergence rate is directly 
related to the number of iterations required for the algorithm to converge. According to Theorem~\ref{Sec:ContractiveMapping Theorem2}, the asynchronous iteration~\eqref{Sec:ContractiveMapping Eq3} needs
\begin{align*}
k\geq \frac{B + D + 1}{-\log(c)}\log\left(\frac{\|x_0 - x^{\star}\|_{b,\infty}^w}{\epsilon}\right)
\end{align*}
iterations to satisfy $ \|x_k - x^{\star}\|_{b,\infty}^w\leq \epsilon$. We can see that the time to reach a fixed target accuracy deteriorates linearly with both $B$ and $D$. 

We now use  Lemma~\ref{Sec:Main Results Lemma3} to develop a result that provides guaranteed convergence rates for the asynchronous algorithm~\eqref{Sec:ContractiveMapping Eq3} under a rather broad family of communication delays and update rates, in between the partially and totally asynchronous models.

\begin{theorem}
\label{Sec:ContractiveMapping Theorem3}
Suppose that $T$ is pseudo-contractive with respect to the block-maximum norm with contraction modulus $c$. Suppose also that there exists a function $\Lambda:\mathbb{R}\rightarrow \mathbb{R}$ such that the following conditions hold:
\begin{itemize}
\item[(i)] $\Lambda(0)=1$.
\item[(ii)] $\Lambda$ is non-increasing.
\item [(iii)] $\lim_{k\rightarrow +\infty} \Lambda(k) = 0$ and 
\begin{align*}
c \Lambda(k - \tau_k) \leq \Lambda(k+1), \quad k\in \mathbb{N}_0,
\end{align*}
where $\tau_k$ is defined in~\eqref{Sec:ContractiveMapping Eq4}.
\end{itemize}
Then, the iterates $\{x_k\}$ generated by the asynchronous iteration~\eqref{Sec:ContractiveMapping Eq3} satisfy
\begin{align*}
 \|x_k - x^{\star}\|_{b,\infty}^w \leq \Lambda(k) \|x_ 0- x^{\star}\|_{b,\infty}^w,\quad k\in\mathbb{N}.
\end{align*}
\end{theorem}

\begin{proof}
See Appendix~\ref{Sec:ContractiveMapping Theorem3 Proof}.
\end{proof}

According to Theorem~\ref{Sec:ContractiveMapping Theorem3}, any function $\Lambda$ satisfying conditions $(i)$--$(iii)$ can be used to estimate the convergence rate of the asynchronous iteration~\eqref{Sec:ContractiveMapping Eq3}. Condition (iii) implies that the admissible choices for~$\Lambda$ depend on~$\tau_{k}$. This means that the rate at which the nodes execute their updates as well as the way communication delays tend large both affect the guaranteed convergence rate of~\eqref{Sec:ContractiveMapping Eq3}. For example, one can verify that under partial asynchronism, the function 
\begin{align*}
\Lambda(t) = c^{\frac{t}{B+D+1}}
\end{align*}
satisfies conditions~$(i)$--$(iii)$. In the following example, we use Theorem~\ref{Sec:ContractiveMapping Theorem3} and Corollary~\ref{Sec:Main Results Corollary1} to establish convergence rates for a particular class of totally asynchronous iterations, where the communication delays can grow unbounded at a linear rate. 

\paragraph{Example.} For the asynchronous iteration~\eqref{Sec:ContractiveMapping Eq3}, suppose that $\mathcal{K}_i = \mathbb{N}_0$ for each $i \in [m]$. Suppose also that there exist scalars $\alpha\in(0,1)$ and $\beta\geq 0$ such that
\begin{align}
(1-\alpha)k - \beta \leq s_{i j, k}, \quad i,j\in [m],\; k \in \mathcal{K}_i.
\label{Sec:ContractiveMapping Example1}
\end{align}
Note that~\eqref{Sec:ContractiveMapping Example1} implies that the communication delays $\tau_{i j,k}$ belong to the interval $[0, \alpha k + \beta]$ and may therefore grow unbounded. In this example, as $t_i(k) = k$, we have
\begin{align*}
(1-\alpha)k - \beta \leq \min_{i \in [m]}\min_{j \in [m]} s_{i j, t_i(k)}.
\end{align*}
Thus, by Lemma~\ref{Sec:ContractiveMapping Lemma1}, the iterates generated by the asynchronous iteration~\eqref{Sec:ContractiveMapping Eq3} satisfy
\begin{align*}
V_{k+1} &\leq q V_k + p \max_{(k - \tau_k)_+ \leq \ell \leq k} V_{\ell},\quad k\in\mathbb{N}_0,
\end{align*}
with $V_k=\|x_k - x^{\star}\|_{b,\infty}^w$, $q=0$, $p=c$, and $\tau_k \leq \alpha k + \beta$. Since $q+p=c<1$, it follows from Corollary~\ref{Sec:Main Results Corollary1} that the function 
\begin{align*}
\Lambda(t) = \left(\frac{\alpha t}{1-\alpha + \beta}  + 1\right)^{-\eta}
\end{align*}
with $\eta = \ln(c)/\ln(1-\alpha)$ satisfies conditions $(i)$--$(iii)$ of Theorem~\ref{Sec:ContractiveMapping Theorem3}. Therefore,
\begin{align*}
\|x_k - x^{\star}\|_{b,\infty}^w \leq  \left(\frac{\alpha k}{1-\alpha + \beta}  + 1\right)^{-\eta} \|x_0 - x^{\star}\|_{b,\infty}^w,\quad k\in\mathbb{N}_0.
\end{align*}
We can see that under unbounded communication delays satisfying~\eqref{Sec:ContractiveMapping Example1}, the convergence rate of the asynchronous iteration~\eqref{Sec:ContractiveMapping Eq3} is of order $\mathcal{O}(1/k^\eta)$. 

Although pseudo-contractions in the block-maximum norm converge when executed in a totally asynchronous manner, we note that in many applications, it is rare that the associated fixed-point iterations are maximum norm contractions. For instance, the gradient descent iterations for unconstrained optimization problems are maximum norm contractions if the Hessian matrix of the cost function is diagonally dominant~\cite[Section 3.1.3]{BeT:89}. The diagonal dominance assumption is quite strong and violated even by some strongly convex quadratic objective functions~\cite[Example 6.3.1]{BeT:89}.

\begin{remark}
\cite{BeT:89} proved that contractions in the block-maximum norm converge when executed in a totally asynchronous manner. However, they did not quantify how bounds on the communication delays and update rates of agents affect the convergence rate of the iterates. There are very few results in the literature on convergence rates of asynchronous iterations involving block-maximum norm contractions. Notable exceptions are the works of~\cite{Bertsekas:89} and~\cite{Zeng:20}, where the convergence rate of iterates was estimated under partial asynchronism. Theorems~\ref{Sec:ContractiveMapping Theorem1}, ~\ref{Sec:ContractiveMapping Theorem2} and~\ref{Sec:ContractiveMapping Theorem3} demonstrate that not only do we recover the asymptotic convergence results in~\cite{BeT:89} for maximum norm pseudo-contractions, but we also provide explicit bounds on the convergence rate of asynchronous iterations for various classes of bounded and unbounded communication delays and update rates.
\end{remark}

%
%
\section{Conclusions}
We have introduced a number of novel sequence results for asynchronous iterations that appear in the analysis of parallel and asynchronous algorithms. In contrast to previous analysis frameworks, which have used conservative bounds for the effects of asynchrony, our results attempt to capture the inherent structure in the asynchronous iterations. The results balance simplicity, applicability and power, and provide explicit bounds on how the amount of asynchrony affects the guaranteed convergence rates. To demonstrate the potential of the sequence results, we illustrated how they can be used to improve our theoretical understanding of several important classes of asynchronous optimization algorithms. First, we derived better iteration complexity bounds for the proximal incremental aggregated gradient method, reducing the dependence of the convergence times on the maximum delay and problem condition number. Second, we provided tighter guarantees for the asynchronous stochastic gradient descent method that depend on the average delay rather on the maximal delay. Third, we gave an improved analysis of the ARock framework for asynchronous block-coordinate updates of Krasnosel'ski\u{i}–Mann iterations, proving a larger range of admissible step-sizes, faster convergence rates and better scaling properties with respect to the number of parallel computing elements. Finally, we gave a uniform treatment of asynchronous iterations involving block-norm contractions under partial and (several versions of) total asynchronism.

%
%

\bibliographystyle{plain}


%
%
\clearpage
\begin{appendices}
\section*{Appendicies}

\section{Proofs for Section~\ref{Sec:Motivation}}
\label{Appendix_Movation}



\begin{lemma}
\label{Lemma_delayedgradient_1}
Suppose $f$ is $L$-smooth and $\mu$-strongly convex. Consider the delayed gradient method~\eqref{eqn:delayed_gd} with $\gamma \in \left(0, \frac{2}{\mu + L}\right]$ and let $\omega_k$ be the perturbation term defined in \eqref{eqn:delayed_gd_error}. Then,
\begin{align*}
\omega_k \leq\left( 
    \gamma^4 L^4 \tau^2 + 2\gamma^2L^2 \tau
    \right) \max_{(k-2\tau)_+\leq \ell \leq k}\left\{ \Vert x_\ell-x^{\star}\Vert^2 \right\}.
\end{align*}
\end{lemma}

\begin{proof}
Applying the Cauchy-Schwarz inequality, we obtain
\begin{align*}
    \omega_k \leq 2 \gamma
    \Vert \nabla f(x_k) - \nabla f(x_{k-\tau})\Vert \Vert x_k-\gamma\nabla f(x_k)-x^{\star}\Vert + \gamma^2 \Vert \nabla f(x_k) - \nabla f(x_{k-\tau}) \Vert^2.
\end{align*}
We use $L$-smoothness of $f$ to bound $\Vert \nabla f(x_k) - \nabla f(x_{k-\tau})\Vert$. Specifically,
\begin{align*}
    \Vert \nabla f(x_k) - \nabla f(x_{k-\tau})\Vert \leq L 
    \Vert x_k - x_{k-\tau}\Vert = L \left\Vert \sum_{\ell=(k-\tau)_+}^{k-1} x_{\ell +1} - x_\ell\right\Vert &\overset{\eqref{eqn:delayed_gd}}{=}
    L \left\Vert \sum_{\ell=(k-\tau)_+}^{k-1} \gamma\nabla f(x_{\ell-\tau})\right\Vert\\
    &\leq \gamma L \sum_{\ell=(k-\tau)_+}^{k-1} \Vert \nabla f(x_{\ell-\tau})\Vert,
\end{align*}
where the last inequality follows from the triangle inequality. Since $\nabla f(x^\star) = 0$, we have
\begin{align*}
 \Vert \nabla f(x_k) - \nabla f(x_{k-\tau})\Vert \leq \gamma L^2 \sum_{\ell=(k-\tau)_+}^{k-1} \Vert x_{\ell-\tau}-x^{\star}\Vert \leq \gamma L^2 \tau \max_{(k-2\tau)_+\leq \ell \leq k}\left\{ \Vert x_\ell-x^{\star}\Vert \right\}.  
\end{align*}
From~\eqref{Sec:Main Results Eq1}, $\Vert x_{k}-\gamma \nabla f(x_k)-x^{\star}\Vert\leq \Vert x_k-x^{\star}\Vert$ holds for the proposed range of $\gamma$. Thus,
\begin{align*}
\omega_k &\leq 2\gamma^2 L^2 \tau \Vert x_k-x^{\star}\Vert \max_{(k-2\tau)_+\leq \ell \leq k}\left\{ \Vert x_\ell-x^{\star}\Vert \right\} + \gamma^4 L^4 \tau^2 \max_{(k-2\tau)_+\leq \ell \leq k}\left\{ \Vert x_\ell-x^{\star}\Vert^2 \right\}\\
& \leq  \left( 
    \gamma^4 L^4 \tau^2 + 2\gamma^2L^2 \tau
    \right) \max_{(k-2\tau)_+\leq \ell \leq k}\left\{ \Vert x_\ell-x^{\star}\Vert^2 \right\}.
\end{align*}
\end{proof}

\begin{lemma}
\label{Lemma_delayedgradient_2}
Suppose $f$ is $L$-smooth and convex. Then, the iterates generated by~\eqref{eqn:delayed_gd} satisfy
\begin{align*}
f(x_k)-f^\star + \frac{1}{2L}\Vert \nabla f(x_{k-\tau})\Vert^2 - \frac{\gamma^2 L\tau}{2} \sum_{\ell=(k-\tau)_+}^{k-1} \Vert \nabla f(x_{\ell-\tau)}\Vert^2 \leq \langle \nabla f(x_{k-\tau}), x_k-x^{\star} \rangle.
\end{align*}
\end{lemma}

\begin{proof}
According to Theorem $2.1.5$ in~\cite{Nesterov:13}, we have 
\begin{align}
f(x_k)&\leq f(x_{k-\tau}) + \bigl\langle \nabla f(x_{k-\tau}),\; x_k - x_{k-\tau} \bigr\rangle +\frac{L}{2} \Vert x_k - x_{k-\tau}\Vert^2 \nonumber\\
& \leq f^\star + \bigl\langle \nabla f(x_{k-\tau}),\; x_k - x^\star \bigr\rangle  - \frac{1}{2L}\bigl\| \nabla f(x_{k-\tau}) \bigr\|^2+\frac{L}{2} \Vert x_k - x_{k-\tau}\Vert^2.
\label{Lemma_delayedgradient_2_p1}
\end{align}
We can bound the final term by
\begin{align}
\Vert x_k - x_{k-\tau}\Vert^2  \overset{\eqref{eqn:delayed_gd}}{=}
\left\Vert \sum_{\ell=(k-\tau)_+}^{k-1} \gamma\nabla f(x_{\ell-\tau})\right\Vert^2
\leq \gamma^2 \tau \sum_{\ell=(k-\tau)_+}^{k-1} \Vert \nabla f(x_{\ell-\tau})\Vert^2,
\label{Lemma_delayedgradient_2_p2}
\end{align}  
where the inequality follows from the convexity of $\Vert\cdot\Vert^2$. Substituting~\eqref{Lemma_delayedgradient_2_p2} into~\eqref{Lemma_delayedgradient_2_p1}, and rearranging the terms conclude the proof.
\end{proof}
\section{Proofs for Section~\ref{Sec:Main Results}}
\subsection{Proof of Lemma~\ref{Sec:Main Results Lemma2}}
\label{Sec:Main Results Lemma2 Proof}
The proof idea is inspired by the convergence theorem for totally asynchronous iterations~\cite[Proposition $6.2.1$]{BeT:89}. We first use perfect induction to prove that
\begin{align}
V_k \leq V_0, \quad k\in\mathbb{N}_0.
\label{Sec:Main Results Lemma2 Proof Eq1}
\end{align}
Since $V_0$ satisfies~\eqref{Sec:Main Results Lemma2 Proof Eq1}, the induction hypothesis is true for $k = 0$. Now, assume that~\eqref{Sec:Main Results Lemma2 Proof Eq1} holds for all $k$ up to some $K\in\mathbb{N}_0$. This means that $V_K\leq V_0$ and 
\begin{align*}
\max_{(K-\tau_K)_+ \leq \ell \leq K} V_{\ell} & \leq V_0.
\end{align*}
Thus, from~\eqref{Sec:Main Results Lemma2 Eq1}, we have
\begin{align*}
V_{K+1} & \leq (q + p) V_0\\
&\leq V_0,
\end{align*}
where the second inequality follows from the fact that $q + p \in (0,1)$. Therefore, $V_{K+1}\leq V_0$ and, hence, the induction proof is complete. Next, we prove that $V_k$ converges to zero as $k\rightarrow +\infty$. To this end, we use induction to show that for each $m\in \mathbb{N}_0$, there exists $K_m \in \mathbb{N}_0$ such that 
\begin{align}
V_k \leq (q + p)^m V_0, \quad \forall k\geq K_m. 
\label{Sec:Main Results Lemma2 Proof Eq2}
\end{align}
From~\eqref{Sec:Main Results Lemma2 Proof Eq1}, $V_k\leq (q + p)^0 V_0$ for all $k\geq 0$. Thus, the induction hypothesis holds for $m = 0$ (with $K_0 = 0$). Assuming~\eqref{Sec:Main Results Lemma2 Proof Eq2} is true
for a given $m$, we will show that there exists $K_{m+1} \in \mathbb{N}_0$ such that $V_k \leq (q + p)^{m+1} V_0$ for all $k\geq K_{m+1}$. From~\eqref{Sec:Main Results Lemma2 Eq2}, 
one can find a sufficiently large $\overline{K}_m\geq K_m$ such that $k-\tau_k\geq K_m$ for $k\geq \overline{K}_m$. Since~\eqref{Sec:Main Results Lemma2 Proof Eq2} holds by induction, we have
\begin{align*}
\max_{(k-\tau_k)_+ \leq \ell \leq k} V_{\ell} \leq  (q + p)^{m} V_0, \quad \forall k\geq \overline{K}_m.
\end{align*}
It follows from~\eqref{Sec:Main Results Lemma2 Eq1} that
\begin{align*}
V_{k+1} & \leq (q + p)(q + p)^{m} V_0\\
& = (q + p)^{m+1} V_0, \quad \forall k\geq \overline{K}_m,
\end{align*}
which implies that 
\begin{align*}
V_k \leq (q + p)^{m+1} V_0, \quad \forall k\geq \overline{K}_{m}+1.
\end{align*}
Set $K_{m+1} = \overline{K}_m + 1$. The induction proof is complete. In summary, we conclude that for each $m$, there exists $K_m$ such that $V_k \leq (q + p)^{m} V_0$ for all $k\geq K_{m}$. Since $q + p < 1$, $(q+p)^m$ approaches zero as $m\rightarrow +\infty$. Hence, the sequence $V_k$ asymptotically converges to zero.

\subsection{Proof of Lemma~\ref{Sec:Main Results Lemma3}}
\label{Sec:Main Results Lemma3 Proof}
We will show by induction that
\begin{align}
V_k \leq \Lambda(k) V_0, \quad k\in\mathbb{N}_0.
\label{Sec:Main Results Lemma3 Proof Eq1}
\end{align}
Since $\Lambda(0)=1$, the induction hypothesis is true for $k = 0$. Assume for induction that~\eqref{Sec:Main Results Lemma3 Proof Eq1} holds for all $k$ up to some $K$. From~\eqref{Sec:Main Results Lemma3 Eq1}, we have
\begin{align*}
V_{K+1} &\leq q \Lambda(K) V_0 + p \max_{(K-\tau_K)_+ \leq \ell \leq K} \Lambda(\ell) V_0\\
&\leq q \Lambda(K) V_0 + p \max_{K-\tau_K  \leq \ell \leq K} \Lambda(\ell) V_0.
\end{align*}
Since $\Lambda$ is non-increasing on $\mathbb{R}$, we obtain $V_{K+1}  \leq  \left(q + p \right) \Lambda(K - \tau_K) V_0$. It follows from~\eqref{Sec:Main Results Lemma3 Eq2} that $V_{K + 1} \leq \Lambda(K+1) V_0$. Therefore,~\eqref{Sec:Main Results Lemma3 Proof Eq1} holds for $k = K + 1$. The induction proof is complete.

\subsection{Proof of Lemma~\ref{Sec:Main Results Lemma4}}
\label{Sec:Main Results Lemma4 Proof}
Let us define $W_k=0$ for $k\in \{\ldots ,-2, -1\}$. Dividing both sides of~\eqref{Sec:Main Results Lemma4 Eq0} by $Q_{k+1}$, summing from $k=0$ to $K$, and then using telescoping cancellation, we have
\begin{align*}
\sum_{k=0}^{K} \frac{X_{k}}{Q_{k+1}} + \frac{V_{K+1}}{Q_{K+1}} & \leq V_0 + \sum_{k=0}^{K} \sum_{\ell = (k-\tau_k)_+}^{k} \frac{p_k W_{\ell}}{Q_{k+1}} - \sum_{k=0}^{K}\frac{r_k W_k }{Q_{k+1}} + \sum_{k=0}^{K} \frac{e_{k}}{Q_{k+1}}\\
&= V_0 +  \sum_{k=0}^{K} \sum_{\ell = k-\tau_k}^{k} \frac{p_k W_{\ell} }{Q_{k+1}} -  \sum_{k=0}^{K}\frac{r_k W_k }{Q_{k+1}} + \sum_{k=0}^{K} \frac{e_{k}}{Q_{k+1}}.
\end{align*}
Since $\tau_k\leq \tau$ for all $k\in\mathbb{N}_0$, we obtain
\begin{align}
\sum_{k=0}^{K} \frac{X_{k}}{Q_{k+1}} + \frac{V_{K+1}}{Q_{K+1}} &\leq V_0 +  \sum_{k=0}^{K} \sum_{\ell = k-\tau}^{k} \frac{p_k W_{\ell} }{Q_{k+1}} - \sum_{k=0}^{K}\frac{r_k W_k }{Q_{k+1}} + \sum_{k=0}^{K} \frac{e_{k}}{Q_{k+1}} \nonumber\\
&=V_0 +  \sum_{\ell = 0}^\tau \sum_{k = - \ell}^{K-\ell} \frac{p_{k+\ell} W_{k}}{Q_{k + \ell + 1}} - \sum_{k=0}^{K}\frac{r_k W_k }{Q_{k+1}} + \sum_{k=0}^{K} \frac{e_{k}}{Q_{k+1}} \nonumber\\
&\leq V_0 +  \sum_{\ell = 0}^\tau \sum_{k = 0}^{K} \frac{p_{k+\ell} W_{k}}{Q_{k+\ell+1}} - \sum_{k=0}^{K}\frac{r_k W_k }{Q_{k+1}} + \sum_{k=0}^{K} \frac{e_{k}}{Q_{k+1}} \nonumber\\
&= V_0 - \sum_{k = 0}^{K} \left(\frac{r_k}{Q_{k+1}} - \sum_{\ell=0}^\tau \frac{p_{k+\ell}}{Q_{k+\ell+1}}\right)W_k + \sum_{k=0}^{K} \frac{e_{k}}{Q_{k+1}},
\label{Sec:Main Results Lemma4 Proof Eq1}
\end{align}
where the second inequality comes from our assumption that $W_k\geq 0$ for $k\geq 0$ and $W_k=0$ for $k<0$. We are now ready to prove Lemma~\ref{Sec:Main Results Lemma4}.

$1.$ In this case, $q_k=1$ and, hence, $Q_k=1$ for each $k$. Thus,~\eqref{Sec:Main Results Lemma4 Proof Eq1} simplifies to
\begin{align*}
\sum_{k=0}^{K} X_{k} + V_{K+1} \leq V_0 -  \sum_{k = 0}^{K} \left( r_k - \sum_{\ell = 0}^\tau p_{k+\ell}\right)W_k + \sum_{k=0}^{K} e_{k}.
\end{align*}
The assumption that 
\begin{align*}
 \sum_{\ell=0}^\tau p_{k+\ell} \leq r_k
\end{align*}
holds for every $k\in\mathbb{N}_0$ implies that 
\begin{align*}
\sum_{k=0}^{K} X_{k} + V_{K+1} \leq V_0 + \sum_{k=0}^{K} e_{k}.
\end{align*}
Since $\{X_k\}$ and $\{V_k\}$ are non-negative sequences, it follows that 
\begin{align*}
\sum_{k=0}^{K} X_{k} \leq V_0 + \sum_{k=0}^{K} e_{k}\;\;\textup{and}\;\; V_{K+1}\leq V_0 + \sum_{k=0}^{K} e_k
\end{align*}
hold for $K\in\mathbb{N}_0$.

$2.$ In this case, $p_k = p$ and $r_k=r$. Thus,~\eqref{Sec:Main Results Lemma4 Proof Eq1} simplifies to
\begin{align}
\sum_{k=0}^{K} \frac{X_{k}}{Q_{k+1}} + \frac{V_{K+1}}{Q_{K+1}} &\leq  V_0 - \sum_{k = 0}^{K} \left(\frac{r}{Q_{k+1}} - \sum_{\ell=0}^\tau \frac{p}{Q_{k+\ell+1}}\right)W_k + \sum_{k=0}^{K} \frac{e_{k}}{Q_{k+1}}.
\label{Sec:Main Results Lemma4 Proof Eq2}
\end{align}
For each $k\in\mathbb{N}_0$ and $\ell\in[0, \tau]$, we have
\begin{align*}
Q_{k+\ell +1} &= Q_{k+1} \prod_{s = 0}^{\ell-1} q_{k + s + 1}\\
&\geq Q_{k+1} q^\ell
\end{align*}
where the inequality uses $q_k \geq q$ for $k\in\mathbb{N}_0$. Since
$q \in (0, 1]$, we have $q^\ell\geq q^\tau$ for $\ell\in[0, \tau]$, implying that $Q_{k+\ell +1}\geq  Q_{k+1} q^\tau$. Combining this inequality with~\eqref{Sec:Main Results Lemma4 Proof Eq2} yields
\begin{align}
\sum_{k=0}^{K} \frac{X_{k}}{Q_{k+1}} + \frac{V_{K+1}}{Q_{K+1}} &\leq  V_0 - \left(r - p (\tau + 1) q^{-\tau}  \right)\sum_{k = 0}^{K} \frac{W_k}{Q_{k+1}} + \sum_{k=0}^{K} \frac{e_{k}}{Q_{k+1}}.
\label{Sec:Main Results Lemma4 Proof Eq3}
\end{align}
By Bernoulli's inequality, i.e., $(1 + x)^n\geq 1 + n x$ for any $n\in\mathbb{N}_0$ and any $x> -1$, we have
\begin{align}
q^\tau &= (1-(1-q))^\tau \nonumber \\
&\geq 1- (1-q)\tau.
\label{Sec:Main Results Lemma4 Proof Eq4}
\end{align}
The assumption that $2\tau +1 \leq 1/(1-q)$ implies that $(1-q) \leq 1/(2\tau+1)$ and thereby
\begin{align*}
 1 - (1-q)\tau \geq \frac{\tau+1}{2\tau+1}.
\end{align*}
By~\eqref{Sec:Main Results Lemma4 Proof Eq4}, it therefore holds that
\begin{align*}
q^\tau \geq \frac{\tau+1}{2\tau+1},
\end{align*}
or equivalently, $q^{-\tau}(\tau+1)\leq (2\tau+1)$. From~\eqref{Sec:Main Results Lemma4 Proof Eq3}, we then have
\begin{align*}
\sum_{k=0}^{K} \frac{X_{k}}{Q_{k+1}} + \frac{V_{K+1}}{Q_{K+1}} &\leq  V_0 - \left(r - (2\tau+1)p\right) \sum_{k = 0}^{K-1} \frac{W_k }{Q_{k+1}} + \sum_{k=0}^{K} \frac{e_{k}}{Q_{k+1}}.
\end{align*}
The assumption that $2\tau +1 \leq r/p$ allows us to drop the second term on the right-hand side. Thus,
\begin{align*}
\sum_{k=0}^{K} \frac{X_{k}}{Q_{k+1}} + \frac{V_{K+1}}{Q_{K+1}} &\leq  V_0 + \sum_{k=0}^{K} \frac{e_{k}}{Q_{k+1}}, \quad K\in\mathbb{N},
\end{align*}
which concludes the proof.

%
%
\section{Proofs for Subsection~\ref{Sec:PIAG}}
\label{Appendix_PIAG_Proofs}
This section provides the proofs for the results presented in Subsection~\ref{Sec:PIAG}. We first state two key lemmas which establish important recursions for the iterates generated by \textsc{Piag}.

\begin{lemma}
\label{Appendix:PIAG Lemma2}
Suppose Assumptions~\ref{Sec:PIAG Assumption1} and~\ref{Sec:PIAG Assumption2} hold. Let $\{x_k\}$ be the sequence generated by Algorithm~\ref{Sec:PIAG Algorithm2}. Then, for any $x^\star \in \mathcal{X}^\star$ and every $k\in\mathbb{N}_0$, we have
\begin{align}
P(x_{k+1})-P^\star + \frac{1}{2\gamma}\|x_{k+1} - x^\star\|^2  \leq \frac{1}{2\gamma} \|x_k - x^\star\|^2 + \frac{L\bigl(\tau+1\bigr)}{2} \sum_{\ell = (k-\tau)_+}^{k}\|x_{\ell + 1} - x_{\ell}\|^2 - \frac{1}{2\gamma}\|x_{k+1}-x_k\|^2.
\label{Appendix:PIAG Lemma2 Eq1}
\end{align}
\end{lemma}

\begin{proof}
At iteration $k\in\mathbb{N}_0$, the update rule in Algorithm~\ref{Sec:PIAG Algorithm2} is of the form
\begin{align*}
\begin{split}
g_k &= \frac{1}{n}\sum_{i=1}^n \nabla f_i\bigl(x_{s_{i,k}}\bigr),\\
x_{k+1} & = \textup{prox}_{\gamma R}(x_k - \gamma g_k).
\end{split}
\end{align*}
We write the proximal mapping more explicitly as
\begin{align}
x_{k+1} =  \underset{u\in \mathbb{R}^d}{\textup{argmin}} \left\{\frac{1}{2}\bigl\|u - (x_k - \gamma g_k)\bigr\|^2 + \gamma R(u)\right\}.
\label{Appendix:PIAG Lemma2 Proof Eq1}
\end{align}
From the first-order optimality condition for the point $u$ in the minimization problem~\eqref{Appendix:PIAG Lemma2 Proof Eq1}, there is a sub-gradient  $\xi_{k+1} \in \partial R(x_{k+1})$ such that 
\begin{align}
\bigl\langle x_{k+1} - x_k +\gamma(g_k + \xi_{k+1}),\; x^\star - x_{k+1} \bigr\rangle \geq 0.
\label{Appendix:PIAG Lemma2 Proof Eq1-2}
\end{align}
Since $R$ is convex, we have
\begin{align}
R(x_{k+1}) & \leq  R(x^\star) + \bigl \langle \xi_{k+1},\; x_{k+1} - x^\star\bigr\rangle\nonumber\\
& \overset{\eqref{Appendix:PIAG Lemma2 Proof Eq1-2}}{\leq} R(x^\star) + \bigl \langle \frac{x_k - x_{k+1}}{\gamma} - g_k,\; x_{k+1} - x^\star\bigr\rangle\nonumber\\
& = R(x^\star) + \frac{1}{2\gamma}\left(\|x_k - x^\star\|^2  - \|x_{k+1}- x_k\|^2 - \|x_{k+1} - x^\star\|^2\right)- \bigl \langle g_k,\; x_{k+1} - x^\star\bigr\rangle,
\label{Appendix:PIAG Lemma2 Proof Eq2}
\end{align}
where the last equality follows from the fact 
\begin{align*}
2\langle a-b, b-c\rangle = \|a-c\|^2 - \|a-b\|^2 - \|b-c\|^2,\quad a,b,c\in\mathbb{R}^d.
\end{align*}
By Assumption~\ref{Sec:PIAG Assumption2}, each component function $f_i$ is $L_i$-smooth. Thus, from Lemma $1.2.3$ in~\cite{Nesterov:13}, we have
\begin{align*}
f_i(x_{k+1}) &\leq f_i(x_{s_{i,k}}) +  \bigl\langle \nabla f_i(x_{s_{i,k}}), \;x_{k+1} - x_{s_{i,k}} \bigr\rangle  + \frac{L_i}{2}\bigl\| x_{k+1} - x_{s_{i,k}} \bigr\|^2\\
&\leq f_i(x^\star) +  \bigl\langle \nabla f_i(x_{s_{i,k}}), \;x_{k+1} - x^\star \bigr\rangle  + \frac{L_i}{2}\bigl\| x_{k+1} - x_{s_{i,k}} \bigr\|^2,
\end{align*}
where the second inequality follows from the convexity of $f_i$. Dividing both sides of the above inequality by~$n$ and then summing from $i = 1$ to $n$, we obtain
\begin{align}
F(x_{k+1}) \leq F(x^\star) +  \left\langle g_k , \;x_{k+1} - x^\star \right\rangle  + \sum_{i=1}^n\frac{L_i}{2n}\bigl\| x_{k+1} - x_{s_{i,k}} \bigr\|^2.
\label{Appendix:PIAG Lemma2 Proof Eq3}
\end{align}
By adding inequalities~\eqref{Appendix:PIAG Lemma2 Proof Eq2} and~\eqref{Appendix:PIAG Lemma2 Proof Eq3}, rearranging the terms, and recalling that $P(x) = F(x) + R(x)$, we have
\begin{align}
P(x_{k+1}) - P^\star + \frac{1}{2\gamma}\|x_{k+1} - x^\star\|^2 \leq \frac{1}{2\gamma}\|x_k - x^\star\|^2  +  \underbrace{\sum_{i=1}^n\frac{L_i}{2n}\bigl\| x_{k+1} -  x_{s_{i,k}} \bigr\|^2}_{\mathcal{H}}  -  \frac{1}{2\gamma}\|x_{k+1} - x_k\|^2.
\label{Appendix:PIAG Lemma2 Proof Eq4}
\end{align}
Next, we find an upper bound on the term $\mathcal{H}$. We expand $\mathcal{H}$ as follows:
\begin{align*}
\mathcal{H} =\sum_{i=1}^n \frac{L_i}{2n} \left \|\sum_{\ell = s_{i,k}}^{k}x_{\ell + 1} - x_{\ell}\right\|^2 
= \sum_{i=1}^n \frac{L_i\bigl(k-s_{i,k}+1\bigr)^2}{2n} \left \|\sum_{\ell = s_{i,k}}^{k}\frac{x_{\ell + 1} - x_{\ell}}{k-s_{i,k}+1}\right\|^2.
\end{align*}
The squared Euclidean norm ($\|\cdot\|^2$) is convex. Thus, 
\begin{align*}
\mathcal{H} &\leq \sum_{i=1}^n \frac{L_i\bigl(k-s_{i,k}+1\bigr)}{2n} \sum_{\ell = s_{i,k}}^{k}\|x_{\ell + 1} - x_{\ell}\|^2.
\end{align*}
Since $(k-\tau)_+ \leq s_{i,k} \leq k$ for all $i \in [n]$ and $k\in\mathbb{N}_0$, we have 
\begin{align}
\mathcal{H} &\leq \sum_{i=1}^n \frac{L_i\bigl(\tau+1\bigr)}{2n} \sum_{\ell=(k-\tau)_+}^{k}\|x_{\ell + 1} - x_{\ell}\|^2\nonumber\\
& \overset{\eqref{LipschitzConstant}}{=} \frac{L\bigl(\tau + 1\bigr)}{2} \sum_{\ell=(k-\tau)_+}^{k}\|x_{\ell+1} - x_{\ell}\|^2.
\label{Appendix:PIAG Lemma2 Proof Eq5}
\end{align}
Substituting~\eqref{Appendix:PIAG Lemma2 Proof Eq5} into the bound~\eqref{Appendix:PIAG Lemma2 Proof Eq4} concludes the proof.
\end{proof}

\begin{lemma}
\label{Appendix:PIAG Lemma3}
Suppose Assumptions~\ref{Sec:PIAG Assumption1} and~\ref{Sec:PIAG Assumption2} hold. Let $\{\alpha_k\}$ be a sequence of non-negative numbers. Then, for every $k\in\mathbb{N}_0$, the sequence $\{x_k\}$ generated by Algorithm~\ref{Sec:PIAG Algorithm2} satisfies
\begin{align}
\alpha_k\bigl(P(x_{k+1})-P^\star\bigr) \leq \alpha_k\bigl(P(x_{k})-P^\star\bigr)  + \frac{L \alpha_k}{2} \sum_{\ell = (k-\tau)_+}^{k}\|x_{\ell + 1} -x_{\ell}\|^2 - \alpha_k\left(\frac{1}{\gamma} - \frac{L\tau}{2}\right)\|x_{k+1} -x_k\|^2.
\label{Appendix:PIAG Lemma3 Eq1}
\end{align}
\end{lemma}

\begin{proof}
According to Assumption~\ref{Sec:PIAG Assumption2}, each component function $f_i$, $i\in [n]$, is $L_i$-smooth. It follows from second inequality of $(3.5)$ in Lemma $3.4$ in~\cite{Vanli:18} that
\begin{align*}
P(x_{k+1})-P^\star &\leq P(x_{k})-P^\star  + \frac{L}{2} \sum_{\ell=(k-\tau)_+}^{k-1}\|x_{\ell+1} -x_{\ell}\|^2 - \left(\frac{1}{\gamma} - \frac{L(\tau+1)}{2}\right)\|x_{k+1} -x_k\|^2 \\
&= P(x_{k})-P^\star  + \frac{L}{2} \sum_{\ell=(k-\tau)_+}^{k}\|x_{\ell+1} -x_{\ell}\|^2 - \left(\frac{1}{\gamma} - \frac{L\tau}{2}\right)\|x_{k+1} -x_k\|^2.
\end{align*}
Multiplying both sides of the above inequality by the non-negative number $\alpha_k$ proves the lemma.
\end{proof}

\subsection{Proof of Lemma~\ref{Sec:PIAG Lemma1}}
\label{Sec:PIAG Lemma1 Proof}
Since $\alpha_k = k +\alpha_0$ with $\alpha_0 \geq 0$, the sequence $\{\alpha_k\}$ is non-negative and satisfies $\alpha_{k+1} = \alpha_k +1$ for every $k\in\mathbb{N}_0$. Adding inequalities~\eqref{Appendix:PIAG Lemma2 Eq1} and~\eqref{Appendix:PIAG Lemma3 Eq1}, 
and rearranging the terms, we have
\begin{align}
\alpha_{k+1}\bigl(P(x_{k+1})-P^\star\bigr) + \frac{1}{2\gamma}\|x_{k+1} - x^\star\|^2  &\leq \alpha_k\bigl(P(x_{k})-P^\star\bigr)  + \frac{1}{2\gamma} \|x_k - x^\star\|^2 \nonumber\\
&\hspace{5mm} + \frac{L\bigl(\alpha_k + \tau+1\bigr)}{2} \sum_{\ell=(k-\tau)_+}^{k}\|x_{\ell+1} - x_{\ell}\|^2\nonumber\\
&\hspace{5mm} - \frac{1}{2}\left(\frac{2\alpha_k +1}{\gamma} -   \alpha_k L\tau\right)\|x_{k+1}-x_k\|^2.
\label{Appendix:PIAG Lemma1 Proof Eq1}
\end{align}
Multiplying both sides by $2\gamma$ and then letting $V_k = 2\gamma\alpha_k\bigl(P(x_{k})-P^\star\bigr)  + \|x_k - x^\star\|^2$ and $W_k = \|x_{k+1}-x_k\|^2$, we can rewrite~\eqref{Appendix:PIAG Lemma1 Proof Eq1} as
\begin{align*}
V_{k+1} &\leq V_k + \gamma L\bigl(\alpha_k + \tau+1\bigr)\sum_{\ell=(k-\tau)_+}^{k} W_{\ell} - \bigl(2\alpha_k +1 - \gamma \alpha_k L\tau \bigr) W_k,\quad k\in\mathbb{N}_0.
\end{align*}
The proof is complete.

\subsection{Proof of Theorem~\ref{Sec:PIAG Theorem1}}
\label{Sec:PIAG Theorem1 Proof}
According to Lemma~\ref{Sec:PIAG Lemma1}, the iterates generated by Algorithm~\ref{Sec:PIAG Algorithm2} satisfy
\begin{align*}
V_{k+1} &\leq V_k + p_k \sum_{\ell=(k-\tau)_+}^{k} W_{\ell} -  r_k W_k,\quad k\in\mathbb{N}_0,
\end{align*}
where $p_k = \gamma L\bigl(\alpha_k + \tau+1\bigr)$, $r_k =  2\alpha_k +1- \gamma L\tau  \alpha_k$, and $\alpha_k = k + \alpha_0$. To apply Lemma~\ref{Sec:Main Results Lemma4}, 
we need to enforce that the convergence condition
\begin{align*}
\sum_{\ell=0}^\tau  \gamma L\bigl(\alpha_{\ell+k} + \tau+1\bigr) \leq  2\alpha_k +1 - \gamma L\tau \alpha_k
\end{align*}
is satisfied for every $k\in\mathbb{N}_0$. This inequality is equivalent to
\begin{align}
\gamma L \leq \frac{2\alpha_k +1}{\tau \alpha_k + \sum_{\ell=0}^\tau \bigl(\alpha_{\ell+k} +\tau+1\bigr)}.
\label{Appendix:PIAG Theorem1 Proof Eq1}
\end{align}
We will prove that if $\alpha_0 = \tau$ and $\gamma L (2\tau + 1)\leq 1$, then~\eqref{Appendix:PIAG Theorem1 Proof Eq1} holds for all $k\in\mathbb{N}_0$. Replacing $\alpha_{k} = k + \tau$ in~\eqref{Appendix:PIAG Theorem1 Proof Eq1}, we have
\begin{align*}
\frac{2\alpha_k +1}{\tau \alpha_k + \sum_{\ell=0}^\tau \bigl(\alpha_{\ell+k} +\tau+1\bigr)} & =  \frac{2 k + 2\tau +1}{(2\tau+1)k +  \frac{7}{2}\tau^2 + \frac{7}{2}\tau + 1}\\
& = \frac{1}{2\tau +1} +  \frac{(2\tau+1)k +  \frac{1}{2}\tau(\tau +1)}{\left(2\tau +1\right)\left((2\tau+1)k +  \frac{7}{2}\tau^2 + \frac{7}{2}\tau + 1\right)}.
\end{align*}
The second term on the right-hand side is non-negative for any $\tau\in \mathbb{N}_0$ and $k\in \mathbb{N}_0$. Thus, 
\begin{align*}
 \frac{1}{2\tau +1} \leq \frac{2\alpha_k +1}{\tau \alpha_k + \sum_{\ell=0}^\tau \bigl(\alpha_{\ell+k} +\tau+1\bigr)}, \quad k\in\mathbb{N}_0.
\end{align*}
This shows that if 
\begin{align*}
\gamma L \leq \frac{1}{2\tau + 1},
\end{align*}
then~\eqref{Appendix:PIAG Theorem1 Proof Eq1} holds for $k\in\mathbb{N}_0$, and hence the convergence condition~\eqref{Sec:Main Results Lemma4 Eq1} in  Lemma~\ref{Sec:Main Results Lemma4} is satisfied. It follows from part $1$ of Lemma~\ref{Sec:Main Results Lemma4} that $V_k \leq V_0$ for  $k\in\mathbb{N}$. Recalling that $V_k = 2\gamma\alpha_k\bigl(P(x_{k})-P^\star\bigr)  + \|x_k - x^\star\|^2$, we obtain
\begin{align*}
2\gamma(k+\tau)\bigl(P(x_{k})-P^\star\bigr)  + \|x_k - x^\star\|^2 \leq 2\gamma \tau \bigl(P(x_{0})-P^\star\bigr)  + \|x_0 - x^\star\|^2, \quad k\in\mathbb{N}.
\end{align*}
By dropping the second term on the left-hand side and dividing both sides by $2\gamma(k+\tau)$, we finish the proof.

\subsection{Proof of Lemma~\ref{Sec:PIAG Lemma2}}
\label{Sec:PIAG Lemma2 Proof}
According to Lemma~\ref{Appendix:PIAG Lemma2} in Appendix~\ref{Appendix_PIAG_Proofs}, the relation~\eqref{Appendix:PIAG Lemma2 Eq1} holds for every $x^\star \in \mathcal{X}^\star$. Taking $x^\star = \Pi_{\mathcal{X}^\star}(x_k)$ in~\eqref{Appendix:PIAG Lemma2 Eq1}, we get
\begin{align}
P(x_{k+1})-P^\star + \frac{1}{2\gamma}\|x_{k+1} - \Pi_{\mathcal{X}^\star}(x_k)\|^2  &\leq \frac{1}{2\gamma} \|x_k -\Pi_{\mathcal{X}^\star}(x_k)\|^2 \nonumber\\
&\hspace{5mm} + \frac{L\bigl(\tau+1\bigr)}{2} \sum_{\ell = (k-\tau)_+}^{k}\|x_{\ell+1} - x_{\ell}\|^2 - \frac{1}{2\gamma}\|x_{k+1}-x_k\|^2.
\label{PIAG Lemma2 Proof Eq1}
\end{align}
By the projection property, we have $\|x_{k+1} - \Pi_{\mathcal{X}^\star}(x_{k+1})\| \leq \|x_{k+1} - \Pi_{\mathcal{X}^\star}(x_k)\|$. Combining this inequality with~\eqref{PIAG Lemma2 Proof Eq1} yields
\begin{align}
P(x_{k+1})-P^\star + \frac{1}{2\gamma}\|x_{k+1} - \Pi_{\mathcal{X}^\star}(x_{k+1})\|^2  &\leq \frac{1}{2\gamma} \|x_k -\Pi_{\mathcal{X}^\star}(x_k)\|^2 \nonumber \\
&\hspace{5mm}+ \frac{L\bigl(\tau+1\bigr)}{2} \sum_{\ell = (k-\tau)_+}^{k}\|x_{\ell+1} - x_{\ell}\|^2 - \frac{1}{2\gamma}\|x_{k+1}-x_k\|^2.
\label{PIAG Lemma2 Proof Eq2}
\end{align}
Adding inequalities~\eqref{Appendix:PIAG Lemma3 Eq1} and~\eqref{PIAG Lemma2 Proof Eq2}, and setting $\alpha_k = \alpha$ for some $\alpha>0$, we have
\begin{align}
(1+\alpha)\bigl(P(x_{k+1})-P^\star\bigr) + \frac{1}{2\gamma}\|x_{k+1} - \Pi_{\mathcal{X}^\star}(x_{k+1})\|^2  &\leq \alpha\bigl(P(x_{k})-P^\star\bigr)  + \frac{1}{2\gamma} \|x_k -  \Pi_{\mathcal{X}^\star}(x_{k})\|^2 \nonumber\\
& \hspace{5mm} + \frac{L\bigl(\alpha + \tau+1\bigr)}{2} \sum_{\ell=(k-\tau)_+}^{k}\|x_{\ell+1} - x_{\ell}\|^2\nonumber\\
&\hspace{5mm}  - \frac{1}{2}\left(\frac{2\alpha +1}{\gamma} -  \alpha L\tau\right)\|x_{k+1}-x_k\|^2.
\label{PIAG Lemma2 Proof Eq3}
\end{align}
Let $\theta = \frac{Q}{Q+1}$. Note that $\theta \in (0,1)$. It follows from Assumption~\ref{Sec:PIAG Assumption3} that
\begin{align*}
(1+ \alpha) \bigl(P(x_{k+1}) - P^\star\bigr) &=  \theta \bigl(P(x_{k+1}) - P^\star\bigr) + \left(1 - \theta + \alpha\right)\bigl(P(x_{k+1}) - P^\star\bigr)\\
&\geq \frac{\mu \theta}{2}\|x_{k+1} -\Pi_{\mathcal{X}^\star}(x_{k+1})\|^2  + \left(1 - \theta + \alpha\right)\bigl(P(x_{k+1}) - P^\star\bigr).
\end{align*}
Combining the above inequality with~\eqref{PIAG Lemma2 Proof Eq3} and then multiplying both sides by $2\gamma$, we obtain
\begin{align}
2\gamma (1- \theta +\alpha)\bigl(P(x_{k+1})-P^\star\bigr) + (1+\gamma \mu \theta) \|x_{k+1} - \Pi_{\mathcal{X}^\star}(x_{k+1})\|^2 & \leq 2\alpha\gamma\bigl(P(x_{k})-P^\star\bigr)  + \|x_k - \Pi_{\mathcal{X}^\star}(x_{k})\|^2\nonumber\\
& \hspace{5mm} + \gamma L\bigl(\alpha + \tau+1\bigr)\sum_{\ell=(k-\tau)_+}^{k}\|x_{\ell+1} - x_{\ell}\|^2\nonumber\\
& \hspace{5mm} - \left(2\alpha +1 - \gamma  \alpha L\tau \right)\|x_{k+1}-x_k\|^2.
\label{PIAG Lemma2 Proof Eq4}
\end{align}
By letting $\alpha = \frac{1}{\gamma L}$, $V_k = \frac{2}{L}\left( P(x_k)-P^\star\right) + \|x_k -  \Pi_{\mathcal{X}^\star}(x_{k})\|^2$ and $W_k = \|x_{k+1} - x_k\|^2$, the inequality~\eqref{PIAG Lemma2 Proof Eq4} can be rewritten as
\begin{align*}
(1+\gamma \mu \theta) V_{k+1} \leq V_k + \bigl(1+ \gamma L(\tau+1)\bigr) \sum_{\ell=(k-\tau)_+}^{k} W_{\ell} - \left(\frac{2}{\gamma L} + 1 -  \tau\right) W_k.
\end{align*}
Dividing both sides by $1 + \gamma \mu \theta$ completes the proof.

\subsection{Proof of Theorem~\ref{Sec:PIAG Theorem2}}
\label{Sec:PIAG Theorem2 Proof}
According to Lemma~\ref{Sec:PIAG Lemma2}, the iterates generated by Algorithm~\ref{Sec:PIAG Algorithm2} satisfy 
\begin{align*}
V_{k+1} &\leq q V_k + p \sum_{\ell=(k-\tau)_+}^{k} W_{\ell} -  r W_k
\end{align*}
for every  $k\in\mathbb{N}_0$, where
\begin{align*}
q= \frac{1}{1+\gamma \mu \theta}, \quad p =   \frac{1+ \gamma L\bigl(\tau+1\bigr)}{1+\gamma \mu \theta} \quad \textup{and} \quad r =  \frac{\frac{2}{\gamma L} + 1 -  \tau}{1+\gamma \mu \theta}.
\end{align*}
To apply Lemma~\ref{Sec:Main Results Lemma4}, we need to ensure that
\begin{align*}
2\tau + 1\leq \min\left\{\frac{1}{\gamma \mu \theta} +1, \frac{\frac{2}{\gamma L} + 1 -  \tau}{1+ \gamma L\bigl(\tau+1\bigr)}\right\}.
\end{align*}
This convergence condition is equivalent to
\begin{align}
\begin{cases}
2 \gamma \mu \theta \tau &\leq 1,\\
\gamma^2 L^2 (2\tau+1)(\tau+1) + 3\gamma L \tau & \leq  2.
\end{cases}
\label{Appendix:PIAG Theorem2 Proof Eq1}
\end{align}
Define $h=\gamma L (2\tau+1)$. The inequalities~\eqref{Appendix:PIAG Theorem2 Proof Eq1} can be rewritten in terms of $h$, $Q$ (recall that $\theta =  \frac{Q}{Q+1}$), and $\tau$ as
\begin{align}
\begin{cases}
2 h \tau &\leq (Q+1)(2\tau+1),\\
h^2(\tau+1) + 3 h\tau & \leq  2(2\tau+1).
\end{cases}
\label{Appendix:PIAG Theorem2 Proof Eq2}
\end{align}
For any fixed $\tau$, the left-hand side of inequalities~\eqref{Appendix:PIAG Theorem2 Proof Eq2} is non-decreasing in $h \geq 0$ and smaller than the right-hand side for $h=1$. Thus, the inequalities~\eqref{Appendix:PIAG Theorem2 Proof Eq2} hold for any $h\in[0,1]$. This shows that if $\gamma$ is set to
\begin{align*}
\gamma = \frac{h}{L(2\tau + 1)}, \quad h\in (0,1],
\end{align*}
then the convergence condition~\eqref{Appendix:PIAG Theorem2 Proof Eq1} is satisfied. Therefore, by part $2$ of Lemma~\ref{Sec:Main Results Lemma4}, $V_k \leq q^k V_0$ for $k\in\mathbb{N}$. Since $V_k = \frac{2}{L}\left( P(x_k)-P^\star\right) + \|x_k - \Pi_{\mathcal{X}^\star}(x_{k})\|^2$, it follows that
\begin{align*}
\frac{2}{L}\left( P(x_k)-P^\star\right) + \|x_k -  \Pi_{\mathcal{X}^\star}(x_{k})\|^2 &\leq \left(\frac{1}{1+\gamma \mu \theta}\right)^{k}\left(\frac{2}{L} \bigl(P(x_{0})-P^\star\bigr)  + \|x_0 -  \Pi_{\mathcal{X}^\star}(x_{0})\|^2\right)\\
& = \left(1 - \frac{1}{1 + (Q+1)(2\tau+1)/h}\right)^{k}\left(\frac{2}{L} \bigl(P(x_{0})-P^\star\bigr)  + \|x_0 -  \Pi_{\mathcal{X}^\star}(x_{0})\|^2\right).
\end{align*}
%

%
%
\section{Proofs for Subsection~\ref{Sec:SGD}}
\label{Appendix:SGD}
In this section, we provide the proofs for the results presented in Subsection~\ref{Sec:SGD}. We denote by $\mathbb{E}$ the expectation with respect
to the underlying probability space and by $\mathbb{E}_{k}$ the conditional expectation with respect to the past, that
is, with respect to $\{\xi_0,\ldots, \xi_{k-1}\}$. 

We first state a useful lemma which helps us to determine admissible step-sizes for the asynchronous \textsc{Sgd} method.

\begin{lemma}
\label{SGD Proof Lemma1}
Assume that $\alpha$ and $\beta$ are non-negative real numbers. If $\alpha(\beta + 1)\leq 1$, then 
\begin{align*}
\alpha^2 \beta^2 + \alpha \leq 1.
\end{align*}
\end{lemma}

\begin{proof}
If $\alpha(\beta + 1)\leq1$ holds, then $\alpha \leq 1/(\beta + 1)$ and $\alpha^2 \leq 1/(\beta + 1)^2$. Thus,
\begin{align*}
\alpha^2 \beta^2 + \alpha &\leq \frac{\beta^2}{\left(\beta +1\right)^2} + \frac{1}{\beta + 1}\\
& = \frac{\beta^2 + \beta + 1}{(\beta + 1)^2}.
\end{align*}
Since $\beta\geq 0$, one can verify that $\beta^2 + \beta +1 \leq (\beta + 1)^2$, implying that the right-hand side of the above inequality is less than or equal to $1$. Therefore, $\alpha^2 \beta^2 + \alpha \leq 1$.
\end{proof}

\subsection{Proof of Lemma~\ref{Sec:SGD Lemma1}}
\label{Sec:SGD Lemma1 Proof}
When $\gamma_k = 0$, $x_{k+1}=x_k$. This implies that $\|x_{k+1} - x^\star\|^2 = \|x_{k} - x^\star\|^2$, and hence $V_{k+1} = V_k$.  Below, we assume $\gamma_k>0$.

By subtracting $x^\star$ from
both sides of~\eqref{Sec:SGD Eq2} and then taking norm squares, we have
\begin{align}
\|x_{k+1} - x^\star\|^2 = \|x_{k} - x^\star\|^2 - 2\gamma_k \left\langle x_{k} - x^\star,  \nabla f(x_{k-\tau_k}, \xi_k)\right\rangle +  \gamma_k^2\left\|\nabla f(x_{k-\tau_k}, \xi_k)\right\|^2.
\label{ProofSGD0}
\end{align}
Taking conditional expectation on both sides of~\eqref{ProofSGD0} with respect to $\{\xi_0,\ldots, \xi_{k-1}\}$ leads to
\begin{align*}
\mathbb{E}_k\bigl[\|x_{k+1} - x^\star\|^2\bigr] & = \|x_{k} - x^\star\|^2 -  2 \gamma_k \bigl\langle x_{k} - x^\star,  \nabla F(x_{k-\tau_k})\bigr\rangle + \gamma_k^2 \mathbb{E}_k\bigl[\|\nabla f(x_{k-\tau_k}, \xi_k)\|^2\bigr] \nonumber\\
& = \|x_{k} - x^\star\|^2 -  2 \gamma_k \bigl\langle x_{k} - x^\star,  \nabla F(x_{k-\tau_k})\bigr\rangle \\
&\hspace{1cm}+ \gamma_k^2 \mathbb{E}_k\bigl[\|\nabla f(x_{k-\tau_k}, \xi_k) - \nabla F(x_{k-\tau_k})\|^2\bigr] + \gamma_k^2 \|\nabla F(x_{k-\tau_k})\|^2,
\end{align*}
where the second equality uses the fact that $\mathbb{E}[\|x\bigr\|^2] = \mathbb{E}[\|x- \mathbb{E}[x]\|^2] + \|\mathbb{E}[x]\|^2$ for any random vector~$x \in \mathbb{R}^d$. According to Assumption~\ref{Sec:SGD Assumption2}, stochastic gradients have bounded variance. Thus, 
\small
\begin{align}
\mathbb{E}_k\bigl[\|x_{k+1} - x^\star\|^2\bigr] \leq \|x_{k} - x^\star\|^2 -  2 \gamma_k \bigl\langle x_{k} - x^\star,  \nabla F(x_{k-\tau_k})\bigr\rangle+ \gamma_k^2 \sigma^2 + \gamma_k^2 \|\nabla F(x_{k-\tau_k})\|^2.
\label{ProofSGD1}
\end{align}
\normalsize
\textbf{1. Convex case}

Since $F$ is convex and $L$-smooth, it follows from Theorem $2.1.5$ in~\cite{Nesterov:13} that 
\begin{align}
F(x_{k-\tau_k}) +  \bigl\langle x^\star - x_{k-\tau_k},\;\nabla F(x_{k-\tau_k}) \bigr\rangle  + \frac{1}{2L}\bigl\| \nabla F(x_{k-\tau_k}) \bigr\|^2 \leq F^\star.
\label{ProofSGD2}
\end{align}
Moreover, from Lemma $1.2.3$ in~\cite{Nesterov:13}, we have
\begin{align}
F(x_{k}) &\leq F(x_{k-\tau_k}) +  \bigl\langle x_{k} - x_{k-\tau_k}, \; \nabla F(x_{k-\tau_k}) \bigr\rangle  + \frac{L}{2}\bigl\|x_{k} - x_{k-\tau_k} \bigr\|^2.
\label{ProofSGD3}
\end{align}
By adding inequalities~\eqref{ProofSGD2} and~\eqref{ProofSGD3} and then rearranging the terms, we get
\begin{align*}
F(x_{k}) - F^\star &\leq \bigl\langle x_{k} - x^\star,\; \nabla F(x_{k-\tau_k}) \bigr\rangle  + \frac{L}{2}\bigl\|x_{k} - x_{k-\tau_k} \bigr\|^2 - \frac{1}{2L}\bigl\| \nabla F(x_{k-\tau_k}) \bigr\|^2.
\end{align*}
Combining this inequality
with~\eqref{ProofSGD1} and taking the full expectation on both sides, we obtain
\begin{align}
2\gamma_k \mathbb{E}\bigl[F(x_{k}) - F^\star\bigr] + \mathbb{E}\bigl[\|x_{k+1} - x^\star\|^2\bigr]&\leq \mathbb{E}\bigl[\|x_{k} - x^\star\|^2\bigr] + \gamma_k L \underbrace{\mathbb{E}\bigl[\|x_{k} - x_{k-\tau_k} \bigr\|^2\bigr]}_{\mathcal{H}} \nonumber\\
&\hspace{1cm}- \left(\frac{\gamma_k}{L} - \gamma_k^2\right)\mathbb{E}\bigl[\bigl\| \nabla F(x_{k-\tau_k}) \bigr\|^2\bigr] + \gamma_k^2 \sigma^2.
\label{ProofSGD4}
\end{align}
Next, we find an upper bound on the term $\mathcal{H}$. From the first inequality in the proof of Lemma $15$ in~\cite{Koloskova:2022}, we have
\begin{align*}
\mathbb{E}\bigl[\|x_k - x_{k-\tau_k}\|^2\bigr] \leq 2\tau_k \sum_{\ell = (k - \tau_k)_+}^{k-1} \gamma_{\ell}^2\mathbb{E}\bigl[\left\|\nabla F(x_{\ell - \tau_{\ell}}) \right\|^2\bigr] + 2 \sigma^2 \sum_{\ell = (k - \tau_k)_+}^{k-1} \gamma_{\ell}^2.
\end{align*}
Substituting the above inequality into~\eqref{ProofSGD4} yields
\begin{align}
2\gamma_k \mathbb{E}\bigl[F(x_{k}) - F^\star\bigr] + \mathbb{E}\bigl[\|x_{k+1} - x^\star\|^2\bigr] \leq \mathbb{E}\bigl[\|x_{k} - x^\star\|^2\bigr] & +  2\gamma_k\tau_k L \sum_{\ell = (k - \tau_k)_+}^{k-1} \gamma_{\ell}^2\mathbb{E}\bigl[\left\|\nabla F(x_{\ell - \tau_{\ell}}) \right\|^2\bigr] \nonumber\\
&\hspace{-1cm} - \left(\frac{\gamma_k}{L} - \gamma_k^2\right)\mathbb{E}\bigl[\bigl\| \nabla F(x_{k-\tau_k}) \bigr\|^2\bigr] \nonumber \\
&\hspace{-1cm} + \left(\gamma_k^2 + 2\gamma_k L \sum_{\ell = (k - \tau_k)_+}^{k-1} \gamma_{\ell}^2\right)\sigma^2.
\label{ProofSGD5}
\end{align}
By letting $X_k = 2\gamma_k \mathbb{E}\bigl[F(x_{k}) - F^\star\bigr]$, $V_k = \mathbb{E}\bigl[\|x_{k} - x^\star\|^2\bigr]$, and $W_k= \gamma_k^2\mathbb{E}\bigl[\bigl\| \nabla F(x_{k-\tau_k}) \bigr\|^2\bigr]$, we can rewrite~\eqref{ProofSGD5} as
\begin{align*}
X_k + V_{k+1} &\leq V_k +  2\gamma_k\tau_k L \sum_{\ell = (k - \tau_k)_+}^{k-1} W_{\ell} - \left(\frac{1}{\gamma_k L} - 1\right) W_k + \left(\gamma_k^2 + 2\gamma_k L \sum_{\ell = (k - \tau_k)_+}^{k-1} \gamma_{\ell}^2\right)\sigma^2.
\end{align*}
\textbf{2. Strongly convex case}

If $F$ is $\mu$-strongly convex, then
\begin{align*}
\frac{\mu}{2}  \| x_k - x^\star\|^2 \leq F(x_k) - F^\star,    
\end{align*}
which implies that $\gamma_k \mu V_k \leq X_k$. Therefore, 
\begin{align*}
V_{k+1} &\leq (1 - \gamma_k \mu) V_k +  2\gamma_k\tau_k L \sum_{\ell = (k - \tau_k)_+}^{k-1} W_{\ell} - \left(\frac{1}{\gamma_k L} - 1\right) W_k + \left(\gamma_k^2 + 2 \gamma_k L\sum_{\ell = (k - \tau_k)_+}^{k-1} \gamma_{\ell}^2\right)\sigma^2.
\end{align*}
This completes the proof.

\subsection{Proof of Theorem~\ref{Sec:SGD Theorem1}}
\label{Sec:SGD Theorem1 Proof}
For any $K\in\mathbb{N}_0$, define the sets $\mathcal{T}_K = \{k\leq K\;|\; \tau_k \leq \tau_{\textup{th}}\}$ and $\overline{\mathcal{T}}_{K} = \{ k\leq K\;|\; \tau_k > \tau_{\textup{th}}\}$. Since $\gamma_k = \gamma$ for $\tau_k \leq \tau_{\textup{th}}$ and $\gamma_k = 0$ for $\tau_k > \tau_{\textup{th}}$, we have 
\begin{align*}
\gamma_k = 
\begin{cases}
\gamma,\quad & k\in\mathcal{T}_K,\\
0,\quad & k \in \overline{\mathcal{T}}_K.
\end{cases}
\end{align*}
\textbf{1. Convex case}

For any $k\in \mathcal{T}_K$, according to Lemma~\ref{Sec:SGD Lemma1}, the iterates generated by Algorithm~\ref{Algorithm:SGD} satisfy
\begin{align}
X_k + V_{k+1} &\leq V_k +  2\gamma\tau_k L \sum_{\ell  = (k - \tau_k)_+}^{k-1} W_{\ell} - \left(\frac{1}{\gamma L} - 1\right) W_k + \left(1 + 2\gamma \tau_k L  \right)\gamma^2\sigma^2 \nonumber\\
&\leq V_k +  2\gamma\tau_{\textup{th}} L \sum_{\ell  = (k - \tau_{\textup{th}})_+}^{k-1} W_{\ell} - \left(\frac{1}{\gamma L} - 1\right) W_k + \left(1 + 2\gamma \tau_{\textup{th}} L  \right)\gamma^2\sigma^2 \nonumber\\
& = V_k +  2\gamma\tau_{\textup{th}} L \sum_{\ell  = (k - \tau_{\textup{th}})_+}^{k} W_{\ell} - \left(\frac{1}{\gamma L} + 2\gamma\tau_{\textup{th}} L- 1\right) W_k + \left(1 + 2\gamma \tau_{\textup{th}} L  \right)\gamma^2\sigma^2,
\label{SGD Proof Theorem 1}
\end{align}
where the second inequality uses $\tau_k \leq \tau_{\textup{th}}$ for $k \in \mathcal{T}_K$. For $k \in \overline{\mathcal{T}}_K$, $\gamma_k=0$ and, hence, $X_k = 0$. Thus, by Lemma~\ref{Sec:SGD Lemma1}, we have
\begin{align*}
X_k + V_{k+1} &= V_k, \quad k \in \overline{\mathcal{T}}_K.
\end{align*}
Next, note that this implies that~\eqref{SGD Proof Theorem 1} also holds for $k \in \overline{\mathcal{T}}_K$ since $W_\ell$ is non-negative for all $\ell\in\mathbb{N}_0$, $W_k = 0$ for $k \in \overline{\mathcal{T}}_K$, and $\left(1 + 2\gamma \tau_{\textup{th}} L \right)\gamma^2\sigma^2\geq 0$. Therefore, 
\begin{align*}
X_k + V_{k+1} &\leq V_k +  p \sum_{\ell  = (k - \tau_{\textup{th}})_+}^{k} W_{\ell} - r W_k + e
\end{align*}
with $p = 2\gamma\tau_{\textup{th}} L$,  $r = \frac{1}{\gamma L} + 2\gamma\tau_{\textup{th}} L- 1$, and $e = \left(1 + 2\gamma \tau_{\textup{th}} L  \right)\gamma^2\sigma^2$ is satisfied by the iterates for every $k$. To apply Lemma~\ref{Sec:Main Results Lemma4}, we need to ensure that
\begin{align*}
2\gamma \tau_{\textup{th}} L(\tau_{\textup{th}} + 1)\leq \left(\frac{1}{\gamma L} + 2\gamma\tau_{\textup{th}} L- 1\right).
\end{align*}
This convergence condition is equivalent to
\begin{align}
2\gamma^2 \tau_{\textup{th}}^2 L^2 + \gamma L \leq 1.
\label{SGD Proof Theorem 2}
\end{align}
Using Lemma~\ref{SGD Proof Lemma1} in Appendix~\ref{Appendix:SGD} with $\alpha = \gamma L$ and $\beta = \tau_{\textup{th}}\sqrt{2}$, we choose the step-size $\gamma$ as
\begin{align*}
\gamma \in \left(0, \frac{1}{L(\tau_{\textup{th}}\sqrt{2}+1)}\right]     
\end{align*}
to guarantee the convergence condition~\eqref{SGD Proof Theorem 2}. It follows from
part $1$ of Lemma~\ref{Sec:Main Results Lemma4} that
\begin{align*}
\sum_{k = 0}^ K X_k \leq V_0 + (K +1)e, \quad K\in\mathbb{N}_0.
\end{align*}
Since $V_0 = \|x_{0} - x^\star\|^2$, $X_k = 2 \gamma_k \mathbb{E}\bigl[F(x_{k}) - F^\star\bigr]$ and $e = \left(1 + 2\gamma \tau_{\textup{th}} L  \right)\gamma^2\sigma^2$, we obtain
\begin{align*}
2\sum_{k = 0}^ K \gamma_k \mathbb{E}\bigl[F(x_{k}) - F^\star\bigr] &\leq \|x_{0} - x^\star\|^2 + (K+1)\left(1 + 2\gamma \tau_{\textup{th}} L  \right)\gamma^2\sigma^2\\
& \leq \|x_{0} - x^\star\|^2 + (K+1)(1 + \sqrt{2})\gamma^2\sigma^2,
\end{align*}
where the second inequality follows since $\gamma$ is selected to make $\gamma L (\tau_{\textup{th}}\sqrt{2} + 1)\leq 1$. By the convexity of $F$ and definition of $\bar{x}_K$, we have
\begin{align*}
F(\bar{x}_K) - F^\star = F\left(\frac{1}{\sum_{k=0}^K \gamma_k} \sum_{k=0}^K \gamma_k x_k \right) - F^\star\leq \frac{1}{\sum_{k=0}^K \gamma_k}\sum_{k=0}^K \gamma_k (F(x_k) - F^\star),
\end{align*}
which implies that
\begin{align}
\mathbb{E}\bigl[F(\bar{x}_K) - F^\star \bigr] \leq \frac{\|x_{0} - x^\star\|^2 + (K+1)\left(1 + \sqrt{2}\right)\gamma^2\sigma^2}{2 \sum_{k=0}^K \gamma_k}, \quad K\in\mathbb{N}_0.\label{SGD Proof Theorem 3}
\end{align}
Next, we will find a lower bound on  $\sum_{k=0}^K \gamma_k$. As $\gamma_k = \gamma$ for $k\in \mathcal{T}_K$ and  $\gamma_k = 0$ for $k \in \overline{\mathcal{T}}_K$, we have
\begin{align}
\sum_{k=0}^K \gamma_k = \gamma |\mathcal{T}_K|.
\label{SGD Proof Theorem 4}
\end{align}
We consider two cases:
\begin{enumerate}
\item If $2 \tau_{\textup{ave}} \leq \tau_{\textup{max}}$, then $2 \tau_{\textup{ave}}\leq \tau_{\textup{th}}$. As observed by~\cite{Koloskova:2022}, since
\begin{align*}
\sum_{k=0}^K \tau_k = (K+1)\tau_{\textup{ave}},
\end{align*}
at most $\lfloor (K + 1)/2 \rfloor$ of the terms on the left hand side of the above equality can be larger than $2\tau_{\textup{ave}}$. Thus, there are at least half of the iterations with the delay smaller than $2\tau_{\textup{ave}}$ and, hence, $\tau_{\textup{th}}$. This implies that $(K+1)/2 \leq |\mathcal{T}_K|$.
\item If $\tau_{\textup{max}}\leq 2 \tau_{\textup{ave}}$, then $\tau_{\textup{max}}\leq \tau_{\textup{th}}$. In this case, we have $\tau_k \leq \tau_{\textup{th}}$ for all $k=0,\ldots,K$ since $\tau_k \leq \tau_{\textup{max}}$ for $k\in\mathbb{N}_0$. This shows that $|\mathcal{T}_K| = K + 1$.
\end{enumerate}
Therefore, $(K+1)/2 \leq |\mathcal{T}_K|$. From~\eqref{SGD Proof Theorem 4}, we then have
\begin{align*}
\sum_{k=0}^K \gamma_k \geq \frac{\gamma(K+1)}{2}.    
\end{align*}
Combining this inequality with~\eqref{SGD Proof Theorem 3} leads to 
\begin{align*}
\mathbb{E}\bigl[F(\bar{x}_K) - F^\star \bigr] \leq \frac{\|x_{0} - x^\star\|^2 }{\gamma (K + 1)} + (1 + \sqrt{2})\gamma \sigma^2, \quad K\in\mathbb{N}_0.
\end{align*}
\textbf{2. Strongly convex case}

Similar to the proof for the convex case,  
\small
\begin{align}
V_{k+1} &\leq (1-\gamma \mu) V_k +  2\gamma \tau_{\textup{th}} L \sum_{\ell  = (k - \tau_{\textup{th}})_+}^{k} W_{\ell} - \left(\frac{1}{\gamma L} + 2\gamma\tau_{\textup{th}} L- 1\right) W_k + \left(1 + 2\gamma \tau_{\textup{th}} L  \right)\gamma^2\sigma^2
\label{SGD Proof Theorem 5}
\end{align}
\normalsize
is satisfied for $k\in \mathcal{T}_K$. For $k\in \overline{\mathcal{T}}_K$, $\gamma_k = 0$. Thus, according to Lemma~\ref{Sec:SGD Lemma1}, we have $V_{k+1}= V_k$ for $k\in \overline{\mathcal{T}}_K$. This implies that
\begin{align}
V_{k+1} &\leq V_k +  2\gamma \tau_{\textup{th}} L \sum_{\ell  = (k - \tau_{\textup{th}})_+}^{k} W_{\ell} - \left(\frac{1}{\gamma L} + 2\gamma\tau_{\textup{th}} L- 1\right) W_k
\label{SGD Proof Theorem 6}
\end{align}
holds for $k \in \overline{\mathcal{T}}_K$ since $W_\ell$ is non-negative for $\ell\in\mathbb{N}_0$ and $W_k = 0$ for $k \in \overline{\mathcal{T}}_K$. It follows from~\eqref{SGD Proof Theorem 5} and~\eqref{SGD Proof Theorem 6} that the iterates generated by Algorithm~\ref{Algorithm:SGD} satisfy
\begin{align*}
V_{k+1} &\leq q_k V_k +  2\gamma \tau_{\textup{th}} L \sum_{\ell  = (k - \tau_{\textup{th}})_+}^{k} W_{\ell} - \left(\frac{1}{\gamma L} + 2\gamma\tau_{\textup{th}} L- 1\right) W_k + e_k,
\end{align*}
where $q_k=1$ and $e_k=0$ for $k\in \overline{\mathcal{T}}_K$, and $q_k = 1 - \gamma \mu$ and $e_k = \left(1 + 2\gamma \tau_{\textup{th}} L  \right)\gamma^2\sigma^2$ for $k\in \mathcal{T}_K$. Let $q= 1 - \gamma \mu$. It is clear that $q_k \geq q$ for every $k$. Thus, to apply Lemma~\ref{Sec:Main Results Lemma4}, we need to enforce that
\begin{align*}
2\tau_{\textup{th}} + 1 &\leq \min\left\{\frac{1}{\gamma \mu}, \frac{\frac{1}{\gamma L} + 2\gamma\tau_{\textup{th}} L- 1}{2\gamma \tau_{\textup{th}} L}\right\}.
\end{align*}
This convergence condition is equivalent to
\begin{align}
\begin{cases}
\gamma \mu (2\tau_{\textup{th}}+1) &\leq 1,\\
4 \gamma^2 \tau_{\textup{th}}^2 L^2 + \gamma L & \leq 1.
\end{cases}
\label{SGD Proof Theorem 7}
\end{align}
If $\gamma L (2\tau_{\textup{th}} + 1)\leq 1$, the first inequality in~\eqref{SGD Proof Theorem 7} holds since $\mu\leq L$. Using Lemma~\ref{SGD Proof Lemma1} in Appendix~\ref{Appendix:SGD} with $\alpha = \gamma L$ and $\beta = 2\tau_{\textup{th}}$, we can see that if $\gamma L (2\tau_{\textup{th}} + 1)\leq 1$, the second inequality in~\eqref{SGD Proof Theorem 7} also holds. Thus, the convergence condition~\eqref{SGD Proof Theorem 7} is satisfied for 
\begin{align*}
\gamma \in \left(0, \frac{1}{L(2\tau_{\textup{th}}+1)}\right].  
\end{align*}
Therefore, by part $2$ of Lemma~\ref{Sec:Main Results Lemma4}, we have
\begin{align}
V_{K+1} \leq Q_{K+1} V_0 + Q_{K+1}\sum_{k=0}^K \frac{e_k}{Q_{k+1}}, \quad K\in \mathbb{N}_0,
\label{SGD Proof Theorem 8}
\end{align}
where $Q_k = \prod_{\ell = 0}^{k-1} q_{\ell}$. Since $q_\ell=1$ for $\ell \in \overline{\mathcal{T}}_K$ and $q_\ell = 1 - \gamma \mu$ for $\ell\in \mathcal{T}_K$, we have
\begin{align*}
Q_{K+1} = \prod_{\ell = 0}^{K} q_{\ell}=\prod_{\ell \in \mathcal{T}_K}(1-\gamma \mu) = (1 - \gamma \mu)^{|\mathcal{T}_K|}.
\end{align*}
As discussed in the proof for the convex case, $(K+1)/2 \leq |\mathcal{T}_K|$. Thus,
\begin{align}
Q_{K+1} &\leq (1 -\gamma \mu)^{\frac{K+1}{2}} \nonumber\\
& \leq \textup{exp}\left(-\frac{\gamma \mu (K + 1)}{2}\right),
\label{SGD Proof Theorem 9}
\end{align}
where the second inequality follows from the fact that $1 - \alpha \leq \textup{exp}(-\alpha)$ for $\alpha\geq 0$. Next, we bound $Q_{K+1}\sum_{k=0}^{K} e_k/Q_{k+1}$. Since $e_k = 0$ for $k \in \overline{\mathcal{T}}_K$ and $e_k = \left(1 + 2\gamma \tau_{\textup{th}} L  \right)\gamma^2\sigma^2$ for $k\in \mathcal{T}_K$, we have
\begin{align}
Q_{K+1}\sum_{k=0}^K \frac{e_k}{Q_{k+1}} & = (1-\gamma \mu)^{|\mathcal{T}_K|}\sum_{k\in \mathcal{T}_K} \frac{\left(1 + 2\gamma \tau_{\textup{th}} L  \right)\gamma^2\sigma^2}{\prod_{\ell = 0}^{k} q_{\ell}} \nonumber\\
& = \left(1 + 2\gamma \tau_{\textup{th}} L  \right)\gamma^2\sigma^2\sum_{k=0}^{|\mathcal{T}_K|-1} (1-\gamma \mu)^{k} \nonumber\\
&\leq \left(1 + 2\gamma \tau_{\textup{th}} L  \right)\gamma^2\sigma^2\sum_{k=0}^\infty (1-\gamma \mu)^{k} \nonumber\\
&= \frac{\left(1 + 2\gamma \tau_{\textup{th}} L  \right)\gamma\sigma^2}{\mu}.
\label{SGD Proof Theorem 10}
\end{align}
Substituting~\eqref{SGD Proof Theorem 9} and~\eqref{SGD Proof Theorem 10} into~\eqref{SGD Proof Theorem 8} yields
\begin{align*}
V_{K+1} &\leq \textup{exp}\left(-\frac{\gamma \mu (K + 1)}{2}\right) V_0 + \frac{\left(1 + 2\gamma \tau_{\textup{th}} L  \right)\gamma\sigma^2}{\mu},\\
&\leq \textup{exp}\left(-\frac{\gamma \mu (K + 1)}{2}\right) V_0 + \frac{2\gamma\sigma^2}{\mu},
\end{align*}
where the second inequality is due to that $\gamma L (2\tau_{\textup{th}} + 1)\leq 1$. Since $V_k = \mathbb{E}\bigl[\|x_{k} - x^\star\|^2\bigr]$, it follows that
\begin{align*}
\mathbb{E}\bigl[\|x_{K} - x^\star\|^2\bigr] \leq \textup{exp}\left(-\frac{\gamma \mu K}{2}\right) \|x_{0} - x^\star\|^2 + \frac{2\gamma\sigma^2}{\mu}\quad K\in \mathbb{N}.
\end{align*}
This completes the proof. 

\subsection{Proof of Theorem~\ref{Sec:SGD Theorem2}}
\label{Sec:SGD Theorem2 Proof}
\textbf{1. Convex case}

According to Theorem~\ref{Sec:SGD Theorem1}, the iterates generated by the asynchronous \textsc{Sgd} method with $\gamma \leq  \frac{1}{L(\tau_{\textup{th}}\sqrt{2}+1)}$ satisfy
\begin{align}
\mathbb{E}\bigl[F(\bar{x}_K) - F^\star \bigr] \leq \frac{\|x_{0} - x^\star\|^2 }{\gamma (K + 1)} + (1 + \sqrt{2})\gamma \sigma^2, \quad K\in\mathbb{N}_0.
\label{SGD Proof Theorem2 1}
\end{align}
We will find the total number of iterations $K_\epsilon$ necessary to achieve $\epsilon$-optimal solution, i.e., 
\begin{align*}
\mathbb{E}\bigl[F(\bar{x}_K) - F^\star \bigr] \leq \epsilon, \quad \textup{for}\; K\geq K_\epsilon.
\end{align*}
By selecting $\gamma$ so that $(\sqrt{2} + 1)\gamma \sigma^2 \leq \epsilon/2$, the second term on the right-hand side of~\eqref{SGD Proof Theorem2 1} is less than $\epsilon/2$. Set $\gamma$ to
\begin{align}
\gamma = \min \left\{ \frac{1}{L(\tau_{\textup{th}}\sqrt{2}+1)},\; \frac{\epsilon}{2(\sqrt{2}+1)\sigma^2}\right\}.
\label{SGD Proof Theorem2 2}
\end{align}
We now choose $K$ so
that the first term on the right-hand side of~\eqref{SGD Proof Theorem2 1} is less than $\epsilon/2$, i.e.,
\begin{align*}
  \frac{\|x_{0} - x^\star\|^2}{\gamma (K+1)}\leq \frac{\epsilon}{2}.    
\end{align*}
This implies that
\begin{align*}
  K\geq K_\epsilon&:=\frac{2\|x_{0} - x^\star\|^2}{\gamma \epsilon }-1\\
  &\overset{\eqref{SGD Proof Theorem2 2}}{=} 2\max\left\{ \frac{L(\tau_{\textup{th}}\sqrt{2}+1)}{\epsilon},\; \frac{2(\sqrt{2}+1)\sigma^2}{\epsilon^2}\right\}\|x_{0} - x^\star\|^2 - 1.
\end{align*}
\textbf{2. Strongly convex case}

It follows from Theorem~\ref{Sec:SGD Theorem1} that the iterates generated by the asynchronous \textsc{Sgd} method with $\gamma \leq  \frac{1}{L(2\tau_{\textup{th}}+1)}$ satisfy
\begin{align}
\mathbb{E}\bigl[\|x_{K} - x^\star\|^2\bigr] \leq \textup{exp}\left(-\frac{\gamma \mu K}{2}\right) \|x_{0} - x^\star\|^2 + \frac{2\gamma\sigma^2}{\mu}\quad K\in \mathbb{N}.
\label{SGD Proof Theorem2 3}
\end{align}
With the choice of
\begin{align}
\gamma = \min \left\{ \frac{1}{L(2\tau_{\textup{th}}+1)},\;\frac{\epsilon \mu}{4\sigma^2}\right\}, 
\label{SGD Proof Theorem2 4}
\end{align}
the second term on the right-hand side of~\eqref{SGD Proof Theorem2 3} is less than $\epsilon/2$. Next we choose $K$ so that
\begin{align*}
\textup{exp}\left(-\frac{\gamma \mu K}{2}\right) \|x_{0} - x^\star\|^2 \leq \frac{\epsilon}{2}.    
\end{align*}
Taking logarithm of both sides and rearranging the terms gives
\begin{align*}
  K\geq K_\epsilon&:=\frac{2}{\gamma \mu}\log\left(\frac{2\|x_{0} - x^\star\|^2}{\epsilon }\right)\\
  &\overset{\eqref{SGD Proof Theorem2 4}}{=} 2\max\left\{ \frac{L(2\tau_{\textup{th}}+1)}{\mu},\; \frac{4\sigma^2}{\mu^2\epsilon}\right\}\log\left(\frac{2\|x_{0} - x^\star\|^2}{\epsilon }\right).
\end{align*}

\subsection{Proof of Theorem~\ref{Sec:SGD Theorem3}}
\label{Sec:SGD Theorem3 Proof}
\textbf{1. Convex case}

According to Theorem~\ref{Sec:SGD Theorem1}, we have
\begin{align*}
\mathbb{E}\bigl[F(\bar{x}_\mathcal{K}) - F^\star \bigr] \leq \frac{\|x_{0} - x^\star\|^2 }{\gamma (\mathcal{K} + 1)} + (1 + \sqrt{2})\gamma \sigma^2.
\end{align*}
By minimizing the right-hand side of the above inequality with respect to $\gamma$ over the interval 
\begin{align*}
\left(0, \frac{1}{L(\tau_{\textup{th}}\sqrt{2}+1)}\right],
\end{align*}
we obtain
\begin{align*}
\gamma^\star = \min \left\{ \frac{1}{L(\tau_{\textup{th}}\sqrt{2}+1)},\; \frac{\|x_{0} - x^\star\|}{\sigma\sqrt{\sqrt{2} + 1}\sqrt{\mathcal{K} + 1}}\right\}.  
\end{align*}
The rest of the proof is similar to the one for the serial \textsc{Sgd} method derived in~\cite{Lan:12} and thus omitted.

\noindent\textbf{2. Strongly convex case}

By Theorem~\ref{Sec:SGD Theorem1}, we have
\begin{align*}
\mathbb{E}\bigl[\|x_{\mathcal{K}} - x^\star\|^2\bigr] \leq \textup{exp}\left(-\frac{\gamma \mu \mathcal{K}}{2}\right) \|x_{0} - x^\star\|^2 + \frac{2\gamma\sigma^2}{\mu}.
\end{align*}
The result follows from Lemma $1$ in~\cite{Karimireddy:2021}.

%
%
\section{Proofs for Subsection~\ref{Sec:ARock}}
\label{Appendix:ARock}
In this section, we provide the proofs for the results presented in Subsection~\ref{Sec:ARock}. 

\subsection{Preliminaries}
Let $I \in \mathbb{R}^{d \times d}$ be the identity matrix. We define the matrices $U_i\in\mathbb{R}^{d\times d_i}$, $i \in [m]$, for which $I = [U_1, \ldots, U_m]$. Then, any vector~$x=\bigl([x]_1,\ldots, [x]_m\bigr)\in \mathbb{R}^d$ can be represented as 
\begin{align*}
 x = \sum_{i=1}^m U_i [x]_i,\quad [x]_i \in \mathbb{R}^{d_i},\; i \in [m].
\end{align*}
Since $[x_{k+1}]_{j} = [x_{k}]_{j} - \gamma S_{j}(\widehat{x}_k)$ for $j = i_k$ and $[x_{k+1}]_{j} = [x_{k}]_{j}$ for $j \neq i_k$, the update formula of Algorithm~\ref{Algorithm:ARock} can be written as 
\begin{align}
x_{k+1} = x_{k} - \gamma U_{i_k} S_{i_k}(\widehat{x}_k),\quad k\in\mathbb{N}_0.
\label{ArockUpdateRule}
\end{align}
Note that $x_k$ depends on the observed realization of the random variable~$\xi_{k-1} := \{i_0, \ldots, i_{k-1}\}$ but not on~$i_j$ for any $j\geq k$. For convenience, we define $\xi_{-1} = \emptyset$. We use $\mathbb{E}$ to denote the expectation over all random variables, and $\mathbb{E}_{k}$ to denote the conditional expectation in term of $i_k$ given $\xi_{k-1}$. 

For any random vector~$x \in \mathbb{R}^d$, the variance can be decomposed as
\begin{align}
\mathbb{E}\left[\bigl\|x- \mathbb{E}[x]\bigr\|^2\right] = \mathbb{E}\left[\bigl\|x\bigr\|^2\right] - \bigl\|\mathbb{E}[x]\bigr\|^2.
\label{VarianceBound}
\end{align}
For any vectors $a,b \in \mathbb{R}^d$ and any constant $\eta>0$, the inequalities 
\begin{align}
\langle a,\; b \rangle &\leq \frac{\eta\|a\|^2}{2} + \frac{\|b\|^2}{2\eta},\label{Cauchy1}\\
\|a + b\|^2 &\leq \left(1 + \eta\right)\|a \|^2 + \left(1 +  \frac{1}{\eta}\right)\|b\|^2,\label{Cauchy2}\\
-\|a \|^2 &\leq -\frac{\|b\|^2}{1 + \eta} +\frac{\|a -b\|^2}{\eta}, \label{Cauchy3}
\end{align}
hold by the Cauchy-Schwarz inequality.

We need the following two lemmas in the convergence analysis of ARock. The first one provides an upper bound on the expectation of $\|x_k - \widehat{x}_k\|^2$.

\begin{lemma}
\label{Lemma1_ARock}
Let $\{x(k)\}$ be the sequence generated by  Algorithm~\ref{Algorithm:ARock}. Then, it holds that
\begin{align*}
\mathbb{E}\bigl[\|x_k - \widehat{x}_k\|^2\bigr] \leq \frac{\gamma^2 \left(\sqrt{m} + \sqrt{\tau}\right)^2}{m^2}\sum_{\ell = (k - \tau)_+}^{k-1}\mathbb{E}\bigl[\left\|S(\widehat{x}_{\ell}) \right\|^2\bigr], \quad k\in\mathbb{N}_0.
\end{align*}
\end{lemma}

\begin{proof}
Let $\eta$ be a positive constant. From~\eqref{ArockRead}, we have
\begin{align}
\mathbb{E}\bigl[\|x_k - \widehat{x}_k\|^2\bigr] &= \mathbb{E}\left[\left\| \sum_{j\in J_k}(x_{j+1} - x_{j})\right\|^2 \right] \nonumber\\
&\overset{\eqref{ArockUpdateRule}}{=} \gamma^2 \mathbb{E}\left[\left\| \sum_{j\in J_k} U_{i_j} S_{i_j}(\widehat{x}_j)\right\|^2\right] \nonumber\\
& = \gamma^2 \mathbb{E}\left[\left\| \sum_{j\in J_k} \bigl( U_{i_j} S_{i_j}(\widehat{x}_j) - \frac{1}{m}S(\widehat{x}_j)\bigr) + \frac{1}{m} \sum_{j\in J_k}S(\widehat{x}_j) \right\|^2\right]\nonumber\\
& \overset{\eqref{Cauchy2}}{\leq} \gamma^2(1+\eta) \underbrace{\mathbb{E}\left[\left\| \sum_{j\in J_k} \bigl( U_{i_j} S_{i_j}(\widehat{x}_j) - \frac{1}{m}S(\widehat{x}_j)\bigr)\right\|^2\right]}_{\mathcal{H}_1} + \frac{\gamma^2}{m^2}\left(1+\frac{1}{\eta}\right)  \underbrace{\mathbb{E}\left[\left\| \sum_{j\in J_k}S(\widehat{x}_j) \right\|^2\right]}_{\mathcal{H}_2}.
\label{P1_0}
\end{align}
We will find upper bounds on the quantities $\mathcal{H}_1$ and $\mathcal{H}_2$. We expand $\mathcal{H}_1$ as follows
\begin{align*}
\mathcal{H}_1 &= \sum_{j\in J_k} \mathbb{E}\left[\left\| U_{i_j} S_{i_j}(\widehat{x}_j) - \frac{1}{m}S(\widehat{x}_j)\right\|^2\right] + \sum_{j, j'\in J_k, j>j'} 2 \mathbb{E}\left[\left\langle  U_{i_j} S_{i_j}(\widehat{x}_j) - \frac{1}{m}S(\widehat{x}_{j}),\;  U_{i_{j'}} S_{i_{j'}} (\widehat{x}_{j'})- \frac{1}{m}S(\widehat{x}_{j'})\right\rangle\right]\\
&= \sum_{j\in J_k} \mathbb{E}\left[\left\|U_{i_{j}} S_{i_j}(\widehat{x}_j) - \frac{1}{m}S(\widehat{x}_j)\right\|^2\right],
\end{align*}
where the second equality uses the fact that
\begin{align*}
&\sum_{j, j'\in J_k, j>j'} \mathbb{E}\left[\left\langle U_{i_{j}} S_{i_{j}}(\widehat{x}_j) - \frac{1}{m}S(\widehat{x}_{j}) ,\; U_{i_{j'}} S_{i_{j'}}(\widehat{x}_{j'}) - \frac{1}{m}S(\widehat{x}_{j'})\right\rangle\right]\\
&\hspace{2cm}= 
\sum_{j, j'\in J_k, j>j'} \mathbb{E}\left[\mathbb{E}_{j}\left[\left\langle U_{i_{j}} S_{i_{j}}(\widehat{x}_j) - \frac{1}{m}S(\widehat{x}_{j}),\; U_{i_{j'}} S_{i_{j'}}(\widehat{x}_{j'}) - \frac{1}{m}S(\widehat{x}_{j'})\right\rangle\right]\right]\\
&\hspace{2cm}= 
\sum_{j, j'\in J_k, j>j'} \mathbb{E}\left[\left\langle \mathbb{E}_{j}\left[U_{i_{j}} S_{i_{j}}(\widehat{x}_j)- \frac{1}{m}S(\widehat{x}_{j})\right],\;U_{i_{j'}} S_{i_{j'}}(\widehat{x}_{j'}) -
\frac{1}{m}S(\widehat{x}_{j'})\right\rangle\right]\\
&\hspace{2cm}= 0.
\end{align*}
We then bound $\mathcal{H}_1$ by
\begin{align}
\mathcal{H}_1 &= \sum_{j\in J_k} \mathbb{E}\left[\mathbb{E}_{j}\left[\left\|U_{i_{j}}S_{i_j}(\widehat{x}_j) - \frac{1}{m}S(\widehat{x}_j)\right\|^2\right]\right]\nonumber\\
& \overset{\eqref{VarianceBound}}{=} \frac{1}{m}\sum_{j\in J_k} \mathbb{E}\left[\left\|S(\widehat{x}_j)\right\|^2\right] - \frac{1}{m^2} \sum_{j\in J_k} \mathbb{E}\left[\left\|S(\widehat{x}_j)\right\|^2\right]\nonumber\\
&\leq \frac{1}{m}\sum_{j\in J_k} \mathbb{E}\left[\left\|S(\widehat{x}_j)\right\|^2\right].
\label{H1BoundP1}
\end{align}
Next, we turn to  $\mathcal{H}_2$. We rewrite $\mathcal{H}_2$ as
\begin{align*}
\mathcal{H}_2 = |J_k|^2 \mathbb{E}\left[\left\|\sum_{j\in J_k}\frac{S(\widehat{x}_j)}{|J_k|} \right\|^2\right].
\end{align*}
By convexity of the squared Euclidean norm $\|\cdot\|^2$, we get
\begin{align}
\mathcal{H}_2 \leq |J_k|\sum_{j\in J_k} \mathbb{E}\left[\left\|S(\widehat{x}_j) \right\|^2\right].
\label{H2BoundP1}
\end{align}
Substituting~\eqref{H1BoundP1} and~\eqref{H2BoundP1} into~\eqref{P1_0} yields
\begin{align*}
\mathbb{E}\bigl[\|x_k - \widehat{x}_k\|^2\bigr] &\leq \frac{\gamma^2}{m^2}\left((1+\eta)m +\left(1+\frac{1}{\eta}\right)|J_k| \right)\sum_{j\in J_k} \mathbb{E}\left[\left\|S(\widehat{x}_j) \right\|^2\right].
\end{align*}
From Assumption~\ref{Assumption:ARock}.2, $(k - \tau)_+ \leq j \leq k-1$ for any $j\in J_k$ and $|J_k|\leq \tau$ for all $k\in\mathbb{N}_0$. Thus,
\begin{align*}
\mathbb{E}\bigl[\|x_k - \widehat{x}_k\|^2\bigr] \leq \frac{\gamma^2}{m^2}\left((1+\eta)m +\left(1+\frac{1}{\eta}\right) \tau \right)\sum_{j = (k-\tau)_+}^{k-1} \mathbb{E}\left[\left\|S(\widehat{x}_j) \right\|^2\right].
\end{align*}
It is easy to verify that the optimal choice of $\eta$, which minimizes  the right-hand-side of the above inequality, is $\eta = \sqrt{\tau/m}$. Using this choice of $\eta$ and the change of variable $\ell = j$ completes the proof.
\end{proof}

As a consequence of this lemma, we derive a bound on the expectation of $\langle \widehat{x}_k - x_{k},  S(\widehat{x}_k)\bigr\rangle$.

\begin{lemma}
\label{Lemma2_ARock}
The iterates $\{x(k)\}$ generated by Algorithm~\ref{Algorithm:ARock} satisfy
\begin{align*}
\mathbb{E}\bigl[\langle \widehat{x}_k - x_{k},  S(\widehat{x}_k)\bigr\rangle\bigr] \leq \frac{\gamma}{2}\left(\frac{\tau}{m} +\sqrt{\frac{\tau}{m}}\right) \left(\frac{1}{\tau}\sum_{\ell = (k - \tau)_+}^{k-1}\mathbb{E}\bigl[\left\|S(\widehat{x}_{\ell}) \right\|^2\bigr] + \mathbb{E}\bigl[\|S(\widehat{x}_k)\|^2\bigr]\right).
\end{align*}
\end{lemma}

\begin{proof}
 Let $\eta$ be a positive constant. From~\eqref{Cauchy1}, we have
\begin{align*}
\mathbb{E}\bigl[\langle \widehat{x}_k - x_{k},  S(\widehat{x}_k)\bigr\rangle\bigr] \leq \frac{\eta}{2} \mathbb{E}\bigl[\|x_k - \widehat{x}_k\|^2\bigr] + \frac{1}{2\eta}\mathbb{E}\bigl[\|S(\widehat{x}_k)\|^2\bigr].
\end{align*}
It follows from Lemma~\ref{Lemma1_ARock} that
\begin{align*}
\mathbb{E}\bigl[\langle \widehat{x}_k - x_{k},  S(\widehat{x}_k)\bigr\rangle\bigr] \leq \frac{\eta\gamma^2 \left(\sqrt{m} + \sqrt{\tau}\right)^2}{2m^2}\sum_{\ell = (k - \tau)_+}^{k-1}\mathbb{E}\bigl[\left\|S(\widehat{x}_{\ell}) \right\|^2\bigr] + \frac{1}{2\eta}\mathbb{E}\bigl[\|S(\widehat{x}_k)\|^2\bigr].
\end{align*}
Substituting
\begin{align*}
\eta = \frac{m}{\gamma\sqrt{\tau}(\sqrt{m} + \sqrt{\tau})}
\end{align*}
into the above inequality proves the statement of the lemma.
\end{proof}

\subsection{Proof of Lemma~\ref{Sec:ARock Lemma1}}
\label{Sec:ARock Lemma1 Proof}
We reuse the proof technique from~\cite{Hannah:17}. Our point of departure with~\cite{Hannah:17} is to use Lemmas~\ref{Lemma1_ARock} and~\ref{Lemma2_ARock} to bound $\mathbb{E}_k\bigl[\|x_k - \widehat{x}_k\|^2\bigr]$ and $\mathbb{E}_k\bigl[\langle \widehat{x}_k - x_{k},  S(\widehat{x}_k)\bigr\rangle\bigr]$. Let $x^\star$ be a fixed point of the operator $T$. By subtracting~$x^\star$ from both sides of~\eqref{ArockUpdateRule} and then taking norm squares, we have
\begin{align}
\|x_{k+1} - x^\star\|^2 = \|x_{k} - x^\star\|^2 - 2\gamma \left\langle x_{k} - x^\star,  U_{i_k} S_{i_k}(\widehat{x}_k)\right\rangle +  \gamma^2\left\|U_{i_k} S_{i_k}(\widehat{x}_k)\right\|^2.
\label{ProofArock0}
\end{align}
By taking conditional expectation on both sides of~\eqref{ProofArock0} with respect to only the random variable $i_k$, we obtain
\begin{align}
\mathbb{E}_k\bigl[\|x_{k+1} - x^\star\|^2\bigr] &=  \|x_{k} -  x^\star\|^2 - 2\frac{\gamma}{m}\bigl\langle x_{k} -  x^\star,  S(\widehat{x}_k)\bigr\rangle + \frac{\gamma^2}{m} \|S(\widehat{x}_k)\|^2\nonumber\\
& = \|x_{k} - x^\star\|^2 - 2\frac{\gamma}{m}\underbrace{\bigl\langle \widehat{x}_k -  x^\star,  S(\widehat{x}_k)\bigr\rangle}_{\mathcal{H}}
+  2\frac{\gamma}{m}\bigl\langle \widehat{x}_k - x_{k},  S(\widehat{x}_k)\bigr\rangle + \frac{\gamma^2}{m} \|S(\widehat{x}_k)\|^2.
\label{ProofArock1}
\end{align}
We will use $\mathcal{H}$ to generate a $\|x_{k} -  x^\star\|^2$ term to help prove linear convergence. Since $T$ is pseudo-contractive with contraction modulus $c$ and $S(x^\star) = 0$, it follows from Lemma 8 in~\cite{Hannah:17} that
\begin{align}
-\mathcal{H} \leq -\frac{(1-c^2)}{2}\|\widehat{x}_k -  x^\star\|^2 - \frac{1}{2} \|S(\widehat{x}_k)\|^2.
\label{ProofArock2}
\end{align}
We then use~\eqref{Cauchy3} to convert $\|\widehat{x}_k -  x^\star\|^2$ to $\|x_{k} -  x^\star\|^2$ as follows
\begin{align*}
-\|\widehat{x}_k -  x^\star\|^2 \leq  -\left(\frac{1}{1 + \eta}\right)\|x_k - x^\star\|^2 +\frac{1}{\eta}\|x_k -  \widehat{x}_k\|^2,
\end{align*}
where $\eta$ is a positive constant. Combining this inequality and~\eqref{ProofArock2}, we have
\begin{align*}
-\mathcal{H} &\leq -\frac{(1-c^2)}{2(1+\eta)}\|{x}_k -  x^\star\|^2 + \frac{(1-c^2)}{2\eta}\|x_k - \widehat{x}_k\|^2 - \frac{1}{2} \|S(\widehat{x}_k)\|^2\nonumber\\
& \leq -\frac{(1-c^2)}{2(1+\eta)}\|{x}_k -  x^\star\|^2 + \frac{1}{2\eta}\|x_k - \widehat{x}_k\|^2 - \frac{1}{2} \|S(\widehat{x}_k)\|^2.
\end{align*}
Substituting the above inequality into~\eqref{ProofArock1} yields 
\begin{align*}
\mathbb{E}_k\bigl[\|x_{k+1} - x^\star\|^2\bigr] &\leq \left(1 - \frac{\gamma(1-c^2)}{m(1+\eta)} \right)\|x_{k} - x^\star\|^2 +  2\frac{\gamma}{m}\bigl\langle \widehat{x}_k - x_{k},  S(\widehat{x}_k)\bigr\rangle\\
& +\frac{\gamma}{m \eta}\|x_k - \widehat{x}_k\|^2 - \frac{\gamma}{m}\left(1 - \gamma\right)\|S(\widehat{x}_k)\|^2.
\end{align*}
After taking the full expectation on both sides, it follows from Lemmas~\ref{Lemma1_ARock} and~\ref{Lemma2_ARock} that
\begin{align}
\mathbb{E}\bigl[\|x_{k+1} - x^\star\|^2\bigr] &\leq \left(1 - \frac{\gamma(1-c^2)}{m(1+\eta)} \right)\mathbb{E}\bigl[\|x_{k} - x^\star\|^2\bigr]\nonumber\\
&+ \frac{\gamma^2}{m\tau}\left(\frac{\tau}{m} +\sqrt{\frac{\tau}{m}}\right)\left(1 + \left(\frac{\tau}{m} +\sqrt{\frac{\tau}{m}}\right)\frac{\gamma}{\eta}\right)\sum_{\ell = (k - \tau)_+}^{k-1}\mathbb{E}\bigl[\left\|S(\widehat{x}_{\ell}) \right\|^2\bigr] \nonumber\\
& - \frac{\gamma}{m}\left(1 - \gamma\left(1 + \frac{\tau}{m} +\sqrt{\frac{\tau}{m}}\right)\right)\mathbb{E}\bigl[\|S(\widehat{x}_k)\|^2\bigr].
\label{ProofArock04}
\end{align}
Let $V_k = \mathbb{E}\bigl[\|x_{k} - x^\star\|^2 $, $W_k =  \mathbb{E}\bigl[\| S(\widehat{x}_k)\|^2\bigr]$, and
\begin{align*}
\eta = \gamma \left(\frac{\tau}{m} + \sqrt{\frac{\tau}{m}}\right).
\end{align*}
Then, the inequality~\eqref{ProofArock04} can be rewritten as
\begin{align*}
V_{k+1} &\leq \left(1 - \frac{\gamma(1-c^2)}{m\left(1+\gamma \left(\frac{\tau}{m} + \sqrt{\frac{\tau}{m}}\right)\right)} \right) V_k + \frac{2\gamma^2}{m\tau}\left(\frac{\tau}{m} + \sqrt{\frac{\tau}{m}}\right)\sum_{\ell = (k - \tau)_+}^{k-1}W_{\ell} \\
& \hspace{5mm} - \frac{\gamma}{m}\left(1 - \gamma\left(1 + \frac{\tau}{m} + \sqrt{\frac{\tau}{m}}\right)\right) W_k.
\end{align*}
This completes the proof. 

\subsection{Proof of Theorem~\ref{Theorem:ARock}}
\label{Sec:ARock Theorem1 Proof}
According to Lemma~\ref{Sec:ARock Lemma1}, the iterates generated by Algorithm~\ref{Algorithm:ARock} satisfy 
\begin{align*}
V_{k+1} &\leq q V_k + p \sum_{\ell=(k-\tau)_+}^{k} W_{\ell} -  r W_k
\end{align*}
for every  $k\in\mathbb{N}_0$, where
\begin{align*}
q= 1 - \frac{\gamma(1-c^2)}{m\left(1+\gamma \left(\frac{\tau}{m} + \sqrt{\frac{\tau}{m}}\right)\right)}, \quad p =   \frac{2\gamma^2}{m\tau}\left(\frac{\tau}{m} + \sqrt{\frac{\tau}{m}}\right),
\end{align*}
and
\begin{align*}
 r =  \frac{\gamma}{m}\left(1 - \gamma\left(1 + \frac{\tau}{m} + \sqrt{\frac{\tau}{m}}\right)\right) +  \frac{2 \gamma^2}{m\tau}\left(\frac{\tau}{m} + \sqrt{\frac{\tau}{m}}\right).
\end{align*}
To apply Lemma~\ref{Sec:Main Results Lemma4}, we need to enforce that
\begin{align*}
2\tau + 1 &\leq \min\left\{\frac{1}{1-q}, \frac{r}{p}\right\}.
\end{align*}
This convergence condition is equivalent to
\begin{align}
\begin{cases}
\gamma (1-c^2)(2\tau+1) &\leq m\left(1+\gamma \left(\frac{\tau}{m} + \sqrt{\frac{\tau}{m}}\right)\right),\\
\gamma \left(1+5\left(\frac{\tau}{m} + \sqrt{\frac{\tau}{m}}\right)\right)& \leq 1.
\end{cases}
\label{ProofArock05}
\end{align}
Using the change of variable 
\begin{align*}
h = \gamma \left(1+5\left(\frac{\tau}{m} + \sqrt{\frac{\tau}{m}}\right)\right),
\end{align*}
the inequalities~\eqref{ProofArock05} can be rewritten as
\begin{align}
\begin{cases}
 h(1-c^2)(2\tau+1) &\leq m\left( (5 + h) \left(\frac{\tau}{m} + \sqrt{\frac{\tau}{m}}\right) +  1\right),\\
h & \leq 1.
\end{cases}
\label{ProofArock06}
\end{align}
Since $m\geq 1$ and $c \in (0,1)$, we have
\begin{align*}
(1-c^2)(2\tau + 1) &\leq (5+h)\tau + m \\
&= m\left((5+h) \left(\frac{\tau}{m}\right) + 1\right)\\
& \leq m\left((5+h) \left(\frac{\tau}{m} + \sqrt{\frac{\tau}{m}}\right) + 1\right).
\end{align*}
Thus, the inequalities~\eqref{ProofArock06} hold for any $ h\in [0,1]$. This shows that if the step-size $\gamma$ is set to  
\begin{align*}
\gamma = \frac{h}{1+5\left(\frac{\tau}{m} + \sqrt{\frac{\tau}{m}}\right)},\quad h\in (0,1],
\end{align*}
then the convergence condition~\eqref{ProofArock05} is satisfied. Therefore, by part $2$ of Lemma~\ref{Sec:Main Results Lemma4}, $V_k \leq q^k V_0$ for all $k\in\mathbb{N}_0$. Since $V_k = \mathbb{E}\bigl[\|x_{k} - x^\star\|^2\bigr] $, it follows that
\begin{align*}
 \mathbb{E}\bigl[\|x_{k} - x^\star\|^2 &\leq \left(1 - \frac{h(1-c^2)}{m\left(1+(5+h) \left(\frac{\tau}{m} + \sqrt{\frac{\tau}{m}}\right)\right)} \right)^k  \|x_{0} - x^\star\|^2\\
&\leq \left(1 - \frac{h(1-c^2)}{m\left(1+6 \left(\frac{\tau}{m} + \sqrt{\frac{\tau}{m}}\right)\right)} \right)^k  \|x_{0} - x^\star\|^2,
\end{align*}
where the second inequality is due to that $h\leq 1$. This matches the upper bound~\eqref{ARock:Theorem1} on the convergence rate.

Finally, we derive the iteration complexity bound~\eqref{ARock:Theorem2}. Taking logarithm of both sides of~\eqref{ARock:Theorem1} yields
\begin{align*}
\log\left(\mathbb{E}\bigl[\|x_{k} - x^\star\|^2\bigr]\right) \leq k\log\left(1 - \frac{h(1-c^2)}{m\left(1+6 \left(\frac{\tau}{m} + \sqrt{\frac{\tau}{m}}\right)\right)}\right)+ \log\left(\|x_{0} - x^\star\|^2\right).
\end{align*}
Since $\log(1 + x) \leq x$ for any $x > -1$, it follows that
\begin{align*}
\log\left(\mathbb{E}\bigl[\|x_{k} - x^\star\|^2\bigr]\right)  \leq - \left(\frac{h(1-c^2)}{m\left(1+6 \left(\frac{\tau}{m} + \sqrt{\frac{\tau}{m}}\right)\right)}\right) k + \log\left(\|x_{0} - x^\star\|^2\right).
\end{align*}
Therefore, for any $k$ satisfying
\begin{align}
 - \left(\frac{h(1-c^2)}{m\left(1+6 \left(\frac{\tau}{m} + \sqrt{\frac{\tau}{m}}\right)\right)}\right) k + \log\left(\|x_{0} - x^\star\|^2\right) \leq \log(\epsilon),
\label{ProofArock07}
\end{align}
we have $\log\left(\mathbb{E}\bigl[\|x_{k} - x^\star\|^2\bigr]\right)  \leq \log(\epsilon)$, implying that $\mathbb{E}\bigl[\|x_{k} - x^\star\|^2\bigr] \leq \epsilon$. Rearranging terms in~\eqref{ProofArock07} completes the proof.

%
%
\section{Proofs for Subsection~\ref{Sec:ContractiveMapping}}
This section provides the proofs for the results presented in Subsection~\ref{Sec:ContractiveMapping}. 

\subsection{Proof of Lemma~\ref{Sec:ContractiveMapping Lemma1}}
\label{Sec:ContractiveMapping Lemma1 Proof}
For each $i \in [m]$, let $\kappa_i$ and $\kappa'_i$ be two arbitrary consecutive elements of $\mathcal{K}_i$. From~\eqref{Sec:ContractiveMapping Eq3}, we have
\begin{align*}
\bigl[x_{k + 1}\bigr]_i = T_i \left( \left[x_{s_{i 1, \kappa_i}}\right]_1, \ldots, \left[x_{s_{i m, \kappa_i}}\right]_m \right), \quad k\in \bigl[\kappa_i, \kappa'_i\bigr).
\end{align*}
As~$t_i(k) = \kappa_i$ for $k\in \bigl[\kappa_i, \kappa'_i\bigr)$, we obtain
\begin{align*}
\bigl[x_{k + 1}\bigr]_i = T_i \left( \left[x_{s_{i 1, t_i(k)}}\right]_1, \ldots, \left[x_{s_{i m, t_i(k)}}\right]_m \right), \quad k\in \bigl[\kappa_i, \kappa'_i\bigr).
\end{align*}
Since $\kappa_i$ and $\kappa'_i$ are two arbitrary consecutive elements of $\mathcal{K}_i$ and $0\in \mathcal{K}_i$ for each $i$, we can rewrite the asynchronous iteration~\eqref{Sec:ContractiveMapping Eq3} as 
\begin{align}
\bigl[x_{k + 1}\bigr]_i = T_i \left( \left[x_{s_{i 1, t_i(k)}}\right]_1, \ldots, \left[x_{s_{i m, t_i(k)}}\right]_m \right), \quad k\in\mathbb{N}_0.
\label{Appendix:ContractiveMapping Lemma1 Eq1}
\end{align}
Let $V_k = \|x_k - x^{\star}\|_{b,\infty}^w$. From the definition of $\|\cdot\|_{b,\infty}^w$, we have 
\begin{align*}
V_{k+1} &= \max_{i \in [m]} \; w_i\bigl\|[x_{k+1}]_i  -  [x^\star]_i \bigr\|_i\\
&\overset{\eqref{Appendix:ContractiveMapping Lemma1 Eq1}}{=} \max_{i\in[m]} \; w_i\left\|T_i \left( \left[x_{s_{i 1, t_i(k)}}\right]_1, \ldots, \left[x_{s_{i m, t_i(k)}}\right]_m \right) -   [x^\star]_i  \right\|_i\\
&\leq \max_{i\in[m]} \left\| T \left( \left[x_{s_{i 1, t_i(k)}}\right]_1, \ldots, \left[x_{s_{i m, t_i(k)}}\right]_m \right) - x^\star \right\|_{b,\infty}^w,
\end{align*}
where the inequality follows from the fact that $ w_i\bigl\|[x]_i\bigr\|_i \leq \bigl\|x\bigr\|_{b,\infty}^w$ for any $x\in \mathbb{R}^d$ and $i\in[m]$. Since $T$ is pseudo-contractive with respective to the block-maximum norm with contraction modulus $c$, we obtain
\begin{align*}
V_{k+1} &\leq c\max_{i\in[m]}  \left\| \left( \left[x_{s_{i 1, t_i(k)}}\right]_1, \ldots, \left[x_{s_{i m, t_i(k)}}\right]_m \right)  - x^\star \right\|_{b,\infty}^w\\
&= c\max_{i\in[m]}\max_{j\in[m]} \; w_j\left\| \left[x_{s_{i j,  t_i(k)}}\right]_j -  \left[x^\star\right]_j \right\|_j,
\end{align*}
where the equality follows from the definition of $\|\cdot\|_{b,\infty}^w$. Since $w_j\bigl\|[x]_j\bigr\|_j \leq \bigl\|x\bigr\|_{b,\infty}^w$ for any $x\in \mathbb{R}^d$ and $j\in[m]$, we have
\begin{align}
V_{k+1} &\leq c\max_{i\in[m]}\max_{j\in[m]}\left \|x_{s_{i j,  t_i(k)}} - x^\star\right\|_{b,\infty}^w \nonumber\\
& = c\max_{i\in[m]}\max_{j\in[m]} V_{s_{i j,  t_i(k)}}.
\label{Appendix:ContractiveMapping Lemma1 Eq2}
\end{align}
From~\eqref{Sec:ContractiveMapping Eq4}, $ k - \tau_k \leq s_{i j,  t_i(k)}$ for all $i,j\in[m]$ and $k\in\mathbb{N}_0$. Thus, $s_{i j,  t_i(k)} \in \{k - \tau_k,\ldots, k\}$, and hence
\begin{align*}
 V_{s_{i j,  t_i(k)}} \leq \max_{k - \tau_k \leq \ell \leq k} V_{\ell}.
\end{align*}
This together with~\eqref{Appendix:ContractiveMapping Lemma1 Eq2} implies that
\begin{align*}
V_{k+1} \leq c \max_{k - \tau_k \leq \ell \leq k} V_{\ell},
\end{align*}
which is the desired result.

\subsection{Proof of Theorem~\ref{Sec:ContractiveMapping Theorem1}}
\label{Sec:ContractiveMapping Theorem1 Proof}
According to Lemma~\ref{Sec:ContractiveMapping Lemma1}, the iterates generated by~\eqref{Sec:ContractiveMapping Eq3} satisfy
\begin{align*}
V_{k+1} &\leq q V_k + p \max_{(k - \tau_k)_+ \leq \ell \leq k} V_{\ell},\quad k\in\mathbb{N}_0,
\end{align*}
with $q=0$, $p=c$, and 
\begin{align*}
\tau_ k = k -  \min_{i\in[m]}\min_{j\in[m]} s_{i j, t_i(k)}.
\end{align*}
Since $t_i(k)\in\mathcal{K}_i$ for all $i\in[m]$ and $k\in\mathbb{N}_0$, it follows from Assumption~\ref{Sec:ContractiveMapping Assumption1}.1 that $\lim_{k \rightarrow \infty} t_i(k) = \infty$. Thus, by Assumption~\ref{Sec:ContractiveMapping Assumption1}.2,  $s_{i j, t_i(k)} \rightarrow \infty$ as $k \rightarrow \infty$ for all $i,j\in[m]$. This implies that 
\begin{align*}
\lim_{k\rightarrow \infty} k - \tau_k = \infty.
\end{align*}
Therefore, since $c < 1$, using Lemma~\ref{Sec:Main Results Lemma2} with $q = 0$ and $p = c$ completes the proof. 

%
\subsection{Proof of Theorem~\ref{Sec:ContractiveMapping Theorem2}}
\label{Sec:ContractiveMapping Theorem2 Proof}
From Assumption~\ref{Sec:ContractiveMapping Assumption2}.2, $k - D \leq s_{i j, k}$ for all $i,j\in[m]$ and $k\in\mathcal{K}_i$. Since $t_i(k)\in\mathcal{K}_i$ for $k\in\mathbb{N}_0$, we have 
\begin{align*}
t_i(k) - D \leq s_{i j, t_i(k)}.
\end{align*}
By Assumption~\ref{Sec:ContractiveMapping Assumption2}.1, $k - B \leq t_i(k)$ for each $i$ and all $k\in\mathbb{N}_0$. Thus, $k - B - D \leq s_{i j, t_i(k)}$, implying that
\begin{align*}
\tau_ k = k -  \min_{i\in[m]}\min_{j\in[m]} s_{i j, t_i(k)} \leq B + D.
\end{align*}
It follows from Lemma~\ref{Sec:ContractiveMapping Lemma1} that the iterates generated by~\eqref{Sec:ContractiveMapping Eq3} satisfy
\begin{align*}
V_{k+1} &\leq c \max_{(k - \tau_k)_+ \leq \ell \leq k} V_{\ell},\quad k\in\mathbb{N}_0,
\end{align*}
with $V_k = \|x_k - x^{\star}\|_{b,\infty}^w$ and $\tau_k \leq B + D$. As $c < 1$, using Lemma~\eqref{Sec:Main Results Lemma1} with $q=0$, $p=c$ and $\tau = B + D$ leads to
\begin{align*}
V_k \leq c^{\frac{k}{B + D + 1}} V_0 ,\quad k\in\mathbb{N}_0.
\end{align*}
The proof is complete.

%
\subsection{Proof of Theorem~\ref{Sec:ContractiveMapping Theorem3}}
\label{Sec:ContractiveMapping Theorem3 Proof}
According to Lemma~\ref{Sec:ContractiveMapping Lemma1}, the iterates generated by~\eqref{Sec:ContractiveMapping Eq3} satisfy~\eqref{Sec:Main Results Lemma3 Eq1} with $V_k = \|x_k - x^{\star}\|_{b,\infty}^w$, $q=0$, $p=c$, and $\tau_k$ defined in~\eqref{Sec:ContractiveMapping Eq4}. Since $q + p = c < 1$, it follows from Lemma~\ref{Sec:Main Results Lemma3} that $V_k \leq \Lambda(k) V_0$. This completes the proof.

\end{appendices}
\end{document}